\documentclass{amsart}

\usepackage[utf8]{inputenc}
\usepackage{amsmath}
\usepackage{amssymb}
\usepackage{amsthm}
\usepackage{enumitem}
\usepackage{mathrsfs}
\usepackage{graphicx}
\usepackage[hidelinks]{hyperref}
\usepackage[textsize=tiny]{todonotes}
\usepackage{xcolor}
\usepackage{bold-extra} 

\usepackage{tikz}
\usetikzlibrary{decorations.pathreplacing,matrix,arrows,positioning,automata,shapes,shadows,calc,fadings,decorations,through,intersections,decorations.text, decorations.pathmorphing}

\theoremstyle{plain}
\newtheorem{theorem}{Theorem}[section]
\newtheorem{lemma}[theorem]{Lemma}
\newtheorem{corollary}[theorem]{Corollary}
\newtheorem{proposition}[theorem]{Proposition}

\theoremstyle{definition}
\newtheorem{definition}[theorem]{Definition}

\newtheorem{example}[theorem]{Example}

\newtheorem{remark}[theorem]{Remark}

\newtheorem{question}[theorem]{Question}

\theoremstyle{remark}

\newcommand{\bR}{\mathbb{R}}
\newcommand{\R}{\bR}
\newcommand{\bQ}{\mathbb{Q}}
\newcommand{\Q}{\bQ}

\newcommand{\Z}{\mathcal{Z}}
\newcommand{\cA}{\mathcal{A}}
\newcommand{\cB}{\mathcal{B}}
\newcommand{\cC}{\mathcal{C}}

\newcommand{\cI}{\mathcal{I}}
\newcommand{\I}{\cI}
\newcommand{\cJ}{\mathcal{J}}
\newcommand{\J}{\cJ}

\newcommand{\cM}{\mathcal{M}}

\newcommand{\cP}{\mathcal{P}}

\newcommand{\continuum}{\mathfrak{c}}

\newcommand{\fin}{\mathrm{Fin}}
\newcommand{\Fin}{\fin}
\newcommand{\conv}{\mathrm{conv}}

\begin{document}


\title{Borel complexity of sets of ideal limit points}


\author[R.~Filip\'{o}w]{Rafa\l{} Filip\'{o}w}
\address[Rafa\l{}~Filip\'{o}w]{Institute of Mathematics\\ Faculty of Mathematics, Physics and Informatics\\ University of Gda\'{n}sk\\ ul.~Wita Stwosza 57\\ 80-308 Gda\'{n}sk\\ Poland}
\email{Rafal.Filipow@ug.edu.pl}
\urladdr{\url{http://mat.ug.edu.pl/~rfilipow}}

\author[A.~Kwela]{Adam Kwela}
\address[Adam Kwela]{Institute of Mathematics\\ Faculty of Mathematics\\ Physics and Informatics\\ University of Gda\'{n}sk\\ ul.~Wita  Stwosza 57\\ 80-308 Gda\'{n}sk\\ Poland}
\email{Adam.Kwela@ug.edu.pl}
\urladdr{\url{https://mat.ug.edu.pl/~akwela}}

\author[P.~Leonetti]{Paolo Leonetti}
\address[Paolo Leonetti]{Department of Economics\\ Universit\'{a} degli Studi dell’Insubria\\ via Monte Generoso 71 \\ Varese 21100\\ Italy}
\email{Leonetti.Paolo@gmail.com}
\urladdr{\url{https://sites.google.com/site/leonettipaolo}}

\thanks{The second-listed author was supported by the Polish National Science Centre project OPUS No. 2024/53/B/ST1/02494.}


\date{\today}








\begin{abstract}
Let $X$ be an uncountable Polish space and let $\mathcal{I}$ be an ideal on $\omega$. A point $\eta \in X$ is an $\mathcal{I}$-limit point of a sequence $(x_n)$ taking values in $X$ if there exists a subsequence $(x_{k_n})$ convergent to $\eta$ such that the set of indexes $\{k_n: n \in \omega\}\notin \mathcal{I}$. Denote by $\mathscr{L}(\mathcal{I})$ the family of subsets $S\subseteq X$ such that $S$ is the set of $\mathcal{I}$-limit points of some sequence taking values in $X$ or $S$ is empty. 
In this paper, we study 
the relationships between the topological complexity of ideals $\mathcal{I}$, their combinatorial properties, and the families of sets $\mathscr{L}(\mathcal{I})$ which can be attained. 

On the positive side, we provide several purely combinatorial (not depending on the space $X$) characterizations of ideals $\mathcal{I}$ for the inclusions and the equalities between $\mathscr{L}(\mathcal{I})$ and the Borel classes $\Pi^0_1$, $\Sigma^0_2$, and $\Pi^0_3$. 
As a consequence, we prove that if $\mathcal{I}$ is a $\Pi^0_4$ ideal then exactly one of the following cases holds: $\mathscr{L}(\mathcal{I})=\Pi^0_1$ or $\mathscr{L}(\mathcal{I})=\Sigma^0_2$ or $\mathscr{L}(\mathcal{I})=\Sigma^1_1$ (however we do not have an example of a $\Pi^0_4$ ideal with $\mathscr{L}(\mathcal{I})=\Sigma^1_1$). In addition, we provide an explicit example of a coanalytic ideal $\mathcal{I}$ for which $\mathscr{L}(\mathcal{I})=\Sigma^1_1$.

On the negative side, since $\mathscr{L}(\mathcal{I})$ contains all singletons, it is immediate that there are no ideals $\mathcal{I}$ such that $\mathscr{L}(\mathcal{I})=\Sigma^0_1$. On the same direction, we show that there are no ideals $\mathcal{I}$ such that $\mathscr{L}(\mathcal{I})=\Pi^0_2$ or $\mathscr{L}(\mathcal{I})=\Sigma^0_3$. In fact, for instance, if $\mathcal{I}$ is a Borel ideal and $\mathscr{L}(\mathcal{I})$ contains a non $\Sigma^0_2$ set, then it contains all $\Pi^0_3$ sets. We conclude with several open questions. 
\end{abstract}


\maketitle




\section{Introduction}\label{sec:intro}

In this paper, we study systematically the family of subsets of a given topological space $X$ which can be realized as sets of $\mathcal{I}$-limit points of some sequences taking values in $X$, where $\mathcal{I}$ is an ideal on the nonnegative integers $\omega$. 
To be more precise, let us recall some definitions and known results from the literature. 


\subsection{Ideals} 

A family $\mathcal{I}\subseteq \mathcal{P}(\omega)$ is an \emph{ideal} if it is closed under subsets and finite unions. Unless otherwise stated, it is also assumed that $\mathcal{I}$ is admissible, that is, it contains $\mathrm{Fin}=[\omega]^{<\omega}$ while $\omega\notin \I$. 
Let $\I^+=\mathcal{P}(\omega)\setminus \I$ be the family of $\I$-positive sets and denote by $\mathcal{I}^\star=\{S\subseteq \omega: \omega\setminus S \in \I\}$ the dual filter of $\I$. 
An ideal $\mathcal{I}$ is a \emph{P-ideal} if for every sequence $(A_n)$ with values in $\mathcal{I}$ there exists $A \in \mathcal{I}$ such that $A_n\setminus A$ is finite for all $n \in \omega$, see e.g. \cite{MR1711328}. 
Informally, ideals represent the family of \textquotedblleft small subsets\textquotedblright\, of $\omega$. 
We will also consider ideals on countably infinite sets, which are defined analogously. 
Identifying each subsets of $\omega$ with its characteristic function, ideals can be regarded as subsets of the Cantor space $2^{\omega}$, hence we can speak about their topological complexity. For instance, 
$\mathrm{Fin}$ and the summable ideal 
$$
\mathcal{I}_{1/n}=\left\{S\subseteq \omega: \sum_{n\in S} \frac{1}{n+1}<\infty\right\}
$$
are $F_\sigma$ $P$-ideals, while the family of asymptotic density zero sets 
$$
\mathcal{Z}=\left\{S\subseteq \omega: \lim_{n\to \infty}\frac{|S\cap [0,n)|}{n}=0\right\}
$$
is an $F_{\sigma\delta}$ $P$-ideal which is not $F_\sigma$. 
An another example, the family of Banach density zero sets
\begin{equation}\label{eq:Bideal}
\mathcal{B}=\left\{S\subseteq \omega: \lim_{n\to \infty}\max_{k \in \omega}\,\frac{|S\cap [k,k+n)|}{n}=0\right\}
\end{equation}
is an $F_{\sigma\delta}$-ideal which is not a $P$-ideal. 
Finally, maximal ideals, namely, the complements of free ultrafilters on $\omega$, are not analytic, hence not Borel. 


\subsection{\texorpdfstring{$\I$}{I}-limit points} 

Given a sequence $  x =(x_n: n \in \omega)$ taking values in a topological space $X$ and an ideal $\I$ on $\omega$, we say that $\eta\in X$ is an $\mathcal{I}$\emph{-limit point of} $  x $ if there exists $S \in \I^+$ such that the subsequence $  x \upharpoonright S=(x_n: n \in S)$ is convergent to $\eta$. Accordingly, define 
$$
\Lambda_{  x }(\I)=\left\{\eta \in X: \eta \text{ is an }\I\text{-limit point of }  x \right\}.
$$
In the literature, $\mathcal{Z}$-limit points of real sequences have been usually called \textquotedblleft statistical limit points,\textquotedblright\, see e.g. \cite{MR1372186, MR1181163, MR1838788, MR1260176}. 
The topological nature of the sets of $\mathcal{I}$-limits points $\Lambda_{  x }(\I)$ has been studied in \cite{MSP24, MR3883171,  MR2923430, MR2463821, MR4393937, MR4505549}. 
Here, we continue along this line of research. Following \cite{MSP24}, we introduce our main definition:

\begin{definition}\label{defi:familyIlimitpoints}
Let $X$ be a topological space and $\mathcal{I}$ be an ideal on $\omega$. Denote by $\mathscr{L}_X(\mathcal{I})$ the family of sets of $\mathcal{I}$-limit points of sequences ${  x }$ taking values in $X$ together with the empty set, that is, 
\begin{equation}\label{eq:LXI}
\mathscr{L}_X(\mathcal{I})=\left\{A\subseteq X: A=\Lambda_{  x }(\mathcal{I}) \text{ for some sequence }{  x } \in X^\omega\right\}\cup \{\emptyset\}.
\end{equation}
If the topological space $X$ is understood, we write simply $\mathscr{L}(\mathcal{I})$. 
\end{definition}

As remarked in \cite{MSP24}, the addition of $\{\emptyset\}$ in \eqref{eq:LXI} avoids the repetition of known results in the literature and to add further subcases based on the topological structure of the underlying space $X$. As we will see in our main results, if one avoids the empty set case, then it is possible to characterize $\mathscr{L}(\mathcal{I})$ for a large class of topological spaces. In any case, we will study in Section \ref{sec:emptysetcase} conditions on $X$ and $\I$ which ensure the existence of a sequence $  x $ such that $\Lambda_{  x }(\I)=\emptyset$. 


\subsection{\texorpdfstring{$\I$}{I}-cluster points} 

Along a similar direction, given a sequence $  x $ taking values in a topological space $X$ and an ideal $\I$ on $\omega$, we say that $\eta\in X$ is an $\mathcal{I}$\emph{-cluster point of} $  x $ if $\{n \in \omega: x_n \in U\}\notin \I$ for all open neighborhoods $U$ of $\eta$. Then, we write
$$
\Gamma_{  x }(\I)=\left\{\eta \in X: \eta \text{ is an }\I\text{-cluster point of }  x \right\}.
$$
We are going to use throughout the well-known fact that every $\I$-limit point is necessarily an $\I$-cluster point, namely, $\Lambda_{  x }(\I)\subseteq \Gamma_{  x }(\I)$ for every sequence $  x $; moreover, if $X$ is first countable, then $\Lambda_{  x }(\mathrm{Fin})=\Gamma_{  x }(\mathrm{Fin})$, see e.g. \cite[Lemma 3.1]{MR3920799}. 
On the other hand, it is important to note that, even in the case of real bounded sequences, $\mathcal{I}$-cluster points and $\mathcal{I}$-limit points may behave really differently. For instance, it is known 
(\cite[Example 4]{MR1181163})
that if ${  x }$ is an equidistributed sequence with values in $[0,1]$, then 
$$
\Lambda_{  x }(\mathcal{Z})=\emptyset 
\quad \text{ and }\quad 
\Gamma_{  x }(\mathcal{Z})=[0,1].
$$

In analogy with Definition \ref{defi:familyIlimitpoints}, we introduce also the following: 

\begin{definition}\label{defi:familyIclusterpoints}
Let $X$ be a topological space and $\mathcal{I}$ be an ideal on $\omega$. Denote by $\mathscr{C}_X(\mathcal{I})$ the family of sets of $\mathcal{I}$-cluster points of sequences ${  x }$ taking values in $X$ together with the empty set, that is, 
\begin{equation}\label{eq:CXI}
\mathscr{C}_X(\mathcal{I})=\left\{B\subseteq X: B=\Gamma_{  x }(\mathcal{I}) \text{ for some sequence }{  x } \in X^\omega\right\}\cup \{\emptyset\}.
\end{equation}
If the topological space $X$ is understood, we write simply $\mathscr{C}(\mathcal{I})$. 
\end{definition}

It is known that $\mathcal{I}$-cluster points behave better than $\mathcal{I}$-limit points, see e.g. Lemma \ref{lem:clusterclosed}, letting their study be substantially easier. Several characterizations of sets of $\mathcal{I}$-cluster points may be found in \cite{MR3920799}. Regarding the addition of the empty set in \eqref{eq:CXI}, it turns out that there exists a sequence $  x $ such that $\Gamma_{  x }(\I)=\emptyset$ if and only if $X$ is not compact (the \textsc{if} part being a folklore result), see Proposition \ref{prop:emptysetcluster} below.


\subsection{Notations} 

Given a set $X$, denote by $[X]^{\kappa}$ the family of all subsets of $X$ which have the cardinality $\kappa$. Moreover, we write 
$
[X]^{<\kappa}=\bigcup \{[X]^\mu : \mu<\kappa\}$ and $[X]^{\le \kappa}=\bigcup \{[X]^\mu: \mu\le \kappa\}$.
In particular, $[X]^{\leq \omega}$ stands for the family of all countable (finite or infinite) subsets of $X$. 
We write $s^\frown t$ to denote the concatenation of two sequences $s\in X^{<\omega}$ and $t\in X^{<\omega}\cup X^\omega$ (and $s^\frown a$ in place of $s^\frown (a)$ for $a\in X$).
We also use the notation 
$
(\ldots,b,a^k,c,\ldots) = (\ldots,b,a,\dots,a,c,\ldots)
$ 
where $a$ appears $k$ times 
and $a^\infty = (a,a,\ldots)$, i.e., $a^\infty$ denotes the infinite constant sequence with the value $a$. Also, $a^0$ is the empty sequence.  
We denote by $\Q(2^\omega)$ the set of all $x\in 2^\omega$ such that $x_n=0$ for all but finitely many $n$. 

Given a topological space $X$ and an ordinal $1\le \alpha <\omega_1$, we use the standard Borel pointclasses notations $\Sigma^0_\alpha(X)$ and $\Pi^0_\alpha(X)$, so that $\Sigma^0_1(X)$ stands for the open sets of $X$, $\Pi^0_1(X)$ for the closed sets, $\Sigma_2^0(X)$ for the $F_\sigma$-sets, $\Pi^0_3(X)$ for the $F_{\sigma\delta}$-sets etc.; 
we denote by $\Delta^0_\alpha(X)=\Sigma^0_\alpha(X) \cap \Pi^0_\alpha(X)$ the ambiguous classes; 
also, if $X$ is a Polish space, $\Sigma^1_1(X)$ stands for the analytic sets, $\Pi^1_1(X)$ for the coanalytic sets, etc.; see e.g. \cite[Section 11.B]{MR1321597}. 
Again, we suppress the reference to the underlying space $X$ if it is clear from the context. 

For a topological space $X$, we consider  $X^\omega$ as a topological space endowed with the product topology. For instance, taking $X=2=\{0,1\}$ or $X=\omega$ with the discrete topology, we obtain the \emph{Cantor space} $2^\omega = \{0,1\}^\omega$ or the \emph{Baire space} $\omega^\omega$. (As usual, a positive integer $n$ is identified with $\{0,1,\ldots,n-1\}$.) For a finite sequence $s\in 2^{<\omega}$ [$s\in\omega^{<\omega}$, resp.] we denote by $|s|$ its length and by $[s]$ the set of all $x\in 2^\omega$ [$x\in\omega^\omega$, resp.] such that $x\restriction |s|=s$: hence, if $s=(s_0,\ldots,s_k)$ is a finite sequence then $s \in 2^{k+1}$, $|s|=k+1$, and $[s]$ is the set of all $x \in 2^\omega$ such that $x\restriction |s|=(x_0,\ldots,x_k)=s$; in addition, $x\restriction 0=\emptyset$. 
Then sets of the form $[s]$ for $s\in 2^{<\omega}$ [$s\in\omega^{<\omega}$, resp.] form a basis of the topology of $2^\omega$ [$\omega^\omega$, resp.]. We will use facts that the Cantor space is \emph{homogeneous}, i.e.,~for every points $x,y\in 2^\omega$ there is a homeomorphism $f:2^\omega\to 2^\omega$ such that $f(x)=y$, and \emph{countable dense homogeneous}, i.e.,~for every countable dense sets $A,B\subseteq 2^\omega$ there is a homeomorphism $f:2^\omega\to 2^\omega$ such that $f[A]=B$ (see e.g.~\cite[Problem~30A(2)]{MR0264581} and \cite[Exercise~4.3.H(e)]{MR1039321}), respectively). 

A topological space $X$, or any of its subsets, is \emph{discrete} if it contains only isolated points. Given a set $A\subseteq X$, we denote by $A^{\,|}$ the set of its accumulation points, that is, the set of all $\eta \in X$ such that every neighborhood $U$ of $\eta$ contains a point of $A$ which is different from $\eta$. 
The family of isolated points of $A$ is denoted by $\mathrm{iso}(A)$.

We denote by $B(\eta,r)$ and $\overline{B}(\eta,r)$ the open and closed balls of radius $r>0$ centered at a point $\eta$ in a metric space $X$, respectively. 
Lastly, for a set $A\subseteq X$ and $r>0$, we write $B(A,r)$ for  the open ball of radius $r$ around $A$, i.e., $B(A,r) = \bigcup\{B(\eta ,r):\eta \in A\}$. 


\subsection{Literature results} 

We list below several known results about the families $\mathscr{L}(\I)$ and $\mathscr{C}(\I)$ introduced in Definitions \ref{defi:familyIlimitpoints} and \ref{defi:familyIclusterpoints}. 

\begin{lemma}\label{lem:clusterclosed}
Let $X$ be a topological space and $\mathcal{I}$ be an ideal on $\omega$. Then $\mathscr{C}(\I)\subseteq \Pi^0_1$. 
\end{lemma}
\begin{proof}
See \cite[Lemma 3.1(iv)]{MR3920799}; cf. also \cite[Theorems~4.1(i)]{MR1844385} for metric spaces.
\end{proof}
Conversely, a characterization of the equality $\mathscr{C}(\I)=\Pi^0_1$ will be given in Proposition \ref{prop:Gamma-vs-closed}. It is worth noting that it is possible that the family $\mathscr{C}(\I)$ is a proper subset of $\Pi^0_1$, even if $X$ is a compact Polish space: indeed, this is always the case, e.g., if $X$ is Hausdorff with $|X|\ge 2$ and $\I$ is maximal, see Proposition \ref{prop:Imaximal} below.

Roughly, Lemma \ref{lem:clusterclosed} explains why the families $\mathscr{C}(\I)$ are easier than the families $\mathscr{L}(\I)$, hence we will focus mainly on the latter ones. 

At this point, we recall the following results: 
\begin{theorem}\label{thm:oldthmballeo}
Let $X$ be a 
first countable Hausdorff space 
where all closed sets are separable, and let $\I$ be an ideal on $\omega$. Then the following hold: 
\begin{enumerate}[label={\rm (\roman{*})}]
\item \label{item:BL1} If $\mathcal{I}$ is an $F_\sigma$-ideal then 
$\Lambda_{  x }(\I)=\Gamma_{  x }(\I)$ for every sequence $  x $;
\item \label{item:BL2} If $\mathcal{I}$ is an $F_\sigma$-ideal then 
$\mathscr{L}(\mathcal{I})=\mathscr{C}(\I)=\Pi_1^0$;
\item \label{item:BL3} If $\mathcal{I}$ is an analytic P-ideal and $X$ is nondiscrete then $\Lambda_{  x }(\I)=\Gamma_{  x }(\I)$ for every sequence $  x $ if and only if $\I$ is an $F_\sigma$-ideal;
\item \label{item:BL4} If $\mathcal{I}$ is an analytic P-ideal and $X$ is locally compact then each isolated point in $\Gamma_{  x }(\I)$ belongs to $\Lambda_{  x }(\I)$ for every sequence $  x $; 
\item \label{item:BL5} If $\mathcal{I}$ is an analytic P-ideal which is not $F_\sigma$ then $\mathscr{L}(\mathcal{I})=\Sigma_2^0$.
\end{enumerate}
\end{theorem}
\begin{proof}
It follows by \cite[Theorems 2.2, 
2.3, 
2,5, 
2.7, 
2.8, and 3.4]{MR3883171}.
\end{proof}

It is worth to recall that there are no $G_\delta$-ideals \cite[Proposition 1.2.1]{alcantara-phd-thesis}, and that analytic $P$-ideals are necessarily $F_{\sigma\delta}$, see e.g. \cite[Lemma 1.2.2 and Theorem 1.2.5]{MR1711328}. Hence Theorem \ref{thm:oldthmballeo} studies the relationship between $\I$-cluster points and $\I$-limit points for ideals $\I$ with low topological complexity. 

In the special case of the ideal $\Z$, 
finer connections are known: if $A,B$ are nonempty subsets of a Polish space $X$ such that $A\subseteq B$, $A$ is an $F_\sigma$-set, $B$ is closed, and $A$ contains the isolated points of $B$, then there exists a sequence $  x $ such that $\Lambda_{  x }(\Z)=A$ and $\Gamma_{  x }(\Z)=B$, see \cite[Corollary 3.3]{MR3883171}. In particular, in Polish spaces, we have
$$
\mathscr{L}(\Z)=\Sigma^0_2 
\quad \text{ and }\quad 
\mathscr{C}(\Z)=\Pi^0_1.
$$
Note that the first equality can also be obtained by Theorem \ref{thm:oldthmballeo}\ref{item:BL5} since $\Z$ is an analytic $P$-ideal which is not $F_\sigma$.

To state the next result, we recall some further definitions. 
An ideal $\mathcal{I}$ has the \emph{hereditary Baire property} if the restriction $\mathcal{I}\upharpoonright A=\{S\cap A: S \in \mathcal{I}\}$ has the Baire property for every $A \in \mathcal{I}^+$, see Section \ref{subsec:baire} below for details. 
Also, an ideal $\mathcal{I}$ on $\omega$ is said to be a $P^+$\emph{-ideal} if, for every decreasing sequence $(A_n)$ with values in $\mathcal{I}^+$, there exists $A \in \mathcal{I}^+$ such that $A\setminus A_n$ is finite for all $n \in \omega$. 
This notion will be studied, together with two weaker variants, in Section \ref{sec:Plike}. 
It is known that all $F_\sigma$-ideals are $P^+$-ideals, see 
\cite{MR748847}. 
On the other hand, as remarked in \cite[p. 2031]{MR3692233}, $P^+$-ideals may have arbitrarily high Borel complexity; cf. Theorem \ref{thm:highdimensionalhighBorel} and its proof below. 

Moreover, an ideal $\mathcal{I}$ is called a \emph{Farah ideal} if there exists a sequence $(K_n)$ of hereditary (i.e., closed under taking subsets of its elements) compact subsets in $\mathcal{P}(\omega)$ such that $S \in \mathcal{I}$ if and only if for all $n \in \omega$ there exists $k \in \omega$ such that $S\setminus k \in K_n$. It is known that all analytic P-ideals are Farah and that all Farah ideals are $F_{\sigma\delta}$, see \cite[pp. 199--201]{MR2849045}. In addition, the ideal $\mathcal{B}$ defined in \eqref{eq:Bideal} and 
$$
\mathrm{nwd}=\left\{S\subseteq \mathbb{Q}: S \text{ is nowhere dense in }\mathbb{R}\right\}
$$
are Farah ideals which are not analytic $P$-ideals. 

\begin{theorem}\label{thm:xi}
Let $X$ be a first countable Hausdorff space, and let $\I$ be an ideal on $\omega$. Then the following hold:
\begin{enumerate}[label={\rm (\roman{*})}]
\item \label{item:1xi} If $\mathcal{I}$ is $P^+$-ideal then $\mathscr{L}(\mathcal{I}) \subseteq \Pi^0_1$;
\item \label{item:2xi} If $\mathcal{I}$ has the hereditary Baire property and $X$ is nondiscrete metrizable then $\mathcal{I}$ is a $P^+$-ideal if and only if $\mathscr{L}(\mathcal{I}) \subseteq \Pi^0_1$;
\item \label{item:3xi} If $\mathcal{I}$ is a Farah ideal then $\mathscr{L}(\mathcal{I}) \subseteq \Sigma^0_2$.
\end{enumerate}
If, in addition, $X$ is second countable, then:
\begin{enumerate}[label={\rm (\roman{*})}]
\setcounter{enumi}{3}
\item \label{item:4xi} If $\mathcal{I}$ has the hereditary Baire property then $\Pi^0_1 \subseteq \mathscr{L}(\mathcal{I})$;
\item \label{item:5xi} If $\mathcal{I}$ has the hereditary Baire property and is not a $P^+$-ideal then $\Sigma^0_2 \subseteq \mathscr{L}(\mathcal{I})$.
\end{enumerate}
\end{theorem}
\begin{proof}
It follows by \cite[Corollary 3.9, and Theorem 4.5]{MR4393937} and \cite[Corollary 2.5 and Theorems 2.4, 2.8, and 2.10]{MR4505549}. 
\end{proof}

Along the same lines, we have also the next results. 
\begin{theorem}\label{thm:MSP24}
Let $X$ be a first countable Hausdorff space and let $\I$ be an ideal on $\omega$. Then the following hold:
\begin{enumerate}[label={\rm (\roman{*})}]
\item \label{item:1MSP} If $\I$ is a $G_{\delta\sigma}$-ideal then $\mathscr{L}(\I)=\Pi^0_1$;
\item \label{item:2MSP} If $\I$ has the hereditary Baire property then $\mathscr{L}(\I)\subseteq \Pi^0_1$ or $\Sigma^0_2\subseteq \mathscr{L}(\I)$;
\item \label{item:3MSP} If $\I$ is coanalytic and $X$ is a Polish space then $\mathscr{L}(\I)\subseteq \Sigma^1_1$.
\end{enumerate}
\end{theorem}
\begin{proof}
It follows by \cite[Theorem 2.6, Corollary 2.9, and Proposition 4.1]{MSP24}: cf. also \cite[Proposition 4.1]{MR3883171} and \cite[Theorem 4.1]{MR4505549} for item \ref{item:3MSP}.
\end{proof}

We remark that item \ref{item:3MSP} holds also in first countable spaces, for a suitable notion of analyticity in topological spaces, see \cite[Definition 3.2]{MSP24}. 

Lastly, given (possibly nonadmissible) ideals $\mathcal{I}$ and $\mathcal{J}$ on two countably infinite sets $Z$ and $W$, respectively, we define their \emph{Fubini product} by 
\begin{equation}\label{eq:Fubiniproduct}
\mathcal{I}\otimes \mathcal{J}=\{S\subseteq Z\times W\colon \{n \in Z: \{k\in W: (n,k)\in S\}\notin \mathcal{J}\}\in \mathcal{I}\}\},
\end{equation}
which is an ideal on the countably infinite set $Z\times W$, see e.g. \cite[Chapter 1]{MR1711328}. Hence, recursively, $\mathrm{Fin}^{\alpha}=\mathrm{Fin}\otimes \mathrm{Fin}^{\alpha-1}$ for all integers $\alpha\ge 2$ is an ideal on $\omega^\alpha$. With an usual abuse of notation, we write $\emptyset \otimes \I$ in place of $\{\emptyset\}\otimes \I$. 

The next result proves that the families $\mathscr{L}(\I)$ may attain large Borel complexities. 
\begin{theorem}\label{thm:highBorelMSP}
Let $X$ be a Polish space and let $\alpha \ge 1$ be an integer. Then the following hold: 
\begin{enumerate}[label={\rm (\roman{*})}]
\item \label{item:1mainlimitfamilies} $\mathrm{Fin}^\alpha$ is a $\Sigma^0_{2\alpha}$-ideal and $\mathscr{L}(\mathrm{Fin}^\alpha)= \Pi^0_{2\alpha-1}$;
\item \label{item:2mainlimitfamilies} $\emptyset \otimes \mathrm{Fin}^\alpha$ is a $\Pi^0_{2\alpha+1}$-ideal and $\mathscr{L}(\emptyset \otimes \mathrm{Fin}^\alpha)= \Sigma^0_{2\alpha}$.
\end{enumerate}
\end{theorem}
\begin{proof}
    See \cite[Theorem 2.5]{MSP24}. 
\end{proof}

In this work, we continue on the line of research of the above results. 
We provide below an outline of the structure of the manuscript.


\subsection{Plan of the paper} 

We start in Section \ref{sec:preliminaries} with the study of some basic properties of ideals as [nowhere] maximality and characterizations of [hereditary] Baire property, together with the operations of Fubini sums and products. 
In addition, we collect some basic results on preorders of ideals, $\I$-cluster points, and $\I$-limit points. 

In Section \ref{sec:Plike}, we investigate two weaker variants of the notion of $P^+$-ideal, and show that they are the combinatorial properties which characterize several connections between sets of $\I$-cluster points and $\I$-limit points: for instance, the case where $\Gamma_{  x }(\I)=\Lambda_{  x }(\I)$ for every sequence $  x $ (which leads to the inclusion $\mathscr{L}(\I)\subseteq \Pi^0_1$ by Lemma \ref{lem:clusterclosed}), or where every isolated point in $\Gamma_{  x }(\I)$ belongs to $\Lambda_{  x }(\I)$ for every $  x $. 

In Section \ref{sec:Ischemes}, we introduce and study several classes of ideals related to certain $\I$-schemes (which are Cantor trees of $\I$-positive sets), see Definition \ref{def:propertiesIschemesclasses}. 

With the above premises, the next Sections have simple objectives:
\begin{enumerate}[label={\rm (\roman{*})}]
\item In Section \ref{sec:emptysetcase}, we study the existence of a sequence $  x $ which has no $\I$-cluster points or no $\I$-limit points;  
\item In Section \ref{sec:closedsets}, we characterize the connections between the families $\mathscr{C}(\I)$, $\mathscr{L}(\I)$, and the class of closed subsets of $X$; 

\item In Section \ref{sec:Fsigmasets}, Section \ref{sec:Fsigmadeltasets}, and Section \ref{sec:analyticsets}, we study the inclusion and the equality between $\mathscr{L}(\I)$ and the classes $\Sigma^0_2$, $\Pi^0_3$, and $\Sigma^1_1$, respectively (note that, e.g., if there exists a $\Sigma^0_2$ set which is not closed, then equality $\mathscr{C}(\I)=\Sigma^0_2$ fails by Lemma \ref{lem:clusterclosed}; in particular, the analogues for $\mathscr{C}(\I)$ are void).
\end{enumerate}
We conclude in Section \ref{sec:applications} with several applications and open questions:
\begin{enumerate}[label={\rm (\roman{*})}]
\item We show that there are no ideals $\I$ such that $\mathscr{L}(\I)=\Pi^0_2$ or $\mathscr{L}(\I)=\Sigma^0_3$. 
\item We study the equalities $\mathscr{L}(\I)=\Pi^0_1$ and $\mathscr{L}(\I)=\Sigma^0_2$.
\item We prove that if $\mathcal{I}$ is a $\Pi^0_4$ ideal then exactly one of the following cases holds: $\mathscr{L}(\mathcal{I})=\Pi^0_1$ or $\mathscr{L}(\mathcal{I})=\Sigma^0_2$ or $\mathscr{L}(\mathcal{I})=\Sigma^1_1$. However we do not have an example of a $\Pi^0_4$ ideal with $\mathscr{L}(\mathcal{I})=\Sigma^1_1$ -- we leave it as our main open question.
\item We give an explicit example of a coanalytic ideal $\I$ such that $\mathscr{L}(\I)=\Sigma^1_1$.
\end{enumerate}


\section{Preliminaries}\label{sec:preliminaries}

Let us start with a simple lemma: 
\begin{lemma}\label{lem:Xlen}
Pick an integer $n\ge 1$. Let $X$ be a Hausdorff space and let $\I$ be an ideal on $\omega$ such that $|X|\ge n$ and there exists a partition $\{S_0,\ldots,S_{n-1}\}$ of $\omega$ into $\I$-positive sets. Then 
$$
[X]^{\le n}\subseteq \mathscr{L}(\I)\cap \mathscr{C}(\I).
$$
\end{lemma}
\begin{proof}
Pick distinct point $\eta_0,\ldots,\eta_{n-1} \in X$ and define the sequence $  x =(x_i: i \in \omega)$ such that $x_i=\eta_j$ for all $i \in S_j$ and $j<n$. Then $\Lambda_{  x }(\I)=\Gamma_{  x }(\I)=\{\eta_0,\ldots,\eta_{n-1}\}$. 
\end{proof}


\subsection{Maximality}\label{subsec:maximal} 

An ideal $\I$ on $\omega$ is \emph{maximal} if there is no ideal $\mathcal{J}$ on $\omega$ such that $\I\subseteq \J$ and $\I\neq \J$. In other words, $\I$ is maximal if its dual filter $\I^\star$ is a free ultrafilter on $\omega$. It is useful to recall that $\I$ is maximal if and only if either $S$ or $S^c=\omega\setminus S$ belongs to $\I$ for each $S\subseteq \omega$, see e.g. \cite[Lemma 7.4]{MR1940513}.  

\begin{proposition}\label{prop:Imaximal}
Let $X$ be a Hausdorff space with $|X|\ge 2$ and let $\I$ be an ideal on $\omega$. Then the following are equivalent: 
\begin{enumerate}[label={\rm (\roman{*})}]
\item \label{item:max1} $\I$ is a maximal ideal;
\item \label{item:max2} $\mathscr{L}(\I)=\mathscr{C}(\I)=
\{\{\eta\}: \eta \in X\}$;
\item \label{item:max3} $\mathscr{C}(\I)=\{\{\eta\}: \eta \in X\}$;
\item \label{item:max4} $\mathscr{L}(\I)=\{\{\eta\}: \eta \in X\}$.
\end{enumerate}
\end{proposition}
\begin{proof}
\ref{item:max1} $\implies$ \ref{item:max2}. 
It follows by Lemma \ref{lem:Xlen} that $\{\{\eta\}: \eta \in X\}$ 
is contained both in $\mathscr{C}(\I)$ and $\mathscr{L}(\I)$. 
Conversely, if there exists $B \in \mathscr{C}(\I)$ with $|B|\ge 2$, then there are distinct $\eta_1,\eta_2 \in B$ and a sequence $  x $ such that $\Gamma_{  x }(\I)=B$. Pick disjoint neighborhoods $U_1,U_2$ of $\eta_1,\eta_2$, respectively. Hence $\{n \in \omega: x_n \in U_1\}$ and $\{n \in \omega: x_n \in U_2\}$ are two disjoint $\mathcal{I}$-positive sets, which contradicts the maximality of $\I$ by \cite[Lemma 7.4]{MR1940513}. This proves that $\mathscr{C}(\I)\subseteq \{\{\eta\}: \eta \in X\}$. The inclusion $\mathscr{L}(\I)\subseteq \{\{\eta\}: \eta \in X\}$ goes similarly, see e.g. \cite[Proposition 1.4]{MR4393937}. Therefore $\mathscr{L}(\I)=\mathscr{C}(\I)=\{\{\eta\}: \eta \in X\}$.

\ref{item:max2} $\implies$ \ref{item:max3}. It is obvious.

\ref{item:max3} $\implies$ \ref{item:max4}. By Lemma \ref{lem:Xlen}, $\mathscr{L}(\I)$ contains all singletons. Suppose that $\mathscr{L}(\I)$ contains a set $\Lambda_{  x }(\I) \subseteq X$ with $|\Lambda_{  x }(\I)|\ge 2$. Since $\Lambda_{  x }(\I)\subseteq \Gamma_{  x }(\I)\in \mathscr{C}(\I)$, hence $\mathscr{C}(\I)$ would contain a set with cardinality $\ge 2$.

\ref{item:max4} $\implies$ \ref{item:max1}. 
Suppose that $\I$ is not maximal. Then there exists a partition $\{S_0,S_1\}$ of $\omega$ into $\I$-positive sets. It follows by Lemma \ref{lem:Xlen} that $[X]^{\le 2}\subseteq \mathscr{L}(\I)$. 
\end{proof}

An ideal $\I$ on $\omega$ is \emph{nowhere maximal} if the restriction $\I\restriction S$ is not maximal for every $S\in \I^+$. Note that, in \cite[Definition~3.12]{MR4358610}, the authors use the term \emph{extremely not min-representable} for dual filters of nowhere maximal ideals. The proof of the next result is straightforward (hence, details are omitted). 

\begin{proposition}
\label{prop:nowhere-maximal}
Let $\I$ be an ideal on $\omega$. Then the following are equivalent:
\begin{enumerate}[label={\rm (\roman{*})}]
    \item $\I$ is nowhere maximal;
    \item For every $S\in \I^+$ there exists an infinite partition of $S$ into $\I$-positive sets.
\end{enumerate}
\end{proposition}

Related characterizations 
can be found in Proposition~\ref{prop:Ischeme-characterizations}.


\subsection{Baire property}\label{subsec:baire} 

Let us recall that a subset $S$ of a topological space $X$ has the \emph{Baire property} if there exists an open set $U\subseteq X$ such that the symmetric difference $S\triangle U$ is meager, that is, it is contained in a countable union of closed sets with empty interior. 
The following result characterizes ideals on $\omega$ with the Baire property. 

For, let $\mathcal{I}$ and $\mathcal{J}$ be two ideals on two countably infinite sets $Z$ and $W$, respectively. Then we say that $\mathcal{I}$ \emph{is below} $\mathcal{J}$ \emph{in the Rudin--Blass ordering}, shortened as 
$\mathcal{I} \le_{\mathrm{RB}} \mathcal{J}$,  
if there is a finite-to-one map $\phi: W\to Z$ such that $S \in \mathcal{I}$ if and only if $\phi^{-1}[S] \in \mathcal{J}$ for all subsets $S\subseteq Z$. 
The restriction of these orderings to maximal ideals $\mathcal{I}$, and the Borel complexity of the quotients $\mathcal{P}(Z)/\mathcal{I}$ have been extensively studied, see e.g. \cite{MR0396267, MR1476761, MR1617463} and references therein. 

\begin{theorem}\label{thm:baire}
Let $\I$ be an ideal on $\omega$. Then the following are equivalent:
\begin{enumerate}[label={\rm (\roman{*})}]
\item \label{item:1baire} $\I$ has the Baire property;
\item \label{item:2baire} $\I$ is meager;
\item \label{item:3baire} $\I\restriction A$ has the Baire property for some $A\in \I^+$; 
\item \label{item:4baire} $\Fin\leq_{\mathrm{RB}}\I$.
\end{enumerate}
\end{theorem}
\begin{proof}
It follows by \cite[Theorem~4.1.1]{MR1350295} and \cite[Th\'{e}or\`{e}me 21, Proposition~22]{MR579439}. 
\end{proof}

Using the above characterization one can easily show that if ideals $\I$ and $\J$ are \emph{isomorphic} (that is, there exists a bijection $\phi:\omega\to\omega$ such that $S\in\I$ if and only if $\phi^{-1}[S]\in\J$ for all $S\subseteq\omega$), then $\I$ has the Baire property if and only if $\J$ has the Baire property. 
We refer to \cite{MR4566746} for further characterizations of meager ideals based on $\mathcal{I}$-limit points and $\I$-cluster points of sequences. 

Recall that an ideal $\I$ on $\omega$ has the \emph{hereditary Baire property} if the restriction $\I\restriction A$ has the Baire property for each $A\in \I^+$. Ideals with the hereditary Baire property were considered in \cite{MR3624786} in the context of ideal quasi-normal convergence of sequences of functions. It is not difficult to show that all Borel ideals, and more generally all analytic and coanalytic ideals, have the hereditary Baire property; cf. \cite[Theorem 3.13]{MR4358610} and \cite[Proposition~2.3]{MR3624786}. 
In addition, there exist ideals with the Baire property but without the hereditary Baire property, see e.g. \cite[Proposition 2.1]{MR3624786}. 

\begin{proposition}
\label{prop:BP-is-nowhere-maximal}
Every ideal with the hereditary Baire property is nowhere maximal. In particular, analytic ideals and coanalytic ideals are nowhere maximal. 
\end{proposition}

\begin{proof}
For each $A\in\I^+$, the ideal $\I\restriction A$ has the Baire property. Hence, it is folklore that it cannot be maximal, cf. \cite{Sierpinski1938}. 
\end{proof}

To distinguish between some properties of ideals considered in the paper, we will use examples constructed with the aid of Fubini products and Fubini sums. For, given (possibly nonadmissible) ideals $\mathcal{I}$ and $\mathcal{J}$ on two countably infinite sets $Z$ and $W$, respectively, their \emph{Fubini product} $\I \otimes \J$ is an ideal on $Z\times W$ which has already been defined in \eqref{eq:Fubiniproduct}. Moreover, their \emph{Fubini sum} $\I\oplus \J$ is an ideal on the countably infinite set $(\{0\}\times Z) \cup (\{1\}\times W)$ defined by $S \in \I\oplus \J$ if and only if 
$$
\{z \in Z: (0,z) \in S\} \in \I \,\,\text{ and }\,\,\{w \in W: (1,w) \in S\} \in \J,
$$
see e.g. \cite[Chapter 1]{MR1711328}. For instance, if $\I$ is an ideal on $\omega$, then $\I \oplus \mathcal{P}(\omega)$ is an ideal on $2\times\omega$. 
The next result will be used to verify the Baire property of ideals in Example~\ref{example:Pplus-vs-Pminus+Pvert}. 

\begin{proposition}
\label{prop:BP-of-sumas-and-products}
    Let $\I,\J$ be ideals on $\omega$. Then the following hold:
    \begin{enumerate}[label={\rm (\roman{*})}]
        \item \label{prop:BP-of-sumas-and-products:sum} $\I\oplus\J$ has the Baire property if and only if $\I$ or $\J$ has the Baire property;
        \item \label{prop:BP-of-sumas-and-products:product-I-times-empty} $\I\otimes \emptyset$ and $\I\otimes \Fin$ have the Baire property;
        \item \label{prop:BP-of-sumas-and-products:product-emptys-times-I} $\emptyset\otimes \I$ has the Baire property if and only if $\Fin\otimes \I$ has the Baire property if and only if $\I$ has the Baire property.
    \end{enumerate}
\end{proposition}

\begin{proof}
\ref{prop:BP-of-sumas-and-products:sum} The \textsc{If part} follows by \cite[Proposition~2.1]{MR3624786}. To prove the \textsc{Only If part}, suppose that $\I \oplus \J$ has the Baire property. Thanks to Theorem~\ref{thm:baire} we have that $\mathrm{Fin}\le_{\mathrm{RB}} \I \oplus \J$, hence there exists a finite-to-one function $f:\{0,1\}\times\omega \to\omega$ such that $A\in \Fin$ if and only if $f^{-1}[A]\in \I\oplus\J$ for every $A\subseteq \omega$. For each $i\in \{0,1\}$, define the map $f_i:\omega\to\omega$ by 
$  f_i(n)=f(i,n) $
for every  $n \in \omega$.
By construction, both $f_0$ and $f_1$ are finite-to-one. If $f_0$ is a witnessing map for $\Fin\leq_{\mathrm{RB}}\I$, then $\I$ has the Baire property by Theorem~\ref{thm:baire}, so the claim holds.
In the opposite, there exists an infinite set $B\subseteq\omega$ such that $f_0^{-1}[B]\in \I$. 
Since $B$ is an infinite set, we have that
$$
f^{-1}[B] = (\{0\}\times f_0^{-1}[B])\cup (\{1\}\times f_1^{-1}[B])\notin \I\oplus\J,
$$
which implies that $A=f_1^{-1}[B]\notin \J$. 
To conclude the proof, it would be enough to show that $g=f_1\restriction A$ is a witnessing map for $\Fin\leq_{\mathrm{RB}} \J\restriction A$, which would imply that $\J\restriction A$ has the Baire property and, hence, $\J$ as well by Theorem \ref{thm:baire}. Note that $g$ is also finite-to-one. 
Suppose for the sake of contradiction that there is an infinite set  $C\subseteq B$ such that $g^{-1}[C]\in \J\restriction A$. Since $A=f_1^{-1}[B]$ and $C\subseteq B$, we get $f^{-1}_1[C]=g^{-1}[C]$. It follows that 
$$
f^{-1}[C] \subseteq (\{0\}\times f^{-1}_0[B])\cup(\{1\}\times g^{-1}[C])\in \I\oplus\J, 
$$
which is the required contradiction.

\ref{prop:BP-of-sumas-and-products:product-I-times-empty}
Note that $\I\otimes \emptyset\subseteq \I\otimes \Fin$. 
Since an ideal has the Baire property if and only if it is meager by Theorem \ref{thm:baire} and the family of meager sets is closed under subsets, it is enough to show that $\I\otimes \Fin$ has the Baire property. For, we proceed as in the proof of \cite[Proposition 3.15]{MR4358610}: 
Set $A_n=\{0,1,\ldots,n\} \times \{n\}$ for each $n \in \omega$, so that $(A_n)$ is a sequence of finite disjoint subsets of $\omega^2$. At this point, it would be enough to show that $A_M=\bigcup_{n \in M}A_n \notin \I\otimes \Fin$ for each infinite $M\subseteq \omega$, which would imply that $\I\otimes \Fin$ is meager by Theorem \ref{thm:baire}. To this aim, recall that
$$
\I\otimes \Fin=\{S\subseteq \omega^2: \{n\in \omega: S \cap (\{n\}\times \omega)\in \Fin^+\} \in \I\}.
$$
To conclude, fix $M \in \mathrm{Fin}^+$. Hence $A_M \cap (\{n\}\times \omega)=\{(n,m): m \in M\cap [n,\infty)\}$ for all $n \in \omega$. It follows that $A_M \notin \I\otimes \Fin$, completing the proof.

\ref{prop:BP-of-sumas-and-products:product-emptys-times-I} 
Let $\I$ be an ideal on $\omega$. For each $n\in \omega$, define the ideal 
$$
\I_{(n)}=\{A\subseteq \omega^2: \{k\in \omega:(n,k)\in A\}\in \I\}.
$$
Since the ideals $\I_{(n)}\restriction (\{n\}\times\omega)$ and $\I$ are isomorphic, 
then  
$\I$ has the Baire property if and only if $\I_{(n)}$ has the Baire property for some $n$. Now, observe that 
\begin{equation}\label{eq:representationeemptysettimesI}
\emptyset\otimes\I = \bigcap\nolimits_{n\in \omega} \I_{(n)}
\quad \text{ and }\quad 
\Fin\otimes\I = \bigcup\nolimits_{n\in \omega}\bigcap\nolimits_{k\geq n} \I_{(k)}.
\end{equation}

First, suppose that $\I$ does not have the Baire property (or, equivalently, it is not meager by Theorem \ref{thm:baire}). 
In \cite[Proposition~23]{MR579439}, Talagrand proved that the intersection of countably many nonmeager ideals is not meager. 
It follows by \eqref{eq:representationeemptysettimesI} that $\emptyset\otimes\I$ is not meager. 

Second, since
$\emptyset\otimes \I\subseteq \Fin\otimes \I$, if $\emptyset\otimes \I$ is not meager then $\Fin\otimes \I$ is not meager.

Third, since union of countably many meager sets is meager, we obtain by \eqref{eq:representationeemptysettimesI} that if $\I$ is meager, then $\Fin\otimes\I$ is meager as well. This concludes the proof of the equivalences. 
\end{proof}

 \begin{remark}\label{label:Finxmaximal}
As a special case of item \ref{prop:BP-of-sumas-and-products:product-emptys-times-I}, we recover \cite[Proposition~2.11]{MR4358610} where it is shown that, if $\I$ is a maximal ideal, then $\Fin\otimes\I$ does not have the Baire property. In particular, ideals with the hereditary Baire property are not the only ones which are nowhere maximal. A stronger claim on the ideal $\Fin\otimes\I$ (namely, $\Fin\otimes\I \notin P(\Pi^0_1)$, see details in Section \ref{sec:Ischemes}) can be found in Proposition \ref{prop:P-for-closed2}\ref{item:1Ppi05} below. 

On the other hand, with the same arguments, $\emptyset\otimes \I$ is an ideal without the Baire property which is not nowhere maximal. 
\end{remark}


\subsection{Some basic properties} 

We conclude this preliminary section with some standard properties of families $\mathscr{C}(\I)$ and $\mathscr{L}(\I)$ related to Fubini sums, products, and Rudin--Blass orderings. (Hereafter, $\limsup_n A_n=\bigcap_{n\ge 0} \bigcup_{k\ge n}A_k$.)
\begin{proposition}\label{prop:basic-prop-about-Lambda-Gamma}
Let $X$ be a topological space, let $\I, \J$ be ideals on $\omega$, and fix $\alpha<\omega_1$. Then the following hold:
\begin{enumerate}[label={\rm (\roman{*})}]
\item \label{prop:RB-for-Lambda:Lambda-sets}
     If $\I\leq_{\mathrm{RB}}\J\restriction A$ for some $A\in \J^+$, then 
    $\mathscr{L}(\I) \subseteq \mathscr{L}(\J\restriction A)\subseteq \mathscr{L}(\J)$;

\item \label{item:0RB} If $\I\le_{\mathrm{RB}} \J$ and $\J$ is $\Sigma^0_\alpha$ \textup{[}resp., $\Pi^0_\alpha$\textup{]}, then $\I$ is $\Sigma^0_2$ \textup{[}resp., $\Pi^0_\alpha$\textup{]};

\item \label{item:2RB} $\mathscr{L}(\emptyset \otimes \I)=\left\{\bigcup_n A_n: A_0,A_1,\ldots \in \mathscr{L}(\I)\right\}$;

\item \label{item:3RB} If $X$ is first countable, $\mathscr{L}(\Fin \otimes \I)=\left\{\limsup_n A_n: A_0,A_1,\ldots \in \mathscr{L}(\I)\right\}$;

\item \label{prop:basic-prop-about-Lambda-Gamma:direct-sum} 
$\mathscr{L}(\I\oplus \J)=\{A\cup B: A \in \mathscr{L}(\I), B \in \mathscr{L}(\J)\}$.
\end{enumerate}
\end{proposition}
\begin{proof}
\ref{prop:RB-for-Lambda:Lambda-sets} 
The inclusion $\mathscr{L}(\I) \subseteq \mathscr{L}(\J\restriction A)$ is shown in \cite[Proposition~3.8]{MR4393937} (note that although in the statement of this result the space $X$ is assumed to be first countable, the proof actually works in every topological space). For the second inclusion: if $  x =(x_n: n \in A)$ is a sequence with nonempty $W=\Lambda_{  x }(\J \restriction A)$ and $\eta \in W$, then the sequence $ y =(y_n: n \in \omega)$ defined by $y_n=x_n$ if $n \in A$ and $y_n=\eta$ if $n\notin A$ satisfies $\Lambda_{ y }(\J)=W$. Therefore $\mathscr{L}(\J\restriction A)\subseteq \mathscr{L}(\mathcal{J})$.

\ref{item:0RB} Let $f:\omega\to\omega$ be a finite-to-one function such that $A\in \I$ if and only if $f^{-1}[A]\in \J$. 
Define the map $\phi:\cP(\omega)\to\cP(\omega)$ by $\phi(A)= f^{-1}[A]$ for all $A\subseteq \omega$. 
Observe that $\phi$ is continuous because $f$ is finite-to-one. 
Since $\I = \phi^{-1}[\J]$, the Borel complexity of $\I$ is the same as the complexity of $\J$.

\ref{item:2RB}  See \cite[Corollary 2.3]{MSP24}.

\ref{item:3RB} See \cite[Corollary 2.4]{MSP24}.

\ref{prop:basic-prop-about-Lambda-Gamma:direct-sum} 
Pick a nonempty set $S \in \mathscr{L}(\I\oplus \J)$. Pick a sequence $  x =(x_{(k,n)}: (k,n) \in \{0,1\}\times \omega)$ such that $\Lambda_{  x }(\I\oplus \J)=S$. For each $k \in \{0,1\}$, define the sequence $x^{(k)}=(x_n^{(k)}: n \in\omega)$ by
\begin{equation}\label{eq:identityfubinisum}
x^{(k)}_n=x_{(k,n)} \quad  \text{ for each $n\in \omega$.}
\end{equation}
By the definition of $\I\oplus \J$, it is clear that $S=\Lambda_{x^{(0)}}(\I) \cup \Lambda_{x^{(1)}}(\J)$. Hence $\mathscr{L}(\I\oplus \J)$ is contained in $\{A\cup B: A \in \mathscr{L}(\I), B \in \mathscr{L}(\J)\}$. For the opposite inclusion: given two sequences $x^{(0)}$ and $x^{(1)}$, define $  x =(x_{(k,n)}: (k,n) \in \{0,1\}\times \omega)$ as the unique sequence which satisfies \eqref{eq:identityfubinisum} for each $k \in \{0,1\}$. The conclusion goes similarly. 
\end{proof}

\begin{corollary}\label{cor:Lemptysetmaximal}
Let $X$ be a first countable Hausdorff space, and let $\I$ be an ideal on $\omega$. Then the following hold:
\begin{enumerate}[label={\rm (\roman{*})}]
\item \label{item:1emptysetIfinI} $[X]^{\le \omega} \subseteq \mathscr{L}(\emptyset \otimes \I)\subseteq \mathscr{L}(\Fin \otimes \I)$;
\item \label{item:2emptysetIfinI} If $\I$ is maximal then 
$$
\mathscr{L}(\emptyset \otimes \I)=\mathscr{L}(\mathrm{Fin}\otimes \I)=[X]^{\le \omega}.
$$
\end{enumerate}
\end{corollary}
\begin{proof}
\ref{item:1emptysetIfinI} First, it is clear that $\{\{\eta\}: \eta \in X\}\subseteq \mathscr{L}(\I)$. The first inclusion follows by Proposition \ref{prop:basic-prop-about-Lambda-Gamma}\ref{item:2RB}. For the second inclusion, pick a nonempty $A \in \mathscr{L}(\emptyset \otimes \I)$. Again by Proposition \ref{prop:basic-prop-about-Lambda-Gamma}\ref{item:2RB}, there exist $A_0,A_1,\ldots \in \mathscr{L}(\I)$ such that $A=\bigcup_nA_n$. Let also $f: \omega \to \omega$ be a surjective map such that $f^{-1}(\{n\})$ is infinite for all $n \in \omega$. Note that the sequence $(A_{f(n)}: n \in \omega)$ satisfies $A=\limsup_n A_{f(n)}$. Therefore $A \in \mathscr{L}(\Fin\otimes \I)$ by Proposition \ref{prop:basic-prop-about-Lambda-Gamma}\ref{item:3RB}.

\ref{item:2emptysetIfinI} By the previous item, it is enough to show that $\mathscr{L}(\Fin\otimes \I)\subseteq [X]^{\le \omega}$. For, thanks to Proposition \ref{prop:Imaximal} and Proposition \ref{prop:basic-prop-about-Lambda-Gamma}\ref{item:3RB}, we have
$$
\mathscr{L}(\Fin \otimes \I)=\left\{\limsup_{n\to \infty}\, \{\eta_n\}: \eta_0,\eta_1,\ldots \in X\right\},
$$
and the latter set is precisely $[X]^{\le \omega}$. 
\end{proof}

We remark that all inclusions in item \ref{item:1emptysetIfinI} can be strict: indeed, if $X$ is Polish and $\I=\Fin$ then $\mathscr{L}(\emptyset\otimes \Fin)=\Sigma^0_2$ and $\mathscr{L}(\Fin^2)=\Pi^0_3$ by Theorem \ref{thm:highBorelMSP}. 

\begin{corollary}\label{cor:emptysettimesmaximalplusfin}
Let $X$ be a first countable Hausdorff space where all closed sets are separable. Also, let $\I$ be a maximal ideal and $\J$ be 
an $F_\sigma$-ideal. Then 
$$
\mathscr{L}((\emptyset \otimes \I)\oplus \J)=\mathscr{L}((\fin \otimes \I)\oplus \J)=\{A\cup B: A \in [X]^{\le \omega}, B \in \Pi^0_1\}.$$
\end{corollary}
\begin{proof}
It follows by 
Theorem \ref{thm:oldthmballeo}, 
Proposition \ref{prop:basic-prop-about-Lambda-Gamma}\ref{prop:basic-prop-about-Lambda-Gamma:direct-sum}, and 
Corollary \ref{cor:Lemptysetmaximal}\ref{item:2emptysetIfinI}. 
\end{proof}

Lastly, we conclude with a slight extension of Proposition \ref{prop:Imaximal}:
\begin{proposition}\label{prop:fubinimaximal}
  Pick an integer $n\ge 1$. Let $X$ be a Hausdorff space with $|X|\ge n+1$ and let $\I$ be an ideal on $\omega$. Then the following are equivalent: 
\begin{enumerate}[label={\rm (\roman{*})}]
\item \label{item:max1BB} $\I=\I_0\oplus \cdots \oplus \I_{n-1}$ for some maximal ideals $\I_0,\ldots,\I_{n-1}$;
\item \label{item:max2BB} $\mathscr{L}(\I)=\mathscr{C}(\I)=
[X]^{\le n}$;
\item \label{item:max3BB} $\mathscr{C}(\I)=[X]^{\le n}$;
\item \label{item:max4BB} $\mathscr{L}(\I)=[X]^{\le n}$.
\end{enumerate}
\end{proposition}
\begin{proof}
\ref{item:max1BB} $\implies$ \ref{item:max2BB}. Suppose that $\I=\I_0\oplus \cdots \oplus \I_{n-1}$ for some maximal ideals $\I_0,\ldots,\I_{n-1}$. It follows by induction using Proposition \ref{prop:Imaximal} and Proposition \ref{prop:basic-prop-about-Lambda-Gamma}\ref{prop:basic-prop-about-Lambda-Gamma:direct-sum} that $\mathscr{L}(\I)=[X]^{\le n}$. The equality $\mathcal{C}(\I)=[X]^{\le n}$ follows similarly as in the proof of Proposition \ref{prop:Imaximal}, considering that there are no $n+1$ pairwise disjoint $\I$-positive sets.

\ref{item:max2BB} $\implies$ \ref{item:max3BB} $\implies$ \ref{item:max4BB}. It goes as in the proof of Proposition \ref{prop:Imaximal}.

\ref{item:max4BB} $\implies$ \ref{item:max1BB}. Assume that $\mathscr{L}(\I)=[X]^{\le n}$. 
    Let $k \in \omega \cup \{\infty\}$ be the maximal cardinality of a partition of $\omega$ into $\I$-positive sets. 
    
    First, suppose that $k\le n-1$. Let $\{S_0,\ldots,S_{k-1}\}\subseteq \I^+$ be a partition of $\omega$. By the maximality of $k$, each $S_i$ cannot be partitioned into two $\I$-positive sets. This implies that each $\I\restriction S_i$ is maximal, hence $\I$ is the Fubini sum of $k$ maximal ideals. It follows by the proof of the first implication above that $\mathscr{L}(\I)=[X]^{\le k}$, which is different from $[X]^{\le n}$ since $X$ has at least $n$ points. 
    
    Second, suppose that $k \ge n+1$. Hence there exists a partition $\{S_0,\ldots,S_{n}\}$ of $\omega$ into $n+1$ $\I$-positive sets. It follows by Lemma \ref{lem:Xlen} that $[X]^{\le n+1}\subseteq \mathscr{L}(\I)$, which contradicts the standing hypothesis $\mathscr{L}(\I)=[X]^{\le n}$. 
    
    The above cases imply that $k=n$. Hence there exists a partition $\{S_0,\ldots,S_{n-1}\}$ of $\omega$ into $\I$-positive sets, and each $S_i$ cannot be partitioned into two $\I$-positive sets. Therefore $\I$ is the Fubini sum of $k=n$ maximal ideals.
\end{proof}


\section{P-like properties}\label{sec:Plike}

Let us start recalling the definition of $P^+$-ideal, together with two weaker variants. 

\begin{definition}\label{def:Pplusandweakers}
Let $\I$ be an ideal on $\omega$. We say that $\I$ is a:

\begin{enumerate}[label={\rm (\roman{*})}]
\item  $P^+$\emph{-ideal} if for every decreasing sequence $(A_n:n \in \omega)$ of $\I$-positive sets there exists $A\in \I^+$ such that $A\setminus A_n$ is finite for every $n \in \omega$;

\item  $P^-$\emph{-ideal} if for every decreasing sequence $(A_n:n \in \omega)$ of $\I$-positive sets with $A_n\setminus A_{n+1}\in \I$ for every $n$ there exists $A\in \I^+$ such that $A\setminus A_n$ is finite for every $n \in \omega$;

\item $P^{\,|}$\emph{-ideal} if for every decreasing sequence $(A_n:n \in \omega)$ of $\I$-positive sets with $A_n\setminus A_{n+1}\in \I^+$ for every $n$ there exists $A\in \I^+$ such that $A\setminus A_n$ is finite for every $n \in \omega$.
\end{enumerate}
\end{definition}

It is worth noting that $P^-$-ideals and $P^{\,|}$-ideals are also called \textquotedblleft hereditary weak $P$-ideals\textquotedblright\, and \textquotedblleft weakly $P^+$-ideals\textquotedblright,\, respectively, see e.g. \cite{MR4584767} and \cite[Definition~2.1]{MR4505549}. 

The following theorem and examples summarize the relationships between $P$-like properties and the topological complexity of ideals. These properties and examples are graphically summarized in Figure~\ref{fig:P-vs-P} below.

For, recall that an ideal $\I$ is \emph{countably generated} if there is a countable family $\cC\subseteq\cP(\omega)$ such that $\I=\{A\subseteq\omega: A\subseteq\bigcup\cC'\text{ for some }\cC'\in[\cC]^{<\omega}\}$. Moreover, if $\mathcal{I}$ and $\mathcal{J}$ are two ideals on two countably infinite sets $Z$ and $W$, respectively, then we say that $\mathcal{I}$ \emph{is below} $\mathcal{J}$ \emph{in the Kat\v{e}tov ordering}, shortened as 
$\mathcal{I} \le_{\mathrm{K}} \mathcal{J}$,  
if there is a map $\phi: W\to Z$ such that $\phi^{-1}[S] \in \mathcal{J}$ for all $S \in \mathcal{I}$. (Of course, this notion is weaker than Rudin--Blass ordering.)

\begin{proposition}\label{prop:Pplus-vs-Pminus+Pvert}
Let $\I$ be an ideal on $\omega$. Then the following hold:
\begin{enumerate}[label={\rm (\roman{*})}]
    \item $\I$ is $P^+$ if and only if it is both $P^{\,|}$ and $P^-$; \label{prop:Pplus-vs-Pminus+Pvert:item}

\item Each countably generated ideal is $\Sigma^0_2$;\label{prop:Pplus-vs-Pminus+Pvert:countably-generated}

\item Each $\Sigma^0_3$ ideal is $P^+$; \label{thm:Plike-properties-for-definable-ideals:Fsigma}

\item Each analytic $P$-ideal is $\Pi^0_3$;\label{thm:Plike-properties-for-definable-ideals:analyticP}

\item Each $\Pi^0_4$ ideal is $P^-$;\label{thm:Plike-properties-for-definable-ideals:Pi-zero-four-is-Pminus}

\item  If $\I$ is an analytic ideal, then it can be extended to a $P^+$-ideal if and only if it can be extended to a $\Sigma^0_2$ ideal; \label{thm:Plike-properties-for-definable-ideals:exension-to-Fsigma}

\item If $\I$ is nowhere maximal, then 
$\I$ is $P^+$ if and only if it is $P^{\,|}$;\label{prop:Pplus-vs-Pminus+Pvert:nowhere-max}

 \item If $\I$ has the hereditary Baire property, then 
    $\I$ is $P^+$ if and only if it is $P^{\,|}$;\label{prop:Pplus-vs-Pminus+Pvert:hBaire}

 \item If $\I$ has the hereditary Baire property and $\I$ is not $ P^{\,|}$, then 
    $\emptyset\otimes\fin \leq_{\mathrm{RB}} \I$;\label{prop:Pplus-vs-Pminus+Pvert:hBaire-Pvert}

\item  $\I$ is $P^-$ if and only if $\Fin^2\not\leq_{\mathrm{K}} \I\restriction A$ for every $A\in \I^+$. In particular, every ideal extending $\Fin^2$ is not $P^-$.\label{prop:Pplus-vs-Pminus+Pvert:FinSQRD} 
\end{enumerate}
\end{proposition}

\begin{proof}
\ref{prop:Pplus-vs-Pminus+Pvert:item} 
The \textsc{Only If part} is clear from Definition \ref{def:Pplusandweakers}, hence let us prove the \textsc{If part}. 
Fix a decreasing sequence $(A_n: n\in \omega)$ of $\I$-positive sets and define $S=\{n \in \omega: A_n\setminus A_{n+1} \in \I^+\}$. 
First, suppose that $S$ is finite. Since $\I$ is $P^-$ and the sequence is decreasing, there exists $A \in \I^+$ such that $A\setminus A_n$ is finite for all $n\in \omega$. 
In the opposite, $S$ is infinite. Let $(s_n: n \in \omega)$ be the increasing enumeration of $S$, so that $A_{s_{n+1}}\setminus A_{s_n} \in \I^+$ for all $n \in \omega$. Since $\I$ is $P^{\,|}$, there exists $A \in \I^+$ such that $A\setminus A_{s_n}$ is finite for all $n\in \omega$. The stronger claim that $A\setminus A_{n}$ is finite for all $n\in \omega$ follows from the fact that $(A_n: n \in \omega)$ is decreasing.

\ref{prop:Pplus-vs-Pminus+Pvert:countably-generated}
Let $\mathcal{C}=\{C_n:n\in\omega\}\subseteq \mathcal{P}(\omega)$ be a countable family which generates the ideal $\I$, so that 
$$
\I=\bigcup_{n \in \omega} \, \left\{S\subseteq \omega:n \notin S \text{ for all }n \in \omega\setminus (C_0\cup \cdots \cup C_{n})\right\}.
$$
Therefore $\I$ is a countable union of closed sets. More precisely, an ideal on $\omega$ is countably generated if and only if it is isomorphic to $\mathrm{Fin}$ or $\Fin \otimes \emptyset$ or $\Fin \oplus \mathcal{P}(\omega)$, see \cite[Proposition 1.2.8]{MR1711328} and \cite[Remark 2.16]{Leo24unbounded}, and the latter ones are $\Sigma^0_2$.

\ref{thm:Plike-properties-for-definable-ideals:Fsigma} 
See \cite[Corollary 2.7]{MSP24}; cf. also \cite[Lemma 1.2]{MR748847} for the folklore case of $\Sigma^0_2$ ideals (however, there exists a $\Delta^0_3$-ideal on $\omega$ which is not $\Sigma^0_2$, see \cite[p. 345]{MR1416872} and \cite{MR1321463}).

\ref{thm:Plike-properties-for-definable-ideals:analyticP} 
See \cite[Lemma 1.2.2 and Theorem 1.2.5]{MR1711328}.

\ref{thm:Plike-properties-for-definable-ideals:Pi-zero-four-is-Pminus} 
See \cite[Proposition~4.9]{MR4584767}.

\ref{thm:Plike-properties-for-definable-ideals:exension-to-Fsigma}
See \cite[Theorem~3.7]{MR3692233}.

\ref{prop:Pplus-vs-Pminus+Pvert:nowhere-max} 
By item \ref{prop:Pplus-vs-Pminus+Pvert:item}, every $P^+$-ideal is $P^{\,|}$. 
Conversely, 
let $\I$ be a nowhere maximal ideal which is $P^{\,|}$. We will show that $\I$ is $P^-$, which will complete the proof by item \ref{prop:Pplus-vs-Pminus+Pvert:item}. Take a decreasing sequence $(A_n: n \in \omega)$ of $\I$-positive sets such that $A_n\setminus A_{n+1}\in \I$ for each $n\in \omega$. 
Thanks to Proposition~\ref{prop:nowhere-maximal}, there is an infinite partition $\{B_n:n \in \omega\}$ of $A_0$ into $\I$-positive sets. 
Define $C_n=A_n\cap B_n$ and $D_n=\bigcup_{k\ge n}C_k$ for each $n \in \omega$. 
Since $A_0\setminus A_n\in \I$ for each $n$, we have $C_n=B_n\setminus (A_0\setminus A_n)\in \I^+$. Then $(D_n: n \in \omega)_{n\in\omega}$ is a decreasing sequence of $\I$-positive sets such that $D_{n}\setminus D_{n+1}=C_n\in\I^+$. 
Since $\I$ is $P^{\,|}$, there exists $D\in\I^+$ such that $D\setminus D_n$ is finite for each $n\in \omega$. 
Considering that $D_n=\bigcup_{k\ge n}(A_k\cap B_k) \subseteq \bigcup_{k\ge n}A_k = A_n$, we obtain that $D\setminus A_n\subseteq D\setminus D_n$. Hence $D\setminus A_n$ is finite for each $n\in \omega$. This shows that $\I$ is $P^-$.

\ref{prop:Pplus-vs-Pminus+Pvert:hBaire} 
It follows from Proposition~\ref{prop:BP-is-nowhere-maximal} and item \ref{prop:Pplus-vs-Pminus+Pvert:nowhere-max}; for a different (direct) proof, see also \cite[Theorem~2.8]{MR4505549}.

\ref{prop:Pplus-vs-Pminus+Pvert:hBaire-Pvert} 
Since $\I$ is not a  $P^{\,|}$-ideal, there exists a decreasing sequence $(A_n:n\in \omega)$ of subsets of $\omega$ such that $A_n\setminus A_{n+1}\in \I^+$ for each $n \in \omega$ and, for every $B\in \I^+$, there is $k\in\omega$ with $B\setminus A_k \notin \mathrm{Fin}$. Accordingly, define 
$$
B_0 = (A_0\setminus A_1)\cup (\omega\setminus A_0)\cup \bigcap_{n\in \omega}A_n
\quad \text{ and } \quad 
B_n = A_n\setminus A_{n+1} \text{ for all }n\ge 1.
$$
Then $\{B_n: n\in \omega\}$ is a partition of $\omega$ into $\I$-positive sets. 
Thanks to Theorem~\ref{thm:baire}, for each $n\in \omega$ there exists a finite-to-one map $\phi_n:B_n\to\omega$ such that $C\in \fin$ if and only if $\phi_n^{-1}[C]\in \I\restriction B_n$ for each $C\subseteq\omega$.
We define a function $f:\omega\to\omega^2$ by 
$f(k)= (n,\phi_n(k))$ for every $k\in B_n$ and $n\in \omega$.
It is no difficult to see that the function $f$ is a witnessing function for $\emptyset\otimes \fin\leq_{\mathrm{RB}} \I$, so the proof is finished.

\ref{prop:Pplus-vs-Pminus+Pvert:FinSQRD} 
See \cite[Theorem~3.8]{MR3692233}.
\end{proof}

In the following examples, all considered properties either follow from Proposition \ref{prop:Pplus-vs-Pminus+Pvert} or we provide an appropriate reference, or otherwise they  are  straightforward.

\begin{example}
\label{example:P-like}
\begin{enumerate}[label={\rm (\roman{*})}]
\item $\Fin$ is a $\Sigma^0_2$ ideal which is 
$P^+$.

\item $\Z$ is a $\Pi^0_3$ ideal, see \cite[Example 1.2.3]{MR1711328}, hence it is $P^-$ by Proposition \ref{prop:Pplus-vs-Pminus+Pvert}\ref{thm:Plike-properties-for-definable-ideals:Pi-zero-four-is-Pminus}. But it is not $P^{\,|}$: for, it is enough to consider the witnessing sequence $(S_k: k \in \omega)$, with $S_k=\{n \in \omega: 2^{k+1} \text{ divides } n-2^k\}$ for each $k \in \omega$.\label{example:P-like:Id}

\item $\Fin^2$ is a $\Sigma^0_4$ ideal, see 
Theorem \ref{thm:highBorelMSP}, 
but it is neither $P^-$ nor $P^{\,|}$. It follows by Proposition \ref{prop:Pplus-vs-Pminus+Pvert}\ref{thm:Plike-properties-for-definable-ideals:Pi-zero-four-is-Pminus} that $\Fin^2$ is not $\Pi^0_4$. 
    \end{enumerate}
\end{example}

\begin{example}
\label{example:maximal-vs-P-like}
\begin{enumerate}[label={\rm (\roman{*})}]

        \item Every maximal ideal is $P^{\,|}$, see \cite[Proposition~2.6]{MR4505549}. In particular, there is no maximal ideal which is $P^-$ but not $P^+$.\label{example:maximal-vs-P-like:max-is-P-prime}

\item There exists a maximal ideal which is not $P^+$ (hence $P^{\,|}$ but not $P^-$): for instance, any maximal ideal extending $\Fin^2$ is not $P^-$, see e.g.~\cite[p.~76]{MR1940513}.

\item \label{item:ZFCexample:maximal-vs-P-like} It is consistent (e.g.,~assuming the Continuum Hypothesis) that there exists a maximal ideal which is $P^+$, see \cite{MR0080902}. However, it is also consistent that there are no maximal ideals which are $P^+$, see \cite{MR0728877}. 

    \end{enumerate}
\end{example}

In the following example, the Baire property of ideals is verified with the aid of 
Proposition~\ref {prop:BP-of-sumas-and-products}, and the other properties are straightforward.

\begin{example}
\label{example:Pplus-vs-Pminus+Pvert}
Let $\cM,\cM_p,\cM_{n}$ be maximal ideals on $\omega$ such that $\cM_p$ is $P^+$ and $\cM_{n}$ is not $P^+$ (the same notation is used in Figure~\ref{fig:P-vs-P} below). Note that, by Example \ref{example:maximal-vs-P-like}\ref{item:ZFCexample:maximal-vs-P-like}, the existence of $\cM_p$ is independent of ZFC axioms. 
\begin{enumerate}[label={\rm (\roman{*})}]

    \item $\cM_p\oplus \cM_p$ is a $P^+$-ideal, is not maximal, and does not have the Baire property. 
    
    \item $\cM\oplus \cM_{n}$ is $P^{\,|}$, is not $P^-$, is not maximal, and does not have the Baire property. 

    \item $\cM_p\oplus \Fin$ is $P^+$, has the Baire property, does not have the hereditary Baire property. In addition, it is immediate that it is not nowhere maximal. \label{item:MplusFin}

    \item $\cM_{n}\oplus \Fin$ is $P^{\,|}$, is not $P^-$, has the Baire property, and does not have the hereditary Baire property.

    \item $\emptyset\otimes \cM_p$ is $P^-$, is not $P^{\,|}$, and does not have the Baire property.

    \item $\fin\oplus(\emptyset\otimes \cM_p)$ is $P^-$, is not $P^{\,|}$, and has the Baire property.

\item $\emptyset\otimes \cM_{n}$ is not $P^-$, is not $P^{\,|}$, and does not have the Baire property.

\item $\Fin^2\oplus  \cM_{n}$ is not $P^-$, is not $P^{\,|}$, has the Baire property, and does not have the hereditary Baire property.
\end{enumerate}
\end{example}

\begin{figure}
\centering
\begin{tikzpicture}[scale =0.8] 

\draw (0,0) ellipse (7.4cm and 5.3cm);

\draw ({3/2},{3*sqrt(3)/3}) circle (3cm);
\draw ({-3/2},{3*sqrt(3)/3}) circle (3cm);

\draw
(-0.5,{-3*sqrt(3)/3}) ellipse (5.2cm and 3.2cm);
\draw (1.8,2.5) ellipse (1.8cm and 1cm);

\draw (-1.2,-1.5) ellipse (1.2cm and 2.5cm);

\node at (2,4.2) {\Large $\mathbf{P}^{\,|}$};
\node at (0,3.7) {\Large $\mathbf{P}^+$};
\node at (-2,4.2) {\Large $\mathbf{P}^-$};

\node at (-.4,-4.5) {\textbf{\textsc{Baire}}};

\node at (1.7,1.95) {\textbf{\textsc{Maximal}}};

\node at (-1.15,-3) {\small \textbf{\textsc{Hered.}}};
\node at (-1.15,-3.3) {\small \textbf{\textsc{Baire}}};

\draw[fill] (-0.5,0) circle (0.06cm);
\node at (-0.7,0.3) {$\fin$};

\draw[fill] (-1.5,-0.9) circle (0.06cm);
\node at (-1.2,-0.9) {$\Z$};

\draw[fill] (-2.6,.3) circle (0.06cm);
\node 
[rotate=20] 
at (-2.8,.7) {\small $\Fin\oplus (\emptyset\otimes  \mathcal{M}_p)$};

\draw[fill] (-1.25,-2) circle (0.06cm);
\node at (-0.7,-2) {\small $\fin^2$};

\draw[fill] (0.7,-2.5) circle (0.06cm);
\node at (1.8,-2.5) {\small $\fin^2\oplus \mathcal{M}_n$};

\draw[fill] (0.5,0.5) circle (0.06cm);
\node at (0.4,0.8) {\small $\fin\oplus \mathcal{M}_p$};

\draw[fill] (-3.6,2) circle (0.06cm);
\node at (-2.8,2) {\small $\emptyset\otimes  \mathcal{M}_p$};

\draw[fill] (0.5,2.8) circle (0.06cm);
\node at (0.7,2.4) {\small $\mathcal{M}_p$};

\draw[fill] (2.4,2.8) circle (0.06cm);
\node at (2.6,2.4) {\small $\mathcal{M}_n$};

\draw[fill] (-0.8,2.2) circle (0.06cm);
\node at (-0.6,1.8) {\small $\mathcal{M}_p\oplus \mathcal{M}_p$};

\draw[fill] (3.5,0.7) circle (0.06cm);
\node at (3.4,1.1) {\small $\mathcal{M}\oplus \mathcal{M}_n$};

\draw[fill] (1.7,-.2) circle (0.06cm);
\node at (1.7,-0.5) {\small $\fin\oplus \mathcal{M}_n$};

\draw[fill] (5.2,0) circle (0.06cm);
\node at (6,0) {\small $\emptyset\otimes  \mathcal{M}_n$};

\end{tikzpicture}
\caption{P-like properties and the Baire property among ideals on $\omega$.}
    \label{fig:P-vs-P}
\end{figure}

Finally, we show how the combinatorial properties given in Definition \ref{def:Pplusandweakers} allow to characterize several relationships between $\I$-cluster points and $\I$-limit points. 
For the sake of simplicity, most of the next results are formulated in the setting of metric spaces, even though some of them can be easily generalized to first countable spaces. 
(Recall that $A^{\,|}$ stands for the set of accumulation points of a set $A$.)

\begin{proposition}
\label{prop:Pprimecharacterization}
Let $X$ be a nondiscrete metric space and let $\I$ be an ideal on $\omega$. Then the following conditions are equivalent:

\begin{enumerate}[label={\rm (\roman{*})}]

\item $\I$ is $P^{\,|}$;\label{prop:P-like_for_spaces:Pbar-nonisolated-in-Gamma-are-in-Lambda:Pbar}

\item $\Gamma_{  x }(\I)^{\,|}\subseteq \Lambda_{  x }(\I)$ for all sequences $  x $;\label{prop:P-like_for_spaces:Pbar-nonisolated-in-Gamma-are-in-Lambda:derived-set}

\item $\Lambda_{  x }(\I)$ is closed for all sequences $  x $, i.e., $\mathscr{L}(\I)\subseteq \Pi^0_1$.\label{item:3pprime}

\end{enumerate}
\end{proposition}
\begin{proof}
\ref{prop:P-like_for_spaces:Pbar-nonisolated-in-Gamma-are-in-Lambda:Pbar} $\implies$ \ref{prop:P-like_for_spaces:Pbar-nonisolated-in-Gamma-are-in-Lambda:derived-set}. 
Let $  x =(x_i)_{i\in\omega}$ be a sequence with values in $X$ and take a one-to-one 
sequence $ y =(y_n)_{n\in\omega}$ in $\Gamma_{  x }(\I)$ which converges to a point $\eta \in X$. We need to show that $\eta \in \Lambda_{  x }(\I)$. 
Let $(r_n: n \in \omega)$ be a sequence of positive reals such that $\lim_n r_n=0$ and, in the addition, the open balls $B(y_n,r_n)$ are pairwise disjoint. Define 
$$
A_n= \{i\in \omega: x_i\in B (y_n,r_n)\} 
\quad \text{ and }\quad 
B_n=\bigcup_{k\ge n}A_k
$$
for all $n \in \omega$. 
It follows by construction that $\{A_n: n \in \omega\}$ is a family of pairwise disjoint $\I$-positive sets, and that $(B_n: n \in \omega)$ is a decreasing sequence such that $B_n\setminus B_{n+1}\notin\I$ for every $n\in \omega$. 
Since $\I$ is $P^{\,|}$, there exists $B\in \I^+$ such that $B\setminus B_n$ is finite for each $n\in \omega$.
This implies that the subsequence $(x_i: i\in B)$ converges to the point $\eta$, therefore $\eta$ is an $\I$-limit point of $  x $. 

\ref{prop:P-like_for_spaces:Pbar-nonisolated-in-Gamma-are-in-Lambda:derived-set}
$\implies$ \ref{prop:P-like_for_spaces:Pbar-nonisolated-in-Gamma-are-in-Lambda:Pbar}. 
For each sequence $  x $, the set $\Gamma_{  x }(\I)$ is closed by Lemma \ref{lem:clusterclosed}, hence $\{\Gamma_{  x }(\I)^{\,|}, \mathrm{iso}(\Gamma_{  x }(\I))\}$ is a partition of $\Gamma_{  x }(\I)$. Hence item \ref{prop:P-like_for_spaces:Pbar-nonisolated-in-Gamma-are-in-Lambda:derived-set} can be rewritten equivalently as $\Gamma_{  x }(\I) \setminus \mathrm{iso}((\Gamma_{  x }(\I)) \subseteq \Lambda_{  x }(\I)$ for all sequences $  x $. 

At this point, fix a decreasing sequence $(A_n: n\in \omega)$ of subsets of $\omega$ such that $A_n\setminus A_{n+1}\in \I^+$ for all $n \in \omega$. 
If $A=\bigcap_n A_n \notin \I$, we are done (because $A$ satisfies the $P^{\,|}$-property for the sequence $(A_n)$). 
Otherwise, assume hereafter that $A\in \I$. 

Since $X$ has a nonisolated point, there is a one-to-one sequence  $ y =(y_n)_{n\in\omega}$ which converges to some $\eta \in X\setminus \{y_n:n\in \omega\}$. Define the sequence $  x =(x_i: i \in \omega)$ by
$$
x_i = \begin{cases}
\,y_0 &\text{if $i\in \omega\setminus A_0$,}\\
\,\eta &\text{if $i\in A$,}\\
\,y_n&\text{if $i\in A_n\setminus A_{n+1}$.}
\end{cases}
$$

Notice that each $y_n$ is an $\I$-cluster point of $  x $:  
indeed, if $U$ is an open neighborhood of $y_n$, then $\{i\in \omega:x_i\in U\}\supseteq A_n\setminus A_{n+1}\in \I^+$. In addition, it follows by Lemma \ref{lem:clusterclosed} that $\eta$ is an $\I$-cluster point as well. 
Consequently, $\eta$ is a nonisolated 
point of $\Gamma_{  x }(\I)$, hence $\eta \in \Lambda_{  x }(\I)$. 
This means that there exists $A\in \I^+$ such that the subsequence $(x_i: i\in A)$ converges to $\eta$, therefore $A\setminus A_n$ is finite for every $n\in \omega$.

\ref{prop:P-like_for_spaces:Pbar-nonisolated-in-Gamma-are-in-Lambda:Pbar} $\Longleftrightarrow$ \ref{item:3pprime}. See \cite[Theorem 2.4]{MR4505549}. 
\end{proof}

An additional characterization of 
the inclusion 
$\mathscr{L}(\I)\subseteq \Pi^0_1$, 
in the case of ideals with the hereditary Baire property, 
will be given in 
Corollary \ref{cor:inclusionclosedsets} below.

Similarly, we characterize the $P^-$-property as the case where every isolated $\I$-cluster point is necessarily an $\I$-limit point: 
\begin{proposition}
\label{prop:Pminuscharacterization}
Let $X$ be a locally compact nondiscrete metric space and let $\I$ be an ideal on $\omega$. Then the following conditions are equivalent:

\begin{enumerate}[label={\rm (\roman{*})}]

\item $\I$ is $P^-$;\label{prop:P-like_for_spaces:Pminus-isolated-in-Gamma-are-in-Lambda-for-locally-compact:IFF:Pminus}

\item $\mathrm{iso}((\Gamma_{  x }(\I)) \subseteq \Lambda_{  x }(\I)$ for all sequences $  x $.
\label{prop:P-like_for_spaces:Pminus-isolated-in-Gamma-are-in-Lambda-for-locally-compact:IFF:isolated-points}

\end{enumerate}
\end{proposition}
\begin{proof}
\ref{prop:P-like_for_spaces:Pminus-isolated-in-Gamma-are-in-Lambda-for-locally-compact:IFF:Pminus} $\implies$ 
\ref{prop:P-like_for_spaces:Pminus-isolated-in-Gamma-are-in-Lambda-for-locally-compact:IFF:isolated-points}. 
Let $\eta$ be an isolated $\I$-cluster point of a sequence $  x $.
Then there exists $n_0\in \omega$ such that $\overline{B}(\eta,2^{-n_0})\cap \Gamma_{  x }(\I)=\{\eta\}$ and $\overline{B}(\eta,2^{-n_0})$ is compact (since $X$ is locally compact).
For each $n\geq n_0$, define
$$
A_n= \{i\in \omega: x_i\in \overline{B}(\eta,2^{-n})\}.
$$
Then $A_n\in\I^+$. 
Moreover, $A_n\setminus A_{n+1}\in \I$ for every $n\geq n_0$: 
indeed, suppose for the sake of contradiction that there is $n\ge n_0$ with $A_n\setminus A_{n+1}\in \I^+$.
Then $K= \overline{B}(\eta,2^{-n})\setminus B(\eta,2^{-n-1})$ is   compact and  
$\{i\in \omega:x_i\in K\} 
\supseteq 
A_n\setminus A_{n+1}\in\I^+$.
Now, we can use \cite[Theorem~6]{MR2923430} to obtain the existence of some $\I$-cluster point $\eta^\prime$ of $  x $ which belongs to $K$. By the definition of $K$, we have $\eta^\prime \neq \eta$ and $\eta^\prime \in K\subseteq \overline{B}(\eta,2^{-n})\subseteq \overline{B}(\eta,2^{-n_0})$. This is the required contradiction. 

At this point, since $\I$ is $P^-$, there exists $A\in \I^+$ such that $A\setminus A_n$ is finite for every $n\geq n_0$.
Hence the subsequence $(x_i: i \in A)$ is convergent to $\eta$, i.e., $\eta\in \Lambda_{  x }(I)$.

\ref{prop:P-like_for_spaces:Pminus-isolated-in-Gamma-are-in-Lambda-for-locally-compact:IFF:isolated-points} $\implies$ \ref{prop:P-like_for_spaces:Pminus-isolated-in-Gamma-are-in-Lambda-for-locally-compact:IFF:Pminus}. 
The proof goes on the same lines of the implication \ref{prop:P-like_for_spaces:Pbar-nonisolated-in-Gamma-are-in-Lambda:derived-set}
$\implies$ \ref{prop:P-like_for_spaces:Pbar-nonisolated-in-Gamma-are-in-Lambda:Pbar} of Proposition \ref{prop:Pprimecharacterization}, the unique difference being that each $y_n$ is \emph{not} an $\I$-cluster point of $  x $, while $\eta \in \Gamma_{  x }(\I)$ (since every neighborhood of $\eta$ contains $y_n$ for all large $n$). Hence $\eta$ is an isolated $\I$-cluster point, and the conclusion follows similarly. 
\end{proof}

Note that, in the proof above, the assumption of local compactness is used only in the implication \ref{prop:P-like_for_spaces:Pminus-isolated-in-Gamma-are-in-Lambda-for-locally-compact:IFF:Pminus} $\implies$ 
\ref{prop:P-like_for_spaces:Pminus-isolated-in-Gamma-are-in-Lambda-for-locally-compact:IFF:isolated-points}. However, it cannot be completely dropped from the statement as it is shown in the following example. 

\begin{example}
\label{example:local-compacness-Pminus}
Let $(r_n: n \in \omega)$ be a strictly decreasing sequence of positive reals with $\lim_n r_n=0$. Define $I_n=(r_{n+1},r_n)$ for every $n\in \omega$ and consider $X=\{0\}\cup \bigcup_{n\in\omega} I_n$ endowed with the Euclidean metric. Note that $X$ is not locally compact as witnessed by the point $0$. 

Pick an ideal $\I$ which is $P^-$ but not $P^{\,|}$ (for instance,~$\I=\Z$ by Example~\ref{example:P-like}\ref{example:P-like:Id}). Hence there exists a decreasing sequence $(A_n: n \in \omega)$ such that $A_0=\omega$, $A_n\setminus A_{n+1}\in \I^+$ for each $n$ and there is no $A\in \I^+$ with $A\setminus A_n$ finite for each $n$ (in the case $\I=\Z$, one can choose $A_n=\{2^nm: m \in \omega\}$ for each $n$). Note that $\bigcap_n A_n \in \I$.

For each $n$, pick a strictly decreasing sequence $x^{(n)}=(x^{(n)}_i: i\in A_n\setminus A_{n+1})$ with values in the interval $I_n$ which has the limit $r_{n+1}$ (note that $r_{n+1}\notin X$, hence this sequence is not convergent in $X$). Now define a sequence $  x =(x_i: i\in \omega)$ by $x_i=x^{(n)}_i$ for each $i\in A_n\setminus A_{n+1}$ and $n\in \omega$, and $x_i=0$ for each $i\in \bigcap_n A_n$. 

It follows that $\Gamma_{  x }(\I)=\{0\}$ and, in particular, $0$ is an isolated point of $\Gamma_{  x }(\I)$. To conclude, we show that $0\notin \Lambda_{  x }(\I)$.
Suppose that there exists $A\in \I^+$ such that $(x_i: i\in A)$ is convergent to $0$. Then for every $n$ we have $x_i\notin I_n$ for all but finitely many $i\in A$. Consequently, $A\setminus A_n$ is finite for every $n$, which is the claimed contradiction. 
\end{example}

As an immediate consequence of the previous results, we have the following:
\begin{corollary}\label{cor:analyticPgeneral}
    Let $X$ be a locally compact nondiscrete metric space and let $\I$ be a $\Pi^0_4$ ideal on $\omega$. Then $\mathrm{iso}((\Gamma_{  x }(\I)) \subseteq \Lambda_{  x }(\I)$ for all sequences $  x $.
\end{corollary}
\begin{proof}
It follows by Proposition \ref{prop:Pplus-vs-Pminus+Pvert}\ref{thm:Plike-properties-for-definable-ideals:Pi-zero-four-is-Pminus} 
and Proposition \ref{prop:Pminuscharacterization}.
\end{proof}

In particular, Corollary \ref{cor:analyticPgeneral} holds for analytic $P$-ideals (since the latter are $\Pi^0_3$ by Proposition \ref{prop:Pplus-vs-Pminus+Pvert}\ref{thm:Plike-properties-for-definable-ideals:analyticP}), hence recovering \cite[Theorem~2.8]{MR3883171}. 

\begin{proposition}
\label{prop:Ppluscharacterization}
Let $X$ be a nondiscrete metric space and let $\I$ be an ideal on $\omega$. Then the following conditions are equivalent:

\begin{enumerate}[label={\rm (\roman{*})}] 
\item  $\I$ is $P^+$;\label{prop:P-like_for_spaces:Pplus-Lambda-equals-Gamma:IFF:Pplus}

\item $\Gamma_{  x }(\I)=\Lambda_{  x }(\I)$ for all sequences $  x $.
\label{prop:P-like_for_spaces:Pplus-Lambda-equals-Gamma:IFF:Lambda-Gamma}

\end{enumerate}
\end{proposition}
\begin{proof}
    In \cite[Theorem~3.4]{MR4393937}, the authors prove the claim for $X=\R$. The proof in the general case goes verbatim (we omit details).  
\end{proof}

In particular, since $\Z$ is not a $P^+$-ideal, there exists a sequence $  x $ such that $\Gamma_{  x }(\I)\neq \Lambda_{  x }(\I)$. This recovers \cite[Example 3]{MR1181163} in the real case. 

\begin{corollary}\label{cor:sigma03pplus}
    Let $X$ be a nondiscrete metric space and let $\I$ be a $\Sigma^0_3$ ideal on $\omega$. Then $\Gamma_{  x }(\I)=\Lambda_{  x }(\I)$ for all sequences $  x $. 
\end{corollary}
\begin{proof}
    It follows by Proposition \ref{prop:Pplus-vs-Pminus+Pvert}\ref{thm:Plike-properties-for-definable-ideals:Fsigma}
    and Proposition \ref{prop:Ppluscharacterization}
\end{proof}

A direct proof of Corollary \ref{cor:sigma03pplus} for $\Sigma^0_2$ ideals can be found in \cite[Theorem~2.3]{MR3883171}.


\section{Ideal schemes}\label{sec:Ischemes}

\begin{definition}
Let $\I$ be an ideal on $\omega$. A family $\{A_s:s\in 2^{<\omega}\}$ of subsets of $\omega$ is called an \emph{$\I$-scheme} 
if the following properties hold for every $s\in 2^{<\omega}$:
\begin{enumerate}[label={\rm (\roman{*})}] 
    \item $A_s\in \I^+$;
    \item $A_{s^\frown 0}\cap A_{s^\frown 1}=\emptyset$; 
    \item $A_{s^\frown 0}\cup A_{s^\frown 1}\subseteq A_s$.
\end{enumerate}
In addition, it is said to be 
\emph{full} if $A_\emptyset=\omega$
and 
$A_{s^\frown 0}\cup A_{s^\frown 1} = A_s$ for every $s\in 2^{<\omega}$.
\end{definition}

\begin{example}
\label{example:full-Fin-scheme}
It is not difficult to see that full $\Fin$-schemes exist. For instance, one can define the family $\{A_s:s\in 2^{<\omega}\}$ as it follows: 
$A_\emptyset=\omega$ and
$$
A_s=\left\{n \in \omega: 2^{k+1} \,\text{ divides }\, n-\sum_{i=0}^k 2^is_i\right\}
$$
for each nonempty $s=(s_0,\ldots,s_k) \in 2^{<\omega}$ (hence $|s|=k+1)$. 
\end{example}

It is also immediate to see that, if $\I$ is a nowhere maximal ideal, then it admits an $\I$-scheme. 
We remark that full $\I$-schemes have been used in \cite[Theorem~1]{MR0480934} to characterize  
quotient Boolean algebras $\cP(\omega)/\I$ which carry 
a non-zero positive bounded non-atomic finitely additive measures on it, see also \cite[Theorem~5.3.2]{MR0751777}.

\begin{proposition}
\label{prop:Ischeme-characterizations}
Let $\I$ be an ideal on $\omega$. 
Then the following are equivalent:
\begin{enumerate}[label={\rm (\roman{*})}] 
    \item 
There exists a full $\I$-scheme;\label{prop:Ischeme-characterizations:full-Ischeme}

    \item 
There exists an $\I$-scheme;\label{prop:Ischeme-characterizations:Ischeme}

\item 
There exists an infinite  partition of $\omega$ into $\I$-positive sets;\label{prop:Ischeme-characterizations:partition}

\item There exists an infinite family of pairwise disjoint $\I$-positive sets;\label{prop:Ischeme-characterizations:pairwise-disjoint}

\item $\I$ is \emph{not} a Fubini sum of finitely many maximal ideals;\label{prop:Ischeme-characterizations:maximal-ideal}

\item There exists an infinite Hausdorff space $X$ such that 
$\mathscr{L}_{X}(\I)\neq [X]^{\le n}$ for all $n \in \omega$;\label{prop:fubinisummaximal}

\item For all infinite Hausdorff spaces $X$, it holds  
$\mathscr{L}_{X}(\I)\neq [X]^{\le n}$ for all $n \in \omega$.\label{prop:fubinisummaximal2}
\end{enumerate}
\end{proposition}

\begin{proof}
It is clear that 
\ref{prop:Ischeme-characterizations:full-Ischeme} $\iff$ \ref{prop:Ischeme-characterizations:Ischeme}
and
\ref{prop:Ischeme-characterizations:partition} $\iff$ \ref{prop:Ischeme-characterizations:pairwise-disjoint}.

\ref{prop:Ischeme-characterizations:Ischeme} $\implies$ \ref{prop:Ischeme-characterizations:pairwise-disjoint}. 
Let $\{A_s:s\in 2^{<\omega}\}\subseteq\I^+$ be an $\I$-scheme. Define $B_{n}=A_{(1^n)^\frown 0}$ for all $n\geq 1$. Then $\{B_n: n\geq 1\}$ is an infinite family of pairwise disjoint $\I$-positive sets.

\ref{prop:Ischeme-characterizations:partition} $\implies$ \ref{prop:Ischeme-characterizations:full-Ischeme}. 
Let $\{B_n:n\in \omega\}$ be an infinite partition of $\omega$ into $\I$-positive sets.
Pick a full $\Fin$-scheme $\{A_s:s\in 2^{<\omega}\}\subseteq\Fin^+$ (for instance,~the one given in Example~\ref{example:full-Fin-scheme}), and define $B_s=\bigcup_{n\in A_s}B_n$. Then $\{B_s: s\in 2^{<\omega}\}\subseteq\I^+$ is a full $\I$-scheme.

\ref{prop:Ischeme-characterizations:pairwise-disjoint} $\implies$ \ref{prop:Ischeme-characterizations:maximal-ideal}. 
Let $\I_0,\ldots,\I_{n-1}$ be maximal ideals for some $n\ge 1$ and suppose for the sake of contradiction that $\I=\I_0\oplus \cdots \oplus \I_{n-1}$, i.e., 
$$
\I=
\left\{A\subseteq n\times\omega: \{i\in\omega: (k,i)\in A\}\in\I_k \text{ for all }k \in n\right\}.
$$
Suppose also that there exist $n+1$ disjoint $\I$-positive sets $A_0,\dots,A_{n}$. 
Then there is $k_0\in n$ and distinct $m,m^\prime \in n+1$ such that 
$$
S=\{i\in\omega: (k_0,i)\in A_{m}\} \in \I_{k_0}^+
\quad \text{ and }\quad 
S^\prime=\{i\in\omega: (k_0,i)\in A_{m^\prime}\} \in \I_{k_0}^+.
$$
But $S\cap S^\prime=\emptyset$ and $\I_{k_0}$ is maximal, which is a contradiction.

\ref{prop:Ischeme-characterizations:maximal-ideal} $\implies$ \ref{prop:Ischeme-characterizations:pairwise-disjoint}. 
Note that, if we have finitely many pairwise disjoint $\I$-positive sets, then one of them can be splitted into two $\I$-positive sets (since $\I$ is not a Fubini sum of finitely many maximal ideals).
Hence, we can construct inductively an infinite sequence of pairwise disjoint $\I$-positive sets.

Lastly, the equivalences \ref{prop:Ischeme-characterizations:maximal-ideal} $\iff$ \ref{prop:fubinisummaximal} $\iff$ \ref{prop:fubinisummaximal2} 
follow by Proposition \ref{prop:fubinimaximal}.
\end{proof}

Further characterization of $\I$-schemes will be given later in Theorem \ref{prop:Gamma-vs-closed}. 

For each $\I$-scheme $\mathcal{A}$, it will be useful to define certain subsets of $2^\omega$: 
\begin{definition}\label{defi:BACA}
Given an ideal $\I$ on $\omega$ and an $\I$-scheme $\mathcal{A}=\{A_s:s\in 2^{<\omega}\}$, define 
$$
B_\I(\mathcal{A})=
\left\{x\in 2^\omega: \forall C\in\I^+ \, \exists  n\in\omega \,( C\setminus A_{x\restriction n}\in \mathrm{Fin}^+)\right\}
$$
and 
$$
C_\I(\mathcal{A})=
\left\{x\in 2^\omega: \forall C\in\I^+ \, \exists  n\in\omega \, (C\setminus A_{x\restriction n}\in \I^+)\right\}.
$$
\end{definition}

Of course, it follows by definition that $C_\I(\mathcal{A})\subseteq B_\I(\mathcal{A})$. 
The following result will be used repeatedly in the next sections.
\begin{lemma}
\label{lem:I-scheme-with-a-given-B}
Let $X$ be a nondiscrete metric space and $\I$ be an ideal on $\omega$.
\begin{enumerate}[label={\rm (\roman{*})}] 
    \item \label{item:1Ischeme}
 Let $H\subseteq 2^\omega$ be a set with empty interior.
 Suppose that there exist
 a set $P \in \mathscr{L}(\I)$
 and 
 a continuous injection 
 $f:2^\omega\to X$ 
 such that 
\begin{equation}\label{eq:inclusionIschemeequation}
  f[H]\subseteq X\setminus P
 \quad \text{ and } \quad 
 f[2^\omega\setminus H] \subseteq P.
\end{equation}
 Then there exists an $\I$-scheme $\cA$ such that $B_\I(\cA)=H$.

\item \label{item:2Ischeme} Suppose that there exist
 a set $P \in \mathscr{L}(\I)$ 
 and an embedding $f:2^\omega\to X$ such that 
 $f^{-1}[P]$ is dense in $2^\omega$. 
Then there exists an $\I$-scheme $\cA$ such that 
$$
f[B_\I(\cA)]=X\setminus P.
$$

\item \label{item:3Ischeme} Suppose that there exist
 a set $P \in \mathscr{L}(\I)$
 and a homeomorphic copy of the Cantor space $C\subseteq X$  such that 
 $C\cap P$ is dense in $C$. 
Then there exists an $\I$-scheme $\cA$ such that $B_\I(\cA)$ is homeomorphic to $C\setminus P$.

\end{enumerate}
\end{lemma}

\begin{proof}
\ref{item:1Ischeme} Let $H$, $f$, and $P$ be as in the statement. Then there is a sequence  
 $ y =(y_i)_{i\in \omega}$ with values in $X$ such that  $\Lambda_{ y }(\I)=P$. 
 Note that, for each $n\in \omega$, the family $\{f[[s]]: s\in 2^n\}$ consists of pairwise disjoint nonempty compact subsets of $X$. Thus there is $r_n>0$ such that the balls $B(f[[s]],r_n)$, for $s\in 2^n$, are pairwise disjoint.
Without loss of generality, we can assume that the sequence $(r_n: n \in \omega)$ is decreasing and convergent to zero. 
At this point, define 
$$
A_s=\{i\in\omega: y_i\in B(f[[s]], r_{|s|})\}.
$$
for each $s\in 2^{<\omega}$.
We claim that $\mathcal{A}=\{A_s: s \in 2^{<\omega}\}$ is an $\I$-scheme. For, it is easy to see that $A_{s^\frown 0}\cap A_{s^\frown 1}=\emptyset$ and  
$A_{s^\frown 0}\cup A_{s^\frown 1}\subseteq A_s$ for each $s \in 2^{<\omega}$. 
Now, fix $s\in 2^{<\omega}$ and let us check that $A_s\in \I^+$.  Since $H$ has empty interior, there exists $x^\star \in [s]\setminus H$. 
Since $x^\star \notin H$, it follows by \eqref{eq:inclusionIschemeequation} that $f(x^\star)\in P$, hence there is $D\in\I^+$ such that the subsequence $(y_i: i\in D)$ converges to $f(x^\star)$. 
Moreover, $x^\star \in [s]$ implies $f(x^\star) \in  B(f[[s]], r_{|s|})$. Since the latter set is open, we get $y_i\in  B(f[[s]], r_{|s|})$ for all sufficiently large $i\in D$. Thus $D\setminus F \subseteq A_s$ for some finite set $F\subseteq \omega$, which implies that $A_s\in \I^+$. 

To complete the proof, we need to show that $B_\I(\mathcal{A})=H$. With the same reasoning as above, if $x\in 2^\omega\setminus H$ then there exists $D_x \in \I^+$ such that for every $n\in \omega$ we have $y_i \in B(f(x),r_n)$ for all large $i \in D_x$. In particular, for every $n \in \omega$, we have 
$$
y_i\in B(f[[x\restriction n]],r_n)
$$
for all large $i \in D_x$. Therefore $x\notin B_\I(\mathcal{A})$, which proves that $B_\I(\mathcal{A})\subseteq H$. 

Conversely, suppose that $x\in 2^\omega \setminus B_\I(\mathcal{A})$. Then there is $D\in \I^+$ such that 
$D\setminus A_{x\restriction n}$ is finite for every $n\in \omega$.
Consequently, $y_i\in B(f[[x\restriction n]],r_{n})$    
for all but finitely many $i\in D$.
Since $\lim_n r_n=0$, we obtain that 
$$\bigcap_{n\in \omega} B(f[[x\restriction n]],r_{n}) = \{f(x)\},$$ 
hence $f(x)$ is the limit of the subsequence $(y_i: i\in D)$.
Thus $f(x) \in \Lambda_{ y }(\I) = P$. It follows by \eqref{eq:inclusionIschemeequation} that $x\notin  H$, which proves the converse inclusion $H\subseteq B_\I(\mathcal{A})$.

\ref{item:2Ischeme} Set $H=f^{-1}[X\setminus P]$. Since $f^{-1}[P]$ is dense, $H$ has empty interior. In addition, both inclusions \eqref{eq:inclusionIschemeequation} hold with equality. Hence the claim follows by item \ref{item:1Ischeme}.

\ref{item:3Ischeme} Set $H=f^{-1}[C\setminus P]$, where $f: 2^\omega \to C$ is a homeomorphism. Since $C$ is relatively dense in $C$, $H$ has empty interior. Moreover, $f[H]=C\setminus P$ and $f[2^\omega\setminus P]=C\cap P$, hence inclusions \eqref{eq:inclusionIschemeequation} hold. The claim follows by item \ref{item:1Ischeme}. 
\end{proof}

Let us define now four classes of ideals through the families $B_\I(\mathcal{A})$ and $C_\I(\mathcal{A})$ introduced in Definition \ref{defi:BACA}. Recall that $\Q(2^\omega)=\{x\in 2^\omega: \exists n_0\in \omega \, \forall n\ge n_0\, (x_n=0)\}$.
\begin{definition}\label{def:propertiesIschemesclasses}
Let $\I$ be an ideal on $\omega$. We write:
\begin{enumerate}[label={\rm (\roman{*})}] 
    \item $\I\in P(\Pi^0_1)$ if there is an $\I$-scheme $\mathcal{A}$ with $B_\I(\mathcal{A})=\emptyset$;
    \item $\I\in P(\Sigma^0_2)$ if there is an $\I$-scheme $\mathcal{A}$ with $B_\I(\mathcal{A})=\{0^\infty\}$;
    \item $\I\in P(\Pi^0_3)$ if there is an $\I$-scheme $\mathcal{A}$ with $B_\I(\mathcal{A})=\Q(2^\omega)$;
    
    \item $\I\in P^?(\Sigma^1_1)$ if there is an $\I$-scheme $\mathcal{A}$ with $B_\I(\mathcal{A})=C_\I(\mathcal{A})=\Q(2^\omega)$.
\end{enumerate}
\end{definition} 

We anticipate here that the reason for the notation $P^?(\Sigma^1_1)$ is that, if $\I$ belongs to this class and $X$ is sufficiently nice, then $\Sigma^1_1 \subseteq \mathscr{L}(\I)$, see Theorem \ref{thm:analytic} below; on the other hand, we do not know whether the converse holds, hence  we leave as an open question whether $\I \in P^?(\Sigma^1_1)$ is equivalent to the inclusion $\Sigma^1_1 \subseteq \mathscr{L}(\I)$.

We start below characterizing the first three classes of Definition \ref{def:propertiesIschemesclasses}, and proving that each class is contained into the other. (Hereafter, the addition of two sequences in $2^\omega$ or in $2^{<\omega}$ is meant coordinate-wise modulo $2$, provided they have the same lenght.)
\begin{proposition}
\label{prop:properties-of-P-like}
Let $\I$ be an ideal on $\omega$. Then:
\begin{enumerate}[label={\rm (\roman{*})}] 

\item $\I\in P(\Pi^0_1)$ if and only if there is a full  $\I$-scheme $\cA$ with $B_\I(\cA)=\emptyset$;
\label{prop:properties-of-B-empty-for-Ischeme:full}

\item $\I\in P(\Sigma^0_2)$ if and only if there is an $\I$-scheme $\mathcal{A}$ such that  $|B_\I(\mathcal{A})|=1$;\label{prop:properties-of-P-like:Sigma-2}

\item $\I\in P(\Pi^0_3)$ if and only if there is an $\I$-scheme $\mathcal{A}$ such that  $B_\I(\mathcal{A})$ is countable and dense; 
\label{prop:properties-of-P-like:Pi-3} 

\item $P^?(\Sigma^1_1)\subseteq P(\Pi^0_3)\subseteq P(\Sigma^0_2)\subseteq P(\Pi^0_1)$. 
\label{prop:properties-of-P-like:implications}
\end{enumerate}
\end{proposition}

\begin{proof}
\ref{prop:properties-of-B-empty-for-Ischeme:full} 
The \textsc{If part} is obvious, so let us show the \textsc{Only If part}.
Let 
$\mathcal{A}=\{A_s:s\in 2^{<\omega}\}$
be an $\I$-scheme with $B_{\I}(\cA)=\emptyset$.
Define $\cB=\{B_s:s\in 2^{<\omega}\}$ recursively as it follows: 
$B_\emptyset = \omega$, $B_s=A_s$
for every nonempty $s\in 2^{<\omega}$ ending with $1$, and 
$B_{s^\frown 0} = B_s\setminus B_{s^\frown 1}=B_s\setminus A_{s^\frown 1}$
for every $s\in 2^{<\omega}$. 
Then $\cB$ is a full $\I$-scheme and, in addition, $B_{\I}(\cB)=\emptyset$.

\ref{prop:properties-of-P-like:Sigma-2} 
The \textsc{Only If part} is obvious, so let us show the \textsc{If part}. 
For, we assume that $\cA=\{A_s:s\in 2^{<\omega}\}$ is an $\I$-scheme with $B_\I(\cA)=\{x^\star\}$, for some $x^\star \in 2^\omega$. Then define $B_s = A_{s+x^\star\restriction |s|}$ for each $s\in 2^{<\omega}$. Then $\cB=\{B_s:s\in 2^{<\omega}\}$ is an $\I$-scheme with $B_\I(\cB)=\{0^\infty\}$.

\ref{prop:properties-of-P-like:Pi-3}
The \textsc{Only If part} is obvious, so let us show the \textsc{If part}. Let $\mathcal{A}=\{A_s:s\in 2^{<\omega}\}$ be an $\I$-scheme with countable and dense $B_\I(\mathcal{A})$ and let $\{x^n:n\in\omega\}$ be an enumeration without repetitions of $B_\I(\mathcal{A})$. At this point, define the map $g:2^{<\omega}\setminus\{\emptyset\}\to 2^{<\omega}$ by 
$$
g(s)=(s_0,s_1,\ldots,s_{|s|-2},s_{|s|-1}+1)
$$
for each $s\in 2^{<\omega}\setminus\{\emptyset\}$, so that $g(s)$ is actually the sequence $s$ with the last digit changed. 
Also, for each $s\in 2^{<\omega}$, define recursively $t_s\in 2^{<\omega}$ so that: 
\begin{enumerate}[label=(\alph*)]
    \item $t_\emptyset=\emptyset$;
    \item $t_{\,0^n}=x^0\restriction n$ for all $n\geq 1$;
    \item if $s\in 2^{<\omega}\setminus\{\emptyset\}$ ends with $1$ (i.e., $s_{|s|-1}=1$) and $t_{g(s)}$ is already defined, let $k_s\in\omega$ be minimal integer such that $x^{k_s}\in [g(t_{g(s)})]$ and put 
    $$
    t_{s^\frown 0^n}=x^{k_s}\restriction (|s|+n)
    $$
\end{enumerate}
for all $n\in\omega$ (in particular, $t_s=x^{k_s}\restriction |s|$). 
It follows by construction that $\mathcal{B}=\{A_{t_s}:s\in 2^{<\omega}\}$ is an $\I$-scheme. In addition, it is not difficult to see that $B_\I(\mathcal{B})=\mathbb{Q}(2^\omega)$.

\ref{prop:properties-of-P-like:implications}
First, let us show the inclusion $P(\Sigma^0_2)\subseteq P(\Pi^0_1)$. For, pick an ideal $\I\in P(\Sigma^0_2)$ and an $\I$-scheme $\mathcal{A}=\{A_s:s\in 2^{<\omega}\}$ with $B_\I(\mathcal{A})=\{0^\infty\}$. Then $\mathcal{B}=\{A_{(1)^\frown t}: t \in 2^{<\omega}\}$ is an $\I$-scheme with $B_\I(\mathcal{B})=\emptyset$. Hence $\I\in P(\Pi^0_1)$. 

Second, let us prove the inclusion $P(\Pi^0_3)\subseteq P(\Sigma^0_2)$. 
Pick an ideal $\I \in P(\Pi^0_3)$ and let $\mathcal{A}=\{A_s:s\in 2^{<\omega}\}$ be an $\I$-scheme with $B_\I(\mathcal{A})=\mathbb{Q}(2^\omega)$.
For each sequence $s\in 2^{<\omega}$, define a sequence $t_s\in 2^{<\omega}$ so that: 
\begin{enumerate}[label=(\alph*)]
\item
$t_\emptyset=\emptyset$; 
\item $t_{0^n}=0^{2n}$ for all $n\ge 1$;
\item for each nonempty $s \in 2^{<\omega}\setminus \{0^n: n\ge 1\}$, set 
$$t_{s} = {t_{s\restriction (|s|-1)}}^\frown (s_{|s|-1},1).$$
\end{enumerate}
Note that $|t_s|=2|s|$ for all $s \in 2^{<\omega}$ and that if $s_i=1$ and $t_s=(a_0,a_1,\ldots,a_{k})$ then $a_j=1$ for all odd $j\in [i,k]$ (so that $k=2|s|-1$).  
At this point, define $C_s = A_{t_s}$ for each $s\in 2^{<\omega}$.
It is not difficult to see that $\cC=\{C_s:s\in 2^{<\omega}\}$ is an $\I$-scheme.
We need also to show $B_\I(\cC)=\{0^\infty\}$. 
Since $0^\infty\in B_\I(\cA)$ and $C_{0^n}=A_{0^{2n}}$ for each $n\ge 1$, we obtain that $0^\infty\in B_\I(\cC)$. 
Next, pick $x\in 2^\omega \setminus\{ 0^\infty\}$ and let $y\in 2^\omega$ be the unique sequence such that $y\restriction 2n= t_{x\restriction n}$ for all $n \in \omega$. 
Then $y \in 2^\omega\setminus\Q(2^\omega)$, so $y\notin B_\I(\cA)$.
Thus there is  $A\in \I^+$ such that $A\setminus A_{y\restriction n}$ is finite for each $n\in \omega$.
Since $C_{x\restriction n} = A_{t_{x\restriction n}} = A_{y\restriction 2n}$ for each $n\in \omega$, we obtain that $A\setminus C_{x\restriction n}$ is finite for each $n\in \omega$. Therefore $x\notin B_\I(\cC)$.

Lastly, the inclusion $P^?(\Sigma^1_1)\subseteq P(\Pi^0_3)$ is obvious. 
\end{proof}

In the following result, we study some properties of the class $P(\Pi^0_1)$:

\begin{proposition}
\label{prop:P-for-closed2}
Let $\I, \J$ be ideals on $\omega$. Then the following hold:
\begin{enumerate}[label={\rm (\roman{*})}] 
\item \label{item:1Ppi02} If $\I$ is a $P^+$-ideal and there exists an $\I$-scheme, then $\I \in P(\Pi^0_1)$;

\item \label{item:1Ppi07} If $\I \in P(\Pi^0_1)$ then there exists an $\I$-scheme;

\item \label{item:1Ppi03} If $\I\subseteq \J$ and $\J \in P(\Pi^0_1),$ then $\I \in P(\Pi^0_1)$;

\item \label{item:1Ppi04} If $\J\le_{\mathrm{RB}} \I$ and $\J \in P(\Pi^0_1),$ then $\I \in P(\Pi^0_1)$;

\item \label{item:1Ppi01} If $\I$ has the Baire property, then $\I \in P(\Pi^0_1)$;

\item \label{item:1Ppi05} If $\I$ is maximal, then the ideals $\J_1=\emptyset\otimes\I$ and $\J_2=\Fin\otimes \I$ are not in $P(\Pi^0_1)$, even though there are a $\J_1$-scheme and a $\J_2$-scheme;

\item \label{item:1Ppi06} Under the Continuum Hypothesis, there exists an ideal in $P(\Pi^0_1)$ without the Baire property.
\end{enumerate}
\end{proposition}
\begin{proof}
    \ref{item:1Ppi02} Suppose that $\I$ is a $P^+$-ideal and that there exists an $\I$-scheme $\mathcal{A}=\{A_s: s \in 2^{<\omega}\}$. Fix $x \in 2^\omega$ and note that $(A_{x\restriction n}: n \in\omega)$ is a decreasing sequence of $\I$-positive sets. Hence there exists $A \in \I^+$ such that $A\setminus A_{x\restriction n}$ is finite for all $n \in \omega$. Hence $x\notin B_\I(\mathcal{A})$. This implies that $B_\I(\mathcal{A})=\emptyset$, i.e., $\I \in P(\Pi^0_1)$.

    \ref{item:1Ppi07} This is obvious.

    \ref{item:1Ppi03} Let $\mathcal{A}=\{A_s: s \in 2^{<\omega}\}$ be a $\J$-scheme such that $B_\J(\mathcal{A})=\emptyset$. Since $\I\subseteq\J$, $\mathcal{A}$ is also an $\I$-scheme and $B_\I(\mathcal{A})\subseteq B_\J(\mathcal{A})=\emptyset$.

    \ref{item:1Ppi04} 
    Since $\J\leq_{RB}\I$, there is a finite-to-one function $f:\omega\to\omega$ such that $A\in \J$ if and only if $f^{-1}[A]\in \I$ for every $A\subseteq\omega$. Let $\cB=\{B_s:s\in 2^{<\omega}\}$ be a $\J$-scheme with $B_\J(\cB)=\emptyset$. Then it is not difficult to see that $\cA=\{f^{-1}[B_s] :s\in 2^{<\omega}\}$ is an $\I$-scheme with $B_\I(\cA)=\emptyset$.

    \ref{item:1Ppi01} Since $\mathrm{Fin}$ is a $P^+$-ideal which admits an $\I$-scheme (see Example \ref{example:full-Fin-scheme}), then $\mathrm{Fin} \in P(\Pi^0_1)$ by item \ref{item:1Ppi02}. Now, the Baire property of $\I$ is equivalent to $\mathrm{Fin}\le_{\mathrm{RB}} \I$ by Theorem \ref{thm:baire}. The claim follows by item \ref{item:1Ppi04}.

    \ref{item:1Ppi05} Define $A_n=\{n\}\times\omega$ for all $n \in \omega$ and note that $\{A_n: n \in \omega\}$ is a partition of $\omega^2$ into $\J_1$-positive sets, so there is an $\J_1$-scheme by Proposition~\ref{prop:Ischeme-characterizations}. 
    
    Now, let us show that $\J_1 \notin P(\Pi^1_0)$. Suppose for the sake of contradiction that there exists a $\J_1$-scheme $\cA=\{A_s:s\in 2^{<\omega}\}\subseteq \J_1^+$ such that $B_{\J_1}(\cA)=\emptyset$, and define 
    $$
    B_s=\{n \in \omega: \{k \in \omega: (n,k) \in A_s\} \in \I^+\}
    $$
  for each $ s \in 2^{<\omega}$.
  We claim that $\mathcal{B}=\{B_s: s \in 2^{<\omega}\}$ is a $\mathrm{Fin}$-scheme. Indeed, since $\I$ is maximal and $A_{s^\frown 0}\cap A_{s^\frown 1}=\emptyset$, we obtain $B_{s^\frown 0}\cap B_{s^\frown 1}=\emptyset$.
Since $A_s\in \J_1^+$, we obtain $B_s\neq\emptyset$ for each $s\in 2^{<\omega}$.
Moreover, since $B_t\subseteq B_s$ for all $t\subseteq s$, we obtain that each $B_s$ is infinite. 
Note that we can fix $x\in 2^\omega$ such that $\bigcap_n B_{x\restriction n}=\emptyset$: indeed, in the opposite, we could find continuum many elements in $\omega$. 

At this point, $B_{\J_1}(\cA)=\emptyset$ implies that there is $A\notin \J_1$ such that $A\setminus A_{x\restriction n}$ is finite for each $n\in \omega$.
Since $A\in \J_1^+$, there is $k_0\in \omega$ such that $S=\{n \in \omega: (k_0,n) \in A\}\in \I^+$.
Moreover, $\bigcap_n B_{x\restriction n}=\emptyset$ implies that there exists $n_0\in \omega$ such that $k_0\notin B_{x\restriction n_0}$, that is, 
$$
T=\{n \in \omega: (k_0,n) \in A_{x\restriction n_0}\}\in \I.
$$
It follows that $S\setminus T\in\I^+$. Therefore 
$A\setminus A_{x\restriction n_0}\in \J_1^+$; in particular, the latter set is infinite, which is the claimed contradiction.

As a consequence, we cover also the case of the ideal $\J_2$: indeed, $\J_1\subseteq \J_2$ and $\J_1 \notin P(\Pi^0_1)$, hence $\J_2 \notin P(\Pi^0_1)$ by item \ref{item:1Ppi03}; in addition, $\J_2$ is nowhere maximal (cf. Remark \ref{label:Finxmaximal}), hence it admits a $\J_2$-scheme by Proposition \ref{prop:Ischeme-characterizations}.

    \ref{item:1Ppi06} 
     Let  $\cA=\{A_s:s\in 2^{<\omega}\}$ be a $\Fin$-scheme (for instance, the one constructed in Example~\ref{example:full-Fin-scheme}). 
     Let also $\mathscr{F}$ be the family of finite-to-one functions $f : \omega\to \omega$. 
Considering that $\mathscr{F}\subseteq \omega^\omega$ and that $\mathscr{F}$ contains the subset of all functions $f$ with $f(n) \in \{n,n+1\}$ for all $n \in \omega$, 
        it follows that
    $|\mathscr{F}|=\mathfrak{c}$. 
Hence, we can pick enumerations $\{x_\alpha:\alpha<\continuum\}$ 
and
$\{f_\alpha:\alpha<\continuum\}$ of all elements of $2^\omega$
and $\mathscr{F}$, respectively.

We are going to construct by transfinite induction sets $\{A_\alpha: \alpha < \mathfrak{c}\}$ and $\{C_\alpha: \alpha < \mathfrak{c}\}$ such that the following conditions are satisfied for each $\alpha<\mathfrak{c}$:
\begin{enumerate}[label=(\alph*)]
   
    \item  \label{item:P1}
    $A_\alpha$ and $C_\alpha$ are infinite subsets of $\omega$;
    
    \item \label{item:P2}
    $A_\alpha\setminus A_{x_\alpha\restriction n}$ is finite for every $n\in \omega$;
    
    \item \label{item:P3} $\mathcal{I}_\alpha \cap \mathcal{A}_\alpha=\emptyset$, where $\I_\alpha$ stands for the smallest (admissible) ideal on $\omega$ containing $\{f_\beta^{-1}[C_\beta]: \beta<\alpha\}$ and $\cA_\alpha = \cA\cup \{A_\beta:\beta<\alpha\}$ for each $\alpha \le \mathfrak{c}$. 

\end{enumerate}

Assuming for the moment that the above sets are constructed, let us show that the ideal $\I_\continuum$ satisfies the required claim.

First, let us show that $\cA$ is an $\I_\continuum$-scheme. To this aim, since $\cA$ is a $\Fin$-scheme, we only need to show that $A_s\notin \I_\continuum$ for each $s\in 2^{<\omega}$.
Suppose for the sake of contradiction that $A_s\in \I_\continuum$ for some $s\in 2^{<\omega}$.
Then there exists a finite set $F\subseteq \continuum$ such that  $A_s\setminus \bigcup\{f^{-1}_{\beta}[C_{\beta}]:\beta\in F\}$ is finite.
It follows that $\max(F)<\mathfrak{c}$ and $A_s\in \I_{\max(F)+1}$, a contradiction with condition \ref{item:P3}. Notice that, with the same reasoning, we have that 
\begin{equation}\label{eq:contructionCH}
A_\alpha\notin \I_\continuum \text{ for each  $\alpha<\continuum$.}
\end{equation}

Second, let us show that $\I_\continuum \in P(\Pi^0_1)$. 
Fix $\alpha <\continuum$ and observe that $A_\alpha\notin\I_\continuum$ by \eqref{eq:contructionCH} and $A_\alpha\setminus A_{x_\alpha\restriction n}$ is finite for every $n\in \omega$ by condition \ref{item:P2}. Therefore $B_{\I_\continuum}(\cA)=\emptyset$.  

Third, let us show that $\I_\continuum$ does not have the Baire property. Thanks to Theorem \ref{thm:baire}, it is sufficient to show that for every function $f\in \mathscr{F}$ there is some infinite set $C\subseteq \omega$ such that $f^{-1}[C]\in\I_{\mathfrak{c}}$. For, pick $\alpha<\continuum$ and note that $C_\alpha$ is infinite by condition \ref{item:P1}, while $f_\alpha^{-1}[C_\alpha]\in \I_{\alpha+1}\subseteq\I_{\continuum}$ by condition \ref{item:P3}.

Lastly, we show how to construct sets $A_\alpha$ and $C_\alpha$ which satisfy conditions \ref{item:P1}--\ref{item:P3}.  
Suppose that $\alpha<\continuum$ and the sets $A_\beta$ and $C_\beta$ have been constructed for all $\beta<\alpha$. Under CH, it is possible to pick enumerations $\{B_i:i<\omega\}$ and $\{D_j:j<\omega\}$ of $\cA_\alpha$ and $\{f^{-1}_\beta[C_\beta]:\beta<\alpha\}\cup\Fin$, respectively.

Since $B_i\notin\I_\alpha$ by condition \ref{item:P3} and $D_j\in \I_\alpha$ for each $i,j<\omega$, the set 
$B_i\setminus \bigcup\{D_j:j<k\}$ is infinite for every $i,k<\omega$.
The function $f_\alpha$ is finite-to-one, so the set 
$f_\alpha[B_i\setminus \bigcup\{D_j:j<k\}]$ is infinite for every $i,k<\omega$ as well.
Thus, we can inductively pick pairwise distinct elements $y_{i,k}\in f_\alpha[B_i\setminus \bigcup\{D_j:j<k\}]$ and $z_{i,k}\in \omega$.

Define $C_\alpha=\{z_{i,k}:i,k<\omega\}$. 
Then $C_\alpha$ is infinite and we claim that 
$$\I_{\alpha+1}\cap \cA_\alpha=\emptyset.$$
Indeed, suppose for the sake of contradiction that $\I_{\alpha+1}\cap \cA_\alpha\neq \emptyset$.
Since $\I_{\alpha+1}$ is the smallest ideal containing $\{f^{-1}_\alpha[C_\alpha]\}\cup\{D_j:j<\omega\}$, there are $i,k<\omega$ such that 
$B_i\subseteq f^{-1}_\alpha[C_\alpha] \cup \bigcup\{D_j:j<k\}$.
Then 
$f_\alpha[B_i \setminus \bigcup\{D_j:j<k\}]\subseteq C_\alpha$. On the other hand,  $y_{i,k}\in f_\alpha[B_i \setminus \bigcup\{D_j:j<k\}]$ and $y_{i,k}\notin C_\alpha$ since $y_{i,k}\neq z_{j,l}$ for all $i,k,j,l\in \omega$, which is the required contradiction.

Now, we are going to define $A_\alpha$.
Since $\I_{\alpha+1}$ is countably generated (under CH), it is a $P^+$-ideal by Proposition~\ref{prop:Pplus-vs-Pminus+Pvert}.
Taking into account that $(A_{x_\alpha\restriction n}: n\in \omega)$ is a decreasing sequence of $\I_{\alpha+1}$-positive sets (because $\I_{\alpha+1}\cap \cA_\alpha=\emptyset$ and $\cA\subseteq \cA_\alpha$), there exists $A_\alpha\notin\I_{\alpha+1}$ such that $A_\alpha\setminus A_{x_\alpha\restriction n}$ is finite for every $n\in \omega$. 
It follows that $A_\alpha$ is infinite
and 
$\I_{\alpha+1}\cap \cA_{\alpha+1}=\emptyset$. This shows that $A_\alpha$ and $C_\alpha$ satisfy the required conditions \ref{item:P1}--\ref{item:P3}, completing the proof. 
\end{proof}

We summarize in Figure \ref{fig:P-closed} below the relationships related to the class $P(\Pi^0_1)$, Baire property, nowhere maximality, and existence of $\I$-schemes. Notice that all missing arrows are false: for instance, there exists an ideal $\I$ which admits an $\I$-scheme and it is not nowhere maximal (indeed, in the opposite, the Baire property implies $P(\Pi^0_1)$ which implies the existence of an $\I$-scheme which implies nowhere maximality, but we know that it cannot be true by Example \ref{example:Pplus-vs-Pminus+Pvert}\ref{item:MplusFin}).

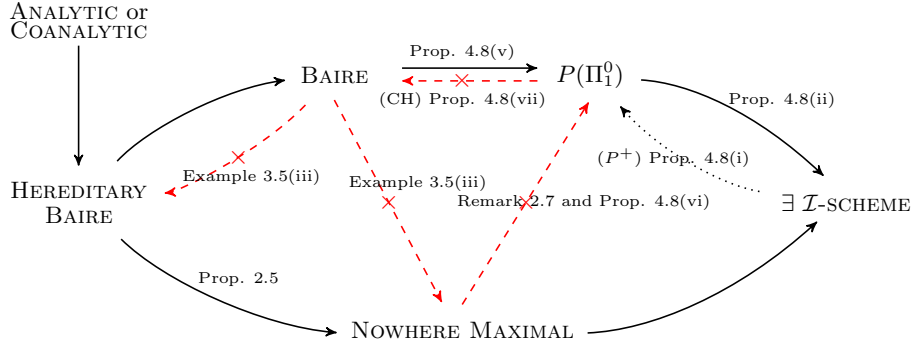
\begin{figure}
\centering
  \begin{tikzpicture}[scale=0.85,->, >=stealth',shorten >=1pt,auto,inner sep=2pt,semithick,bend angle=20]
    \tikzstyle{every state}=[draw=none]
    
    \node[state] (A) at (-6,0){};
    \node[state] (A1) at (-6,0.2){\small \textsc{Hereditary}};
    \node[state] (A2) at (-6,-0.2){\small \textsc{Baire}}; 

    \node[state] (B) at (-2,2){\small \textsc{Baire}};

    \node[state] (C) at (2,2){\small $P(\Pi^0_1)$};

    \node[state] (D) at (6,0){\small $\exists$ $\I$\textsc{-scheme}};

    \node[state] (E) at (0,-2){\small \textsc{Nowhere Maximal}};

    \path    (A) edge [bend left,looseness=.8, shorten <=3mm, shorten >=1mm]       node[above] {} (B);
    \path    (B) edge [bend left,looseness=.8, shorten <=0mm, shorten >=7mm, dashed, red]       node[above] {} (A);
    \node (X1) at (-3.5,.7){\textcolor{red}{$\times$}};
    \node (X11) at (-3.3,.4){\tiny Example \ref{example:Pplus-vs-Pminus+Pvert}\ref{item:MplusFin}};
       
    \path    (-2,2.1) edge [shorten <=9mm, shorten >=7mm]       node[above] {} (2,2.1); 
     \path    (2,1.9) edge [shorten <=7mm, shorten >=9mm, dashed, red]       node[below] 
     {}
     (-2,1.9);
      \node (X2) at (0,1.9){\textcolor{red}{$\times$}};
    \node (X22) at (0,1.6){\tiny (CH) Prop. \ref{prop:P-for-closed2}\ref{item:1Ppi06}};

     \node (X224) at (0,2.35){\tiny Prop. \ref{prop:P-for-closed2}\ref{item:1Ppi01}};

     \node (X223) at (3.3,.7){\tiny ($P^+$) Prop. \ref{prop:P-for-closed2}\ref{item:1Ppi02}};

      \node (X223) at (5,1.6){\tiny Prop. \ref{prop:P-for-closed2}\ref{item:1Ppi07}};

   \path    (A) edge [bend right, looseness=.8, shorten <=3mm, shorten >=1mm]       node[above] {} (E);
    
    \path    (0,-1.6) edge [bend left, looseness=0, dashed, red]       node[right] 
    {}
    (2,1.6);
     \node (X4) at (1,0){\textcolor{red}{$\times$}};
    \node (X44) at (1.9,0){\tiny Remark \ref{label:Finxmaximal} and Prop. \ref{prop:P-for-closed2}\ref{item:1Ppi05}};

\node (X2233) at (-3.5,-1.2){\tiny Prop. \ref{prop:BP-is-nowhere-maximal}};

    \path    (-2,1.6) edge [bend left, looseness=0, dashed, red]       node[above] 
    {}
    (-.3,-1.6) ;
     \node (X3) at (-1.15,0){\textcolor{red}{$\times$}};
    \node (X33) at (-.7,0.3){\tiny Example \ref{example:Pplus-vs-Pminus+Pvert}\ref{item:MplusFin}}; 

    \path    (C) edge [bend left,looseness=.8, shorten <=1mm, shorten >=-5mm]       node[above] {} (D);
    \path    (D) edge [bend left,looseness=.8, shorten <=2mm, shorten >=0mm, dotted]       node[above] {} (C);

    \path    (E) edge [bend right,looseness=.8, shorten <=1mm, shorten >=-5mm]       node[above] {} (D);
    
    \node[state] (W1) at (-6,3){\small \textsc{Analytic} or};
    \node[state] (W2) at (-6,2.7){\small \textsc{Coanalytic}};
    \path (W2) edge [shorten <=-8mm] node[above] {}(A);

  \end{tikzpicture}
  \caption{Relationships between Baire property, nowhere maximality, $P(\Pi^0_1)$, and $\I$-schemes. 
  Dashed arrows represent false implications.}
    \label{fig:P-closed}
\end{figure}

We conclude this section with additional properties of the remaining classes.
\begin{proposition}\label{prop:P^-}
    Let $\I$ be an ideal on $\omega$. Then the following hold:
\begin{enumerate}[label={\rm (\roman{*})}] 

    \item \label{item:1pminus}
    If $\I$ is a $P^+$-ideal, then $\I \notin P(\Sigma^0_2)$;

    \item \label{item:2pminus}  If $\I$ is a $P^+$-ideal with the Baire property, then $\I \in P(\Pi^0_1)\setminus P(\Sigma^0_2)$;

    \item \label{item:3pminus} 
    If $\I$ is a $P^-$-ideal, then $B_{\I}(\mathcal{A})=C_{\I}(\mathcal{A})$ for every $\I$-scheme $\mathcal{A}$;
    
    \item \label{item:4pminus} 
    If $\I$ is a $P^-$-ideal, then $\I \in P(\Pi^0_3)$ if and only if $\I \in P^?(\Sigma^1_1)$.
\end{enumerate}
\end{proposition}
\begin{proof}
    \ref{item:1pminus} Suppose that $\I$ is a $P^+$-ideal. If there is no $\I$-scheme, then claim is obvious. Otherwise, let $\mathcal{A}=\{A_s: s \in 2^{<\omega}\}$ be an $\I$-scheme. It follows from the proof of Proposition \ref{prop:P-for-closed2}\ref{item:1Ppi02} that $B_\I(\mathcal{A})=\emptyset$. Therefore $\I \notin P(\Sigma^0_2)$.

    \ref{item:2pminus} It follows by Proposition \ref{prop:P-for-closed2}\ref{item:1Ppi01} and item \ref{item:1pminus}.

    \ref{item:3pminus} Let $\I$ be a $P^-$-ideal and fix an $\I$-scheme $\mathcal{A}=\{A_s: s \in 2^{<\omega}\}$. It is obvious that $C_\I(\mathcal{A})\subseteq B_\I(\mathcal{A})$. 
    Conversely, fix $x \in  2^\omega\setminus C_\I(\mathcal{A})$. Then there is $C\in\I^+$ such that $C\setminus A_{x\restriction n}\in\I$ for all $n\in \omega$. Since $C\in\I^+$ and $C\setminus A_{x\restriction n}\in\I$, it follows that $(C\cap A_{x\restriction n}: n \in \omega)$ is decreasing sequence of $\I$-positive sets. In addition, $(C\cap A_{x\restriction n})\setminus(C\cap A_{x\restriction n+1})\subseteq C\setminus A_{x\restriction n+1}\in\I$ for all $n \in \omega$. Since $\I$ is $P^-$, there is $D\subseteq C$ such that $D\in\I^+$ and $D\setminus A_{x\restriction n}$ is finite for all $n\in \omega$. This implies that $x\notin B_\I(\mathcal{A})$. Therefore also the opposite inclusion $B_\I(\mathcal{A})\subseteq C_\I(\mathcal{A})$ holds.

    \ref{item:4pminus} On the one hand, $P^?(\Sigma^1_1)\subseteq P(\Pi^0_3)$ by Proposition \ref{prop:properties-of-P-like}\ref{prop:properties-of-P-like:implications}. On the other hand, let $\I$ be a $P^-$-ideal which belongs to $P(\Pi^0_3)$. Then then there exists an $\I$-scheme $\mathcal{A}$ such that $B_\I(\mathcal{A})=\mathbb{Q}(2^\omega)$. Hence $B_\I(\mathcal{A})=C_\I(\mathcal{A})=\mathbb{Q}(2^\omega)$ by item \ref{item:2pminus}, which implies that $\I \in P^?(\Sigma^1_1)$. 
\end{proof}

As a practical consequence of the previous results, we obtain: 
\begin{corollary}\label{cor:classesPborelcomplexity}
    Let $\I$ be an ideal on $\omega$. Then the following hold:
\begin{enumerate}[label={\rm (\roman{*})}] 
    \item \label{item:classesPborelcomplexity1} If $\I$ is $\Sigma^0_3$, then $\I \in P(\Pi^0_1)\setminus P(\Sigma^0_2)$;

    \item \label{item:classesPborelcomplexity2} If $\I$ is $\Pi^0_4$, then $\I \in P(\Pi^0_3)$ if and only if $\I \in P^?(\Sigma^1_1)$.
\end{enumerate}
\end{corollary}
\begin{proof}
\ref{item:classesPborelcomplexity1} It follows by Proposition \ref{prop:P^-}\ref{item:2pminus} 
and Proposition \ref{prop:Pplus-vs-Pminus+Pvert}\ref{thm:Plike-properties-for-definable-ideals:Fsigma}.

\ref{item:classesPborelcomplexity2} It follows by Proposition \ref{prop:P^-}\ref{item:4pminus} 
and Proposition \ref{prop:Pplus-vs-Pminus+Pvert}\ref{thm:Plike-properties-for-definable-ideals:Pi-zero-four-is-Pminus}. 
\end{proof}


\section{Case: Empty set}\label{sec:emptysetcase}

In this section, we study the existence of a sequence $  x $ with no $\I$-cluster points or with no $\I$-limit points. In fact, the former case is straightforward: 
\begin{proposition}\label{prop:emptysetcluster}
Let $X$ be a 
metric 
space and let $\I$ be an ideal on $\omega$. Then there exists a sequence $  x $ such that $\Gamma_{  x }(\I)=\emptyset$ if and only if $X$ is not compact. 
\end{proposition}
\begin{proof}
\textsc{Only If part.} See \cite[Theorem~6]{MR2923430}.

\textsc{If part.} Suppose that $X$ is not compact. Since $X$ is metric, it is not sequentially compact, hence there is a sequence $  x $ with values in $X$ without a convergent subsequence. Therefore $\Gamma_{  x }(\I)\subseteq \Gamma_{  x }(\fin)=\Lambda_{  x }(\fin)=\emptyset$, cf. \cite[Lemma 3.1]{MR3920799}. 
\end{proof}

Considering that 
every $\I$-limit point is always an $\I$-cluster point, 
we get immediately the following corollary in the case of noncompact spaces:
\begin{corollary}\label{cor:noncompactIlimit}
Let $X$ be a noncompact metric space and let $\I$ be an ideal on $\omega$. Then there exists a sequence $  x $ such that $\Lambda_{  x }(\I)=\emptyset$. 
\end{corollary}

Hence, the interesting cases for the study of the existence of a sequence $  x $ with no $\I$-limit points are the compact metric spaces. 

For, we will use the ideal 
$$
\conv = \{S\subseteq \Q\cap [0,1]: S \text{ has at most finitely many limit points}\}.
$$
Ideals $\I$ satisfying $\conv\leq_{K}\I$ have been extensively studied in \cite{alcantara-phd-thesis}. It is known that $\conv\leq_K\Fin^2$: in fact, there is a set $A\in\conv^+$ such that $\conv\restriction A$ is isomorphic to $\Fin^2$, see for instance \cite{MR3696069}.

We need to recall a weaker property than $P^-$: an ideal $\I$ on $\omega$ is a $P^-(\omega)$\emph{-ideal} if for every decreasing sequence $(A_n: n\in \omega)$ of $\I$-positive sets such that $A_0=\omega$ and $A_n\setminus A_{n+1}\in \I$ for every $n\in \omega$ there is $A\in \I^+$ for which $A\setminus A_n$ is finite for every $n\in \omega$. 
This property has been introduced in \cite[p.~2030]{MR3692233}. In \cite[Definition~1.2]{MR1367134}, see also \cite{MR3423409}, the author introduced the \emph{weak P} property, which coincides with $P^-(\omega)$, thanks to \cite[Proposition~6.1]{FKK23}. Moreover, the property $P^-(\omega)$ is equivalent to $\Fin^2\not\leq_K\I$ as was shown in \cite[Lemma 2]{MR2491780} and \cite[Example 4.1]{MR3034318}. 

\begin{proposition}
\label{thm:Gamma-Lambda-vs-emptyset-X}
Let $X$ be a nondiscrete compact metric space and let $\I$ be an ideal on $\omega$. Then the following hold:
\begin{enumerate}[label={\rm (\roman{*})}] 
\item \label{item:2emptyset} If $\mathrm{Fin}^2 \le_{\mathrm{K}} \I$ then $\Lambda_{  x }(\I)=\emptyset$ for some sequence $  x $;

\item \label{item:3emptyset} If $\Lambda_{  x }(\I)=\emptyset$ for some sequence $  x $ then $\mathrm{conv} \le_{\mathrm{K}} \I$;

\item \label{item:4emptyset} If $X$ is uncountable, $\mathrm{conv} \le_{\mathrm{K}} \I$ if and only if $\Lambda_{  x }(\I)=\emptyset$ for some sequence $  x $;

\item \label{item:5emptyset} If $\mathrm{conv} \le_{\mathrm{K}} \I$ then $\I$ is not $P^+$.
\end{enumerate}
\end{proposition}
\begin{proof}
\ref{item:2emptyset} Suppose that $\Fin^2\leq_K\I$. Then $\I$ is not $P^-(\omega)$ (by \cite[Lemma 2]{MR2491780} and \cite[Example 4.1]{MR3034318}), so there is a decreasing sequence $(A_n: n\in\omega)$ of $\I$-positive sets such that $A_0=\omega$, $A_n\setminus A_{n+1}\in \I$ for every $n\in\omega$, and for every $A\in \I^+$ there is $n\in \omega$ such that $A\setminus A_n$ is infinite.
Since $X$ is nondiscrete, there is an injective convergent sequence $ y =(y_n:n \in \omega)$ with values in $X$. 
Set $\eta=\lim_n y_n\in X$ and 
define the sequence $  x =(x_i: i \in \omega)$ as it follows: 
$x_i=y_n$ for all $i\in A_{n}\setminus A_{n+1}$ and $n\in \omega$, 
and
$x_i=\eta$ for all $i\in \bigcap_n A_n$. 
It follows that $A\in \I$ for every convergent subsequence $(x_i: i \in A)$. Therefore $\Lambda_{  x }(\I)=\emptyset$.

Before we proceed to the proof of the next items, we show that the following conditions are equivalent: 
\begin{enumerate}[label=(\alph*)]
\item \label{item:ACantor} $\mathrm{conv}\le_{\mathrm{K}}\I$;
\item \label{item:BCantor} There exists a sequence $  x $ with values in $[0,1]$ such that $\Lambda_{  x }(\I)=\emptyset$;
\item \label{item:CCantor} There exists a sequence $  x $ with values in $2^\omega$ such that $\Lambda_{  x }(\I)=\emptyset$. 
\end{enumerate}
(We remark that, in the literature, an ideal $\I$ which does \emph{not} satisfy condition \ref{item:BCantor} is said to have the $\mathrm{FinBW}$-property, see e.g. \cite{MR3034318, MR2320288} for a discussion and applications of this property.)

\ref{item:ACantor} $\Longleftrightarrow$ \ref{item:BCantor}. This has been proved in \cite[Section 2.7]{alcantara-phd-thesis}, see also \cite[Proposition~6.4]{MR3034318}.

\ref{item:BCantor} $\implies$ \ref{item:CCantor}. Let us suppose that there exists a sequence $  x $ taking values in $[0,1]$ such that $\Lambda_{  x }(\I)=\emptyset$. Since there exists a continuous surjection $g: 2^\omega\to [0,1]$, see e.g. \cite[Theorem~4.18]{MR1321597}, we conclude that $\Lambda_{ y }(\I)=\emptyset$, where $ y $ is an arbitrary sequence with values in $2^\omega$ such that $y_n \in g^{-1}[\{x_n\}]$ for every $n\in \omega$.

\ref{item:CCantor} $\implies$ \ref{item:BCantor}. Suppose that there exists a sequence $  x $ taking values in $2^\omega$ such that $\Lambda_{  x }(\I)=\emptyset$. Since the Cantor space $2^\omega$ and the Cantor ternary set $C$ are homeomorphic and $C$ is a closed subset of $[0,1]$, there exists an embedding $f:2^\omega\to [0,1]$. In other words, $[0,1]$ contains a homeomorphic copy of $2^\omega$. It follows that $ y =(f(x_n): n \in \omega)$ is a sequence taking values in the closed set $C$ such that $\Lambda_{ y }(\I)=\emptyset$.

\ref{item:3emptyset} Let $X$ be a nonempty compact metric space and suppose that there exists a sequence $  x $ taking values in $X$ such that $\Lambda_{  x }(\I)=\emptyset$. Since $X$ is a continuous image of the Cantor space $2^\omega$, see again \cite[Theorem~4.18]{MR1321597}, it follows as in the proof \ref{item:BCantor} $\implies$ \ref{item:CCantor} above that there exists a sequence $ y $ taking values in $2^\omega$ such that $\Lambda_{ y }(\I)=\emptyset$. Hence condition \ref{item:CCantor} holds. The claim follows by the equivalence \ref{item:ACantor} $\Longleftrightarrow$ \ref{item:CCantor}.

\ref{item:4emptyset} Let $X$ be an uncountable compact metric space. The \textsc{If part} follows by item \ref{item:3emptyset}. Conversely, for the \textsc{Only If part}, again by the equivalence \ref{item:ACantor} $\Longleftrightarrow$ \ref{item:CCantor}, there exists a sequence $  x $ taking values in $2^\omega$ such that $\Lambda_{  x }(\I)=\emptyset$. Since $X$ contains a homeomorphic copy of the Cantor space $2^\omega$, see e.g. \cite[Corollary~6.5]{MR1321597}, it follows as in the proof \ref{item:CCantor} $\implies$ \ref{item:BCantor} above that there exists a sequence $ y $ taking values in $X$ such that $\Lambda_{ y }(\I)=\emptyset$.

\ref{item:5emptyset} Suppose that $\I$ is a $P^+$-ideal. It follows by \cite[Corollary 5.6]{MR2961261} that condition \ref{item:BCantor} fails, cf. also \cite[Proposition 2.8]{MR2905404}. Thanks to the equivalence \ref{item:ACantor} $\Longleftrightarrow$ \ref{item:BCantor}, we conclude that $\conv\not\le_{\mathrm{K}} \I$. 
\end{proof}

A graphical representation of the implications stated in Proposition \ref{thm:Gamma-Lambda-vs-emptyset-X} 
is given in Figure \ref{fig:P-emptysetcase} below.

\begin{figure}
\centering
  \begin{tikzpicture}[scale=0.9,->, >=stealth',shorten >=1pt,auto,inner sep=2pt,semithick,bend angle=20]
    \tikzstyle{every state}=[draw=none]

    \node[state] (A1) at (-5.5,2.5){\small $\I$ \textsc{ is not } $P^-(\omega)$};
    
    \node[state] (A2) at (-5.5,0){\small $\mathrm{Fin}^2 \le_{\mathrm{K}} \I$};

    \node[state] (A3) at (-5.5,-2){\small $\I$ \textsc{ is not } $P^-$};

    \node[state] (B) at (0,1.5){};
 \node[state] (B1) at (0,1.5+0.2){\small $\Lambda_{  x }(\I)=\emptyset$ \textsc{ for}};
    \node[state] (B2) at (0,1.5-0.2){\small \textsc{ some sequence }$  x $};

    \node[state] (C1) at (5.5,2.5){\small $\I$ \textsc{ is not }$\mathrm{FinBW}$};

    \node[state] (C2) at (5.5,0){\small $\mathrm{conv} \le_{\mathrm{K}} \I$};

    \node[state] (C3) at (5.5,-2){\small $\I$ \textsc{is not} $P^+$};

    \path    (A2) edge [shorten <=0mm, shorten >=19mm]       node[above] {} (0,1.7);

    \path    (0,1.7+0.1) edge [shorten <=19mm, shorten >=13mm]       node[above] {} (5.5,0.1);
    \path    (5.5,-0.1) edge [shorten <=15mm, shorten >=19mm, dotted]       node[above] {} (0,1.7-0.1);

    \node[state] (Z1) at (3,.4){\tiny \rotatebox{343}{($X$ uncount.)}};
    
     \node[state] (Z2) at (2.9,.05){\tiny \rotatebox{343}{Prop. \ref{thm:Gamma-Lambda-vs-emptyset-X}\ref{item:4emptyset}}};

      \node[state] (Z2) at (3.1,1.1){\tiny \rotatebox{343}{Prop. \ref{thm:Gamma-Lambda-vs-emptyset-X}\ref{item:3emptyset}}};

         \node[state] (Z3) at (-3.1,1){\tiny \rotatebox{17}{Prop. \ref{thm:Gamma-Lambda-vs-emptyset-X}\ref{item:2emptyset}}};

\node[state] (Z4) at (4.6,1.6){\tiny \rotatebox{0}{\cite[Sect. 2.7]{alcantara-phd-thesis}}};

        \node[state] (Z4) at (4.55,-1){\tiny \rotatebox{0}{Prop. \ref{thm:Gamma-Lambda-vs-emptyset-X}\ref{item:5emptyset}}}; 

     \path    (5.5,2.5) edge [shorten <=4mm, shorten >=4mm, <->]       node[above] {} (5.5,0);

    \path    (-5.5,2.5) edge [shorten <=4mm, shorten >=4mm, <->]       node[above] {} (-5.5,0);
    \path    (-5.5,0) edge [shorten <=4mm, shorten >=4mm]       node[above] {} (-5.5,-2);

     \path    (5.5,0) edge [shorten <=4mm, shorten >=4mm]       node[above] {} (5.5,-2);

     \path    (-5.5,-2) edge [shorten <=14mm, shorten >=14mm]       node[above] {} (5.5,-2);

  \end{tikzpicture}
  \caption{Relationships between $P^+$, $P^-$, $P^-(\omega)$-ideals, and sequences with no $\I$-limit points in the case of a \emph{compact} metric space $X$.}
    \label{fig:P-emptysetcase}
\end{figure}
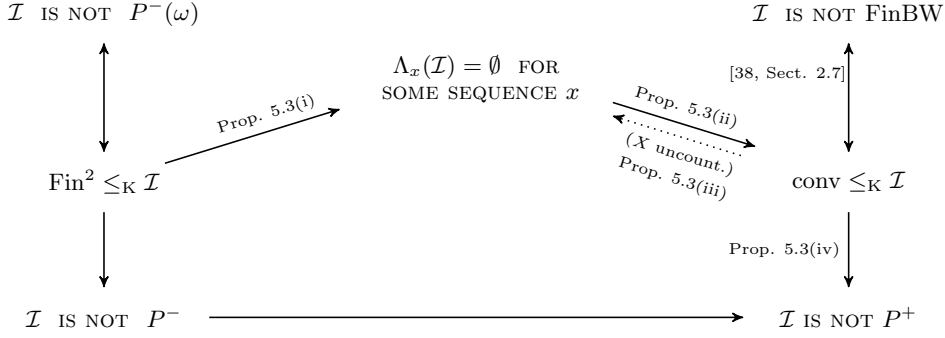


\section{Case: Closed sets}\label{sec:closedsets}

In this section we characterize the inclusions $\Pi^0_1\subseteq \mathscr{C}(\I)$ and $\Pi^0_1\subseteq \mathscr{L}(\I)$, namely, the cases where all closed sets can be realized as sets of $\I$-cluster points and $\I$-limit points, respectively.


\subsection{Closed sets Vs \texorpdfstring{$\I$}{I}-cluster points} 

Let us recall that all sets of $\I$-cluster points are closed, so that $\mathscr{C}(\I)\subseteq \Pi^0_1$ by Lemma \ref{lem:clusterclosed}. Of course, the converse inclusion does not always hold, as we already know by Proposition \ref{prop:fubinimaximal}. The following example shows that the same happens in the nonseparable case.
\begin{example}\label{example:Iclusters}
    Suppose that $X$ is a nonseparable topological space. Then, for every sequence $  x $ taking values in $X$ and for every ideal $\I$ on $\omega$, we have 
    \begin{displaymath}
    \Gamma_{  x }(\I)\subseteq \Gamma_{  x }(\mathrm{Fin})\subseteq \overline{\{x_n: n \in \omega\}}\neq X.
    \end{displaymath}
    In particular, $X\notin \mathscr{C}(\I)$. 
\end{example}

The above example shows that the equality $\mathscr{C}(\I)=\Pi^0_1$ can be realized only in separable spaces. In the following result, we characterize such property:
\begin{theorem}
\label{prop:Gamma-vs-closed}
Let $X$ be 
an infinite 
separable metric space and let $\I$ be an ideal on $\omega$. Then the following are equivalent:
\begin{enumerate}[label={\rm (\roman{*})}] 
\item \label{item:1IclusterCLOSED} There exists an $\I$-scheme;
\item \label{item:1IclusterCLOSEDB} $\Pi^0_1\subseteq \mathscr{C}(\I)$;
\item \label{item:2IclusterCLOSED} $\mathscr{C}(\I)=\Pi^0_1$;
\item \label{item:3IclusterCLOSED} $X \in \mathscr{C}(\I)$;
\item \label{item:4IclusterCLOSED} $\mathscr{C}(\I)$ contains an infinite subset of $X$;
\item \label{item:5IclusterCLOSED} $\mathscr{C}(\I)\neq [X]^{\le n}$ for all $n \in \omega$.
\end{enumerate}
\end{theorem}
\begin{proof}
\ref{item:1IclusterCLOSED} $\implies$ \ref{item:1IclusterCLOSEDB}. Thanks to Proposition~\ref{prop:Ischeme-characterizations}, there exists an infinite partition of $\omega$ into $\I$-positive sets. This implies, thanks to \cite[Theorems~4.1(ii)]{MR1844385}, that every nonempty closed subset can be realized as a set of $\I$-cluster points, i.e., $\Pi^0_1\subseteq \mathscr{C}(\I)$.

\ref{item:1IclusterCLOSEDB} $\implies$ \ref{item:2IclusterCLOSED}. It follows by Lemma \ref{lem:clusterclosed}.

\ref{item:2IclusterCLOSED} $\implies$ \ref{item:3IclusterCLOSED} $\implies$ \ref{item:4IclusterCLOSED} $\implies$ \ref{item:5IclusterCLOSED}. They are obvious.

\ref{item:5IclusterCLOSED} $\implies$ \ref{item:1IclusterCLOSED}. Let us suppose that $\mathscr{C}(\I)\neq [X]^{\le n}$ for all $n \in \omega$. Since $X$ is infinite, it follows by Proposition \ref{prop:fubinimaximal} that $\I\neq \I_0\oplus \cdots \oplus \I_{n-1}$ for all $n \in \omega$ and maximal ideals $\I_0,\ldots,\I_{n-1}$. Hence there is some $\I$-scheme, thanks to Proposition \ref{prop:Ischeme-characterizations}. 
\end{proof}


\subsection{Closed sets Vs \texorpdfstring{$\I$}{I}-limit points} 

Next, we characterize ideals $\I$ for which $\Pi^0_1 \subseteq \mathscr{L}(\I)$. Among several additional equivalent conditions, it turns out that the inclusion holds if and only if $\I \in P(\Pi^0_1)$. 

\begin{theorem}
\label{thm:closed-is-Lambda}
Let $X$ be an uncountable  Polish space and  $\I$ be an ideal on $\omega$. Then the following are equivalent:
\begin{enumerate}[label={\rm (\roman{*})}] 
    \item $\I\in P(\Pi^0_1)$;\label{thm:closed:ideal}\label{thm:closed-is-Lambda:ideal}

    \item $\Pi^0_1\subseteq \mathscr{L}(\I)$ ;\label{thm:closed:lambda}\label{thm:closed-is-Lambda:lambda}

    \item $\Delta^0_1\subseteq \mathscr{L}(\I)$;\label{thm:closed-is-Lambda:delta}

    \item $X\in \mathscr{L}(\I)$; \label{thm:closed-is-Lambda:X}
    
    \item $\mathscr{L}(\I)$ contains an uncountable analytic set;\label{thm:closed:analytic}\label{thm:closed-is-Lambda:analytic}

\item  \label{item:6pi01ilimitpoints}
There exists $P\in \mathscr{L}(\I)$ containing a homeomorphic copy of the Cantor space.
\end{enumerate}
\end{theorem}

Before we proceed to the proof of Theorem \ref{thm:closed-is-Lambda}, we prove the next two intermediate results. The proof of the first one is straightforward (hence, we omit it).
\begin{lemma}
\label{lem:I-scheme-on-Baire-tree}
    Let $\I$ be an ideal on $\omega$ which admits an $\I$-scheme $\cA=\{A_s:s\in 2^{<\omega}\}$ such that $B_{\I}(\cA)=\emptyset$. Define $\mathcal{B}=\{B_t: t\in \omega^{<\omega}\}$ by $B_\emptyset = A_\emptyset$ and    
    $B_t=A_{(0^{t_0},1,0^{t_1},1,\dots0^{t_{k}},1)}$
    for every nonempty $t=(t_0,\ldots,t_k)\in \omega^{<\omega}$. 
    Then the following hold:
    \begin{enumerate}[label={\rm (\roman{*})}] 
        \item $B_t\in \I^+$ for all $t\in \omega^{<\omega}$;
        \item $B_{t^\frown n}\subseteq B_t$ for all $t\in \omega^{<\omega}$ and $n\in \omega$;
        \item $B_{t^\frown n}\cap B_{t^\frown m}=\emptyset$ for all $t\in \omega^{<\omega}$ and distinct $n,m\in \omega$;
\item $\forall y\in \omega^\omega \, \exists x\in 2^\omega \, \forall n\in \omega \, \exists k_n\in \omega \, (B_{y\restriction n} = A_{x\restriction k_n})$;
    
    \item  $\forall y\in \omega^\omega \, \exists B\in \I^+ \, \forall n\in \omega \,\, (B\setminus B_{y\restriction n}\in \mathrm{Fin})$.
    \end{enumerate}
\end{lemma}

\begin{lemma}
\label{lem:closed}
Let $X$ be an uncountable Polish space and let $\I$ be an ideal on $\omega$ which admits an $\I$-scheme $\cA=\{A_s:s\in 2^{<\omega}\}$ such that $B_{\I}(\cA)=\emptyset$. 

Then, for every nonempty closed set $F\subseteq X$, there exists 
a sequence $  x =(x_n: n \in A_\emptyset)$ with values in $F$ such that $ \Lambda_{  x }(\I\restriction A_\emptyset)=F$.
\end{lemma}

\begin{proof}
Fix a nonempty closed set $F\subseteq X$. Then $F$ is an analytic subset of a Polish space, hence there exists a continuous surjection $f:\omega^\omega\to F$, see e.g.~\cite[p.~85]{MR1321597}. 
Let $\cB=\{B_t:t\in \omega^{<\omega}\}$ be the family defined as in Lemma~\ref{lem:I-scheme-on-Baire-tree} (recall that $A_\emptyset=B_\emptyset$). 

Proceeding by the induction  on the length of $t\in \omega^{<\omega}$, define 
the sequence $  x =(x_i: i\in B_\emptyset)$ with values in $X$ as it follows:
if 
$i = \min(B_t\setminus \{\min(B_{t\restriction k}): k<|t|\})$
or 
$i\in B_t\setminus (\{\min(B_{t\restriction k}): k<|t|\}\cup \bigcup\{B_{t^\frown k}:k<\omega\})$, set
$$x_i = f(t^\frown 0^\infty).$$
Note that, since for a fixed $t\in\omega^{<\omega}$ there are only finitely many sets $B_{t\restriction k}$ with $k<|t|$, then $x_i \in f[[t]]$ for all but finitely many $i\in B_t$. Hence, it follows by construction that
\begin{equation}\label{eq:conditionlemmaclosedsets}
    \forall t \in \omega^{<\omega} \, 
    \exists F_t \in \mathrm{Fin} \, 
    \forall i \in B_t\setminus F_t \, 
    \exists s_i \in \omega^{<\omega} \, 
    (x_i=f(t^\frown s_i^\frown 0^\infty)).
\end{equation}
To conclude the proof, we claim that $\Lambda_{  x }(\I\restriction A_\emptyset)=F$. 

Since the inclusion $\Lambda_{  x }(\I\restriction A_\emptyset)\subseteq F$ is obvious (as $x_i\in F$ for each $i\in B_\emptyset$ and $F$ is closed), we only need to show the converse inclusion. For, pick some $\eta \in F$. 
Since $f$ is surjective, there is $y\in \omega^\omega$ such that $f(y)=\eta$.
Thanks to Lemma~\ref{lem:I-scheme-on-Baire-tree}, 
there is $B\in \I^+$ 
such that $B\setminus B_{y\restriction n}$ is finite for every $n\in \omega$. Since $B\setminus B_\emptyset$ is finite, we can assume without loss of generality that $B\subseteq B_\emptyset$. 
We claim that $\eta$ is the limit of the subsequence $(x_i: i\in B)$. For, pick an open neighborhood $U$ of $\eta$. 
Since $f$ is continuous and $f(y)=\eta\in U$, there exists $n_U\in \omega$ such that $f[[t]]\subseteq U$, where $t=y\restriction n_U$. Pick $i \in B_{t}\setminus F_{t}$. 
Taking into account \eqref{eq:conditionlemmaclosedsets} and since $t^\frown s_i^\frown 0^\infty\in [t]$, we obtain $x_i = f(t^\frown s_i^\frown 0^\infty) \in f[[t]] \subseteq U$.
Since $B\setminus B_{t}$ is finite, it follows that $\{i\in B: x_i\notin U\}$ is finite, so that $\eta=\lim_{i \in B}{x_i}$. Therefore the opposite inclusion $F\subseteq \Lambda_{  x }(\I\restriction A_\emptyset)$ holds. 
\end{proof}

We proceed now to the proof of Theorem \ref{thm:closed-is-Lambda}. 
\begin{proof}
    [Proof of Theorem \ref{thm:closed-is-Lambda}] 
    \ref{thm:closed:ideal} $\implies$ \ref{thm:closed:lambda}.
Let $F$ be a nonempty closed subset of $X$ and let $\mathcal{A}=\{A_s: s \in 2^{<\omega}\}$ be an $\I$-scheme such that $B_\I(\mathcal{A})=\emptyset$. 
Thanks to Proposition \ref{prop:properties-of-P-like}\ref{prop:properties-of-B-empty-for-Ischeme:full}, we can assume without loss of generality that $\mathcal{A}$ is a full $\I$-scheme, hence $A_\emptyset=\omega$. 
It follows by Lemma~\ref{lem:closed} that there exists a sequence $  x =(x_n: n \in \omega)$ with values in $F$ such that $\Lambda_{  x }(\I)=F$. Therefore $\Pi^0_1\subseteq \mathscr{L}(\I)$.

\ref{thm:closed:lambda}
$\implies$
 \ref{thm:closed-is-Lambda:delta} $\implies$ \ref{thm:closed-is-Lambda:X} $\implies$ \ref{thm:closed:analytic}. 
They are obvious.

\ref{thm:closed:analytic} $\implies$ \ref{item:6pi01ilimitpoints}. 
It follows by the fact that every uncountable analytic set
 contains a homeomorphic copy of the Cantor space, see e.g.~\cite[Exercise~14.13]{MR1321597}.

\ref{item:6pi01ilimitpoints} $\implies$ \ref{thm:closed:ideal}. 
It follows by Lemma~\ref{lem:I-scheme-with-a-given-B}\ref{item:3Ischeme}.
\end{proof}

\begin{remark}
\label{rem:hBP->P(closed)}
Recall that ideals with the Baire property necessarily belong to $P(\Pi^0_1)$, thanks to Proposition \ref{prop:P-for-closed2}\ref{item:1Ppi01}. It follows by Theorem \ref{thm:closed-is-Lambda} that, if $\I$ has the Baire property, then $\Pi^0_1\subseteq \mathscr{L}(\I)$, recovering \cite[Corollary~3.9]{MR4393937} (see Theorem \ref{thm:xi}\ref{item:4xi}). 

However, our result is stronger since, under CH, there exists an ideal $\I\in P(\Pi^0_1)$ which does not have the Baire property by Proposition \ref{prop:P-for-closed2}\ref{item:1Ppi06}: hence, there exists an ideal $\I$ without the Baire property for which the inclusion $\Pi^0_1 \subseteq \mathscr{L}(\I)$ is still valid. 
\end{remark}

\begin{corollary}\label{cor:inclusionCeI}
    Let $X$ be an uncountable Polish space and let $\I$ be an ideal on $\omega$. Then the following hold\text{:}
    \begin{enumerate}[label={\rm (\roman{*})}] 

\item \label{item:1corCeI} If $\I \in P(\Pi^0_1)$, then $\Pi^0_1=\mathscr{C}(\I)\subseteq \mathscr{L}(\I)$;

\item \label{item:2corCeI} If $\I$ does not admit an $\I$-scheme, then $\mathscr{C}(\I)=\mathscr{L}(\I)=[X]^{\le n}$ for some $n \in \omega$. 
    \end{enumerate}
\end{corollary}
\begin{proof}
    \ref{item:1corCeI} If $\I \in P(\Pi^0_1)$ then there exists an $\I$-scheme by Proposition \ref{prop:P-for-closed2}\ref{item:1Ppi07}. 
    The claim follows by Theorem \ref{prop:Gamma-vs-closed} and Theorem \ref{thm:closed-is-Lambda}. 

    \ref{item:2corCeI} Since there are no $\I$-schemes, then $\I$ is a Fubini sum of finitely many maximal ideals, thanks to Proposition \ref{prop:Ischeme-characterizations}.    
    The conclusion follows by Proposition \ref{prop:fubinimaximal}.
\end{proof}

\begin{remark}
Due to Corollary \ref{cor:inclusionCeI}, one may be tempted to conjecture that $\mathscr{C}(\I)\subseteq \mathscr{L}(\I)$ for every ideal $\I$. However, in general, there is no inclusion between the families $\mathscr{C}(\I)$ and $\mathscr{L}(\I)$: indeed, set $X=[0,1]$ and let $\J$ be a maximal ideal on $\omega$. Then $\I=\fin\otimes \J$ is a nowhere maximal ideal, hence it admits an $\I$-scheme, cf. Remark \ref{label:Finxmaximal}. It follows by Theorem \ref{prop:Gamma-vs-closed} and Corollary \ref{cor:Lemptysetmaximal}\ref{item:2emptysetIfinI} that 
$$
\mathscr{C}(\I)=\Pi^0_1 \quad \text{ and }\quad \mathscr{L}(\I)=[X]^{\le \omega}. 
$$
In particular, $\mathscr{C}(\I)\not\subseteq \mathscr{L}(\I)$ and $\mathscr{L}(\I)\not\subseteq \mathscr{C}(\I)$. 

The same example shows, in addition, that we cannot replace ``uncountable'' by ``nonempty'' or even by ``infinite'' in item \ref{thm:closed-is-Lambda:analytic} of Theorem~\ref{thm:closed-is-Lambda}. 
\end{remark}

\begin{corollary}\label{cor:equalityLIwithclosedsets}
Let $X$ be an uncountable Polish space and 
$\I$ be an ideal on $\omega$. Then 
$$
\mathscr{L}(\I)=\Pi^0_1
$$ 
if and only if $\I$ is a $P^{\,|}$-ideal and $\I \in P(\Pi^0_1)$. 
\end{corollary}
\begin{proof}
It follows by Proposition \ref{prop:Pprimecharacterization} and Theorem \ref{thm:closed-is-Lambda}. 
\end{proof}

In Corollary \ref{cor:inclusionclosedsets}, in the case of ideals with the hereditary Baire property, we will give another characterization of $\mathscr{L}(\I)=\Pi^0_1$, which instead of $P^{\,|}$ uses the property $P(\Sigma^0_2)$ (which is more similar to $P(\Pi^0_1)$ and allows to extend this approach to the case $\mathscr{L}(\I)=\Sigma^0_2$).


\section{Case: \texorpdfstring{$F_\sigma$}{Fsigma}-sets}\label{sec:Fsigmasets}

In this section, we characterize ideals $\I$ for which $\mathscr{L}(\I)$ contains all $F_\sigma$-sets and their relationship with the family $P(\Sigma^0_2)$. 
\begin{theorem}
\label{thm:Fsigma-Lambda}\label{thm:F_sigma-are-Lambda}
Let $X$ be an uncountable  Polish space and let $\I$ be an ideal on $\omega$. Then the following are equivalent:
 \begin{enumerate}[label={\rm (\roman{*})}] 
    \item $\I\in P(\Sigma^0_2)$;\label{thm:Fsigma-Lambda:ideal}
   
    \item $\Sigma^0_2 \subseteq \mathscr{L}(\I)$;\label{thm:Fsigma-Lambda:lambda}
    
    \item $\Sigma^0_1 \subseteq \mathscr{L}(\I)$;\label{thm:Fsigma-Lambda:lambda-open}

    \item $\mathscr{L}(\I)$ contains an open set which is not the union of a closed set and a countable set;\label{thm:Fsigma-Lambda:opennotunion}
    
    \item $\mathscr{L}(\I)$ contains an analytic set which is not the union of a closed set and a countable set;\label{thm:Fsigma-Lambda:analytic}

\item There exists $P\in \mathscr{L}(\I)$ and a homeomorphic copy of the Cantor space $C\subseteq X$ such that $|C\setminus P|=1$.\label{thm:Fsigma-Lambda:injection}

\end{enumerate}

If, in addition, $\I$ has  the hereditary Baire property 
\textup{(}e.g.,~if $\I$ is Borel, analytic, or coanalytic\textup{)}, 
then the above conditions are also equivalent to:

 \begin{enumerate}[label={\rm (\roman{*})}] 
\setcounter{enumi}{6}
    \item $\mathscr{L}(\I)\not\subseteq \Pi^0_1$, that is, 
 $\mathscr{L}(\I)$ contains a set which is not closed;\label{thm:Fsigma-Lambda:for-hBP-ideals}

\item $\mathscr{L}(\I)\neq \Pi^0_1$;\label{thm:lastFsigma}
   
    \item \label{item:7HPB} $\I$ is not $P^{\,|}$;
    
    \item \label{item:8HPB} $\I$ is not $P^+$;
    
    \item \label{item:9HPB} $\Lambda_{  x }(\I)\neq\Gamma_{  x }(\I)$ for some sequence $  x $;

     \item \label{item:1HPB} $\emptyset\otimes\Fin\leq_{\mathrm{RB}}\I$.
    
\end{enumerate}
\end{theorem}

Before we proceed to the proof of Theorem \ref{thm:Fsigma-Lambda}, we need to show an auxiliary lemma. For, recall that a point $p$ in a topological space $X$ is called a \emph{condensation point} of a set $A \subseteq X$ if $U\cap (A\setminus \{p\})$ is uncountable for every neighborhood $U$ of $p$. The set of all condensation points of $A$ is denoted by $A^\star$. 

\begin{lemma}
\label{lem:notclosedunioncountable}
Let $X$ be an uncountable Polish space. Then there exists an open subset of $X$ which is not the union of a closed set and a countable set. 
\end{lemma}
\begin{proof}
Since $X$ is an uncountable Polish space, then the set $X^\star$ of all condensation points of $X$ is nonempty. 
In addition, thanks to Cantor-Bendixson theorem, see e.g.~\cite[Theorem~6.4]{MR1321597}, its complement $X\setminus X^\star$ is countable and open. 
Pick a point $\eta \in X^\star$ and define $U=X\setminus \{\eta\}$. In particular, $U$ is an uncountable open 
set and it does not contain $\eta$. 
Let us suppose for the sake of contradiction that there are a closed set $F$ and a countable set $C$ such that $U = F\cup C$. 
Since $\eta$ is a condensation point of $X$, we can pick $x_n\in B(\eta,2^{-n})\setminus(C\cup\{\eta\})$ for every $n\in \omega$. Then $x_n \in F$ for each $n\in \omega$ and $F$ is closed. Hence  
$\eta\in F\subseteq U$, which is the claimed contradiction. 
\end{proof}

\begin{remark}\label{rmk:idealemptysettimesmaximalplusfin}
Let $X$ be an uncountable Polish space. As it follows by Corollary \ref{cor:emptysettimesmaximalplusfin}, if $\I$ is a maximal ideal and $\J=(\emptyset\otimes \I)\oplus \mathrm{Fin}$, then 
$$
\mathscr{L}(\J)=\{A\cup B: A \in [X]^{\le \omega}, B \in \Pi^0_1\}. 
$$
Thanks to Lemma \ref{lem:notclosedunioncountable}, there exists an (uncountable) open set which is not the union of a closed set and a countable set. And, of course, there is a countable nonclosed set. 
Therefore $\Pi^0_1 \subsetneq \mathscr{L}(\J)\subsetneq \Sigma^0_2$. 
It follows by Theorem \ref{thm:closed-is-Lambda} and Theorem \ref{thm:Fsigma-Lambda} that $\J\in P(\Pi^0_1)\setminus P(\Sigma^0_2)$. 
Notice that $\J$ does not have the hereditary Baire property, 
see Proposition \ref{prop:P-for-closed2}\ref{item:1Ppi01}-\ref{item:1Ppi05}.

The above example shows that item \ref{thm:Fsigma-Lambda:opennotunion} cannot be weakened to the existence of an uncountable open set in $\mathscr{L}(\I)$ (in fact, $X \in \mathscr{L}(\I)$). It also shows that the equivalence with item \ref{thm:Fsigma-Lambda:for-hBP-ideals} 
fails without the additional hypothesis that the ideal is hereditary Baire. 
\end{remark}

We proceed now to the proof of Theorem \ref{thm:Fsigma-Lambda}.
\begin{proof}
    [Proof of Theorem \ref{thm:Fsigma-Lambda}]
    \ref{thm:Fsigma-Lambda:ideal} $\implies$ \ref{thm:Fsigma-Lambda:lambda}. 
Let $F$ be a nonempty $F_\sigma$-set and let $\{F_n: n \in \omega\}$ be nonempty closed sets such that $F = \bigcup_nF_n$. 

Since $\I \in P(\Sigma^0_2)$ there exists an $\I$-scheme $\cA=\{A_s:s\in 2^{<\omega}\}$ such that $B_{\I}(\cA)=\{0^\infty\}$. 
For every $n\in \omega$, the family $\cA_n = \{A^n_s:s\in 2^{<\omega}\}$ with  $A^n_s=A_{(0^n,1)^\frown s}$
is an $\I$-scheme with $B_{\I}(\cA_n)=\emptyset$. 
Thus, for every $n\in \omega$, it follows by Lemma~\ref{lem:closed} that there exists a sequence $ y ^n=(y^n_i: i\in A^n_\emptyset)$ with values in $F_n$ such that $\Lambda_{ y ^n}(\I\restriction A^n_\emptyset)=F_n$. At this point, pick an element $\eta_0 \in F$ and define the sequence $  x =(x_i: i \in \omega)$ by $x_i=y^n_i$ for all $i \in A^n_\emptyset$ and $n \in \omega$, and $x_n=\eta_0$ for all 
$i\in \omega \setminus \bigcup_n A^n_\emptyset$. 

To conclude the proof, it will be enough to show that $\Lambda_{  x }(\I)=F$. 
Since $A^n_\emptyset\in\I^+$, we have $F_n = \Lambda_{  x ^n}(\I\restriction A^n_\emptyset)\subseteq \Lambda_{\I}(x)$ for every $n\in \omega$, hence $F\subseteq \Lambda_{  x }(\I)$. Conversely, pick $\eta \in \Lambda_{  x }(\I)$ and $D\in \I^+$ such that $\eta$ is the limit of the subsequence $(x_n: n \in D)$. 
We have three cases:
\begin{enumerate}[label=(\alph*)]
    \item If $D\setminus \bigcup_n A^n_\emptyset$ is infinite, then $\eta=\eta_0\in F$;
\item If $D\cap A^n_\emptyset$ is infinite for some $n\in\omega$, then $\eta \in F_n\subseteq F$ since $F_n$ is closed;  
\item If $D \setminus \bigcup_n A^n_\emptyset$ is finite 
and 
$D\cap A^n_\emptyset$ is finite for every $n$,
then $D\setminus A_{0^n}$ is finite for every $n\in\omega$. However, this case is impossible by hypothesis $0^\infty\in B_{\I}(\cA)$. 
\end{enumerate}
Therefore also the opposite inclusion $\Lambda_{  x }(\I)\subseteq F$ holds.

\ref{thm:Fsigma-Lambda:lambda} $\implies$ \ref{thm:Fsigma-Lambda:lambda-open}. It is obvious.

\ref{thm:Fsigma-Lambda:lambda-open}
$\implies$
\ref{thm:Fsigma-Lambda:opennotunion}. 
It follows by Lemma \ref{lem:notclosedunioncountable}.

\ref{thm:Fsigma-Lambda:opennotunion}
$\implies$ 
\ref{thm:Fsigma-Lambda:analytic}. It is obvious.

\ref{thm:Fsigma-Lambda:analytic} $\implies$ \ref{thm:Fsigma-Lambda:injection}. 
Let $P\in \mathscr{L}(\I)$ be an analytic set which is not the union of a closed set and a countable set. 
Recall that $P^\star$ is perfect and $P\setminus P^\star$ is countable. It follows that 
$P\cap P^\star = P\setminus (P\setminus P^\star)$ is not closed. 
Hence, there exists an injective sequence $(p_n: n \in\omega)$ with values in $P\cap P^*$ which is convergent to some limit $\eta \notin P^\star\setminus P$. In particular, $\eta\notin P$.  
Also, let $(r_n: n\in\omega)$ be a sequence of positive reals which is convergent to zero such that the balls $B_n = B(p_n,r_n)$ are pairwise disjoint.
Since $p_n$ are condensation points, sets $P\cap B_n$ are uncountable, and consequently we can find a homeomorphic copy of the Cantor space $C_n\subseteq P\cap B_n$. Thus we can find a homeomorphism $f_n: [({0^{n}}, 1)]\to C_n$ for each $n$.
Define the map $f:2^\omega\to X$ by $f(0^\infty)=\eta$ and $f(x) = f_n(x)$ for each $x\in [({0^{n}},1)]$ and $n\in \omega$.
Then $f$ is a continuous injection 
such that $f(0^\infty)\notin P$ and $f[2^\omega\setminus\{0^\infty\}]\subseteq P$, 
so that $C=f[2^\omega]$. In addition, the inverse of $f$ is continuous thanks to \cite[Proposition 4.1(v)]{MR1617463}.

\ref{thm:Fsigma-Lambda:injection} $\implies$ \ref{thm:Fsigma-Lambda:ideal}. 
It follows from Lemma~\ref{lem:I-scheme-with-a-given-B}\ref{item:3Ischeme} and Proposition~\ref{prop:properties-of-P-like}\ref{prop:properties-of-P-like:Sigma-2}.

At this point, we know that items \ref{thm:Fsigma-Lambda:ideal}--\ref{thm:Fsigma-Lambda:injection} are equivalent. To conclude the proof, let us suppose hereafter that $\I$ has the hereditary Baire property.

\ref{thm:Fsigma-Lambda:analytic} $\implies$ \ref{thm:Fsigma-Lambda:for-hBP-ideals} $\implies$ \ref{thm:lastFsigma}. They are obvious.

\ref{thm:lastFsigma} $\implies$ \ref{thm:Fsigma-Lambda:for-hBP-ideals}. 
Since $\I$ has the hereditary Baire property, then $\I\in P(\Pi^0_1)$ by Proposition~\ref{prop:P-for-closed2}\ref{item:1Ppi01}. It follows by Theorem~\ref{thm:closed-is-Lambda} that $\Pi^0_1\subseteq \mathscr{L}(\I)$. By hypothesis $\mathscr{L}(\I)\neq \Pi^0_1$, therefore $\mathscr{L}(\I)\not\subseteq \Pi^0_1$.

\ref{thm:Fsigma-Lambda:for-hBP-ideals} $\implies$ \ref{thm:Fsigma-Lambda:ideal}. 
Assume that $\I$ has the hereditary Baire property and fix a set $F\in \mathscr{L}(\I)$ which is not closed. Let also $ y =(y_n: n\in\omega)$ be a sequence with values in  $X$ such that $F = \Lambda_{ y }(\I)$.

Since $F$ is not closed, there is an injective sequence $(p_n: n\in \omega)$ with values in $F$ which is convergent to some limit $\eta\notin F$. Let $(r_n: n\in\omega)$ be a sequence of positive reals which is convergent to zero such that the open balls $B_n = B(p_n,r_n)$ are pairwise disjoint.
For each $n\in \omega$, define
$A^n = \{i\in \omega: y_i\in B_n\}$.
Then the sets $A^n$ are pairwise disjoint since $B_n$ are so. In addition, each $A_n$ is $\I$-positive because $p_n$ is an $\I$-limit point of $ y $, hence also an $\I$-cluster point.  
Since $\I$ has the hereditary Baire property,  $\I\restriction A^n$ has the Baire property, hence $\I\restriction A^n \in P(\Pi^0_1)$ by Proposition~\ref{prop:P-for-closed2}\ref{item:1Ppi01}.    
Thanks to Proposition \ref{prop:properties-of-P-like}\ref{prop:properties-of-B-empty-for-Ischeme:full}, there exists a full $(\I\restriction A^n)$-scheme $\cA^n = \{A^n_s:s\in 2^{<\omega}\}$ such that $B_{\I\restriction A^n}(\cA^n)=\emptyset$. In particular, $A^n_\emptyset = A^n$ for each $n \in \omega$. 

Now we define a family $\cA=\{A_s:s\in 2^{<\omega}\}$ as it follows: 
\begin{enumerate}[label=(\alph*)]
\item $A_\emptyset = \bigcup_n A^n_\emptyset$;
\item $A_{0^k} = \bigcup_{n\ge k} A^n_\emptyset$ for every $k\geq 1$;
\item $A_{(0^k,1)^\frown s} = A^k_s$ for every $k\geq 0$ and $s\in 2^{<\omega}$.
\end{enumerate}
It follows by construction that $\cA$ is an $\I$-scheme. To complete the proof, it is sufficient to show that $B_\I(\cA)=\{0^\infty\}$. 
For, fix $x\in 2^\omega$. 
If $x\neq 0^\infty$, then there is the smallest $n\in\omega$ with $x_n=1$. Since $B_{\I\restriction A^n}(\cA^n)=\emptyset$, we obtain $x\notin B_\I(\cA)$.
Lastly, assume that $x=0^\infty$ and suppose for the sake of contradiction that 
$x\notin B_\I(\cA)$. Then 
there exists $D\in \I^+$ such that $D\setminus A_{0^n}$ is finite for every $n\in \omega$. 
Then the subsequence $(y_i: i\in D)$ has only finitely many values in each $B_n$, so 
it is convergent to $\eta$. 
Since $D\in\I^+$, we obtain $\eta \in \Lambda_{ y }(\I)=F$, which is the claimed contradiction.

\ref{thm:Fsigma-Lambda:for-hBP-ideals} 
$\iff$ \ref{item:7HPB}. 
See Proposition~\ref{prop:Pprimecharacterization}.

\ref{item:7HPB} $\iff$ \ref{item:8HPB}.
See Proposition~\ref{prop:Pplus-vs-Pminus+Pvert}\ref{prop:Pplus-vs-Pminus+Pvert:hBaire}.

\ref{item:8HPB} $\iff$ \ref{item:9HPB}.
See Proposition \ref{prop:Ppluscharacterization}.

\ref{item:7HPB} $\implies$ \ref{item:1HPB}. See Proposition~\ref{prop:Pplus-vs-Pminus+Pvert}\ref{prop:Pplus-vs-Pminus+Pvert:hBaire-Pvert}.

\ref{item:1HPB} $\implies$ \ref{thm:Fsigma-Lambda:lambda}. 
Thanks to Theorem \ref{thm:highBorelMSP}, we have $\mathscr{L}(\emptyset\otimes \mathrm{Fin})=\Sigma^0_2$. We conclude by Proposition \ref{prop:basic-prop-about-Lambda-Gamma}\ref{prop:RB-for-Lambda:Lambda-sets} that $\mathscr{L}(\emptyset\otimes \mathrm{Fin})\subseteq \mathscr{L}(\I)$. Cf. also \cite[Corollary 2.9]{MSP24}.
\end{proof}

\begin{corollary}\label{ref:cor1Fsigma}
$\I\in P(\Sigma^0_2)$ for every ideal with the hereditary Baire property which is not $P^+$. In particular, 
$\I\in P(\Sigma^0_2)$ for every analytic P-ideal $\I$ which is not $F_\sigma$. 
\end{corollary}

\begin{proof}
    It follows by Theorem~\ref{thm:Fsigma-Lambda}, 
    Theorem \ref{thm:oldthmballeo}\ref{item:BL5}, and Theorem \ref{thm:xi}\ref{item:5xi}.
\end{proof}


\section{Case: \texorpdfstring{$F_{\sigma\delta}$}{Fsigmadelta}-sets}\label{sec:Fsigmadeltasets}

On the same lines of Section \ref{sec:Fsigmasets}, here we characterize ideals $\I$ for which $\mathscr{L}(\I)$ contains all $F_{\sigma\delta}$-sets and their relationship with the family $P(\Pi^0_3)$. 

\begin{theorem}
\label{thm:F_sigmadelta-are-Lambda}
Let $X$ be an uncountable Polish space and let $\I$ be an ideal on $\omega$.
Then following are equivalent:
\begin{enumerate}[label={\rm (\roman{*})}] 
    \item $\I\in P(\Pi^0_3)$;\label{thm:F_sigmadelta-are-Lambda:P}
    \item $\Pi^0_3\subseteq\mathscr{L}(\I)$;\label{thm:F_sigmadelta-are-Lambda:Pi03}
    \item $\Pi^0_2\subseteq\mathscr{L}(\I)$;\label{thm:F_sigmadelta-are-Lambda:Pi02}

  \item $\mathscr{L}(\I)$ contains an analytic set which is not $\Sigma^0_2$;\label{thm:F_sigmadelta-are-Lambda:analytic-set}

\item There exists $P\in \mathscr{L}(\I)$ and a homeomorphic copy of the Cantor space $C\subseteq X$ such that $C\setminus P$ is countable dense in $C$.\label{thm:F_sigmadelta-are-Lambda:copy-of-Cantor}

\end{enumerate}
\end{theorem}

Before we proceed with the proof of Theorem \ref{thm:F_sigmadelta-are-Lambda}, we need to prove a representation of $F_{\sigma\delta}$-sets as superior limits of certain sequences. 
For, recall that a family $\mathscr{F}$ of subsets of a topological space  $X$ is 
said to be \emph{locally finite} if for every $x\in X$ there exists a neighborhood $U$ of $x$ such that the set $\{F\in \mathcal{F}: U\cap F\neq \emptyset\}$ is finite.

\begin{lemma}
\label{lem:decreasing-limsup}
Let $X$ be a separable metric space. 
Then for every $F_{\sigma\delta}$-set $P\subseteq X$ there exists a family $\{C_k:k\in \omega\}$ of closed sets such that:
\begin{enumerate}[label={\rm (\roman{*})}] 
\item \label{item1representationFsigmadelta} $P=\limsup_k C_k$;
\item \label{item2representationFsigmadelta} $\lim_k \mathrm{diam}(C_{n_0} \cap \cdots \cap C_{n_k})=0$ for every strictly increasing sequence of integers $(n_i)$ such that $\bigcap_i C_{n_i}\neq \emptyset$;
\item \label{item3representationFsigmadelta} $P \cap C_k \neq \emptyset$ for every $k \in \omega$.
\end{enumerate}
\end{lemma}

\begin{proof}
Let $P$ be an $F_{\sigma\delta}$-subset of $X$. 
Thanks to \cite[Exercise~23.5(i)]{MR1321597}, there is a family $\{P_n:n\in \omega\}$ of closed sets such that $P=\limsup_n P_n$. 

Since every metric space is paracompact, see e.g.~\cite[Theorem~5.1.3]{MR1039321}, 
and every metric separable space is Lindel\"{o}f, see e.g.~\cite[Theorem~3.8.1]{MR1039321}, for every $\varepsilon>0$ there is a locally finite open cover $\{A^\varepsilon_i:i\in \omega\}$ of $X$ such that $\mathrm{diam}(A^\varepsilon_i)<\varepsilon$ for every $i\in \omega$. Let $B^\varepsilon_i$ be the closure of $A^\varepsilon_i$ for every $\varepsilon$ and $i$. 
It follows that $\{B^\varepsilon_i:i\in \omega\}$ is a locally finite closed cover of $X$ such that $\mathrm{diam}(B^\varepsilon_i)<\varepsilon$ for every $i\in \omega$, see e.g.~\cite[Theorem~1.1.13]{MR1039321}. 

Let $(r_n)$ be a decreasing sequence of positive reals which is convergent to $0$. For each $n,i \in \omega$, define $D^n_i=P_n\cap B^{r_n}_i$.
Then the sets $D^n_i$ are closed, $P_n=\bigcup_i D^n_i$, and the family $\{D^n_i:i\in \omega\}$ is locally finite for every $n\in \omega$.
We claim that $\{D^n_i:n,i\in \omega\}$ is a family which satisfies conditions \ref{item1representationFsigmadelta} and \ref{item2representationFsigmadelta}.

\textsc{Condition \ref{item1representationFsigmadelta}.} Fix $x \in X$. We need to show that $x \in P$ if and only if there are infinitely many pairs $(n,i)\in \omega^2$ such that $x\in D^n_i$. First, suppose that $x\in P$. Then there are $n_0<n_1<\cdots$ such that $x\in P_{n_k}$ for every $k\in \omega$. Then for every $k$ there is $i_k\in \omega$ with $x\in D^{n_k}_{i_k}$. Conversely, suppose that there exist infinitely many pairs $(n_k,i_k)$ such that $x \in D^{n_k}_{i_k}$ for every $k \in \omega$, and define $N=\{n_k: k \in \omega\}$. If $N$ is finite, then there exists $n_\star \in N$ such that $x \in D^{n_\star}_i$ for infinitely many $i\in \omega$, but this would contradict the fact that $\{D^{n_\star}_i: i \in \omega\}$ is locally finite. Hence $N$ has to be infinite, so that there exists a sequence $((n_{k_j},i_{k_j}): j \in \omega)$ with $n_{k_j}<n_{k_{j+1}}$ for all $j \in \omega$ and $x \in D^{n_{k_j}}_{i_{k_j}}$
for all $j \in \omega$. Therefore $x \in P_n$ for infinitely many $n\in\omega$.

\textsc{Condition \ref{item2representationFsigmadelta}.} 
It is enough to show that for every injective sequence $((n_k,i_k): k \in \omega)$ such that $\bigcap_k D^{n_k}_{i_k}\neq \emptyset$, we have 
$$
\lim_{k\to \infty}\, \mathrm{diam}\left(D^{n_0}_{i_0}\cap \cdots \cap D^{n_k}_{i_k}\right)=0.
$$
We proceed similarly as in the previous step. If $N=\{n_k: k \in \omega\}$ is finite, then there exist $n_\star \in N$ and a strictly increasing sequence $({k_j}: 
j\in \omega)$ such that $(n_\star, i_{k_j})=(n_{k_j},i_{k_j})$. In such case, $\emptyset \neq \bigcap_k D^{n_k}_{i_k}\subseteq \bigcap_j D^{n_\star}_{i_{k_j}}$, which contradicts again the fact that $\{D^{n_\star}_i: i \in \omega\}$ is locally finite. Thus $N$ is infinite, and there exists a subsequence $((n_{k_j}, i_{k_j}): j \in \omega)$ such that $n_{k_j}<n_{k_{j+1}}$ for all $j\in\omega$. Therefore
$$
\mathrm{diam}\left(D^{n_{k_0}}_{i_{k_0}}\cap \cdots \cap D^{n_{k_m}}_{i_{k_m}}\right)\subseteq \mathrm{diam}\left(B^{r_{n_{k_0}}}_{i_{k_0}}\cap \cdots \cap B^{r_{n_{k_m}}}_{i_{k_m}}\right)\le r_{n_{k_m}}
$$
for all $m \in \omega$, which implies that $\{D^n_i:n,i\in \omega\}$ satisfies condition \ref{item2representationFsigmadelta}.

To complete the proof it is enough to let $(C_k: k \in \omega)$ be an enumeration of the family $\{D^n_i: P\cap D^n_i \neq \emptyset\}$. Indeed, it is straightforward to check that $\limsup_k C_k$ is precisely the set of all $x \in X$ such that $x \in D^n_i$ for infinitely many pairs $(n,i)$, so that $\limsup_k C_k=P$. Hence the sequence $(C_k: k \in \omega)$ satisfies condition \ref{item1representationFsigmadelta} holds. Lastly, it is clear that it satisfies also conditions \ref{item2representationFsigmadelta} and \ref{item3representationFsigmadelta}.
\end{proof}

Let us proceed now to the proof of Theorem \ref{thm:F_sigmadelta-are-Lambda}.
\begin{proof}
[Proof of Theorem \ref{thm:F_sigmadelta-are-Lambda}]
\ref{thm:F_sigmadelta-are-Lambda:P} $\implies$ \ref{thm:F_sigmadelta-are-Lambda:Pi03}. 
Let $\cA=\{A_s:s\in 2^{<\omega}\}$ be an $\I$-scheme with $B_\I(\cA)=\Q(2^\omega)$.
Fix a nonempty set $P\in \Pi^0_3$ and some $p_0\in P$.
Thanks to Lemma~\ref{lem:decreasing-limsup}, it is possible to pick a sequence $\{C_n:n\in \omega\}$ of closed subsets of $X$ which satisfies conditions \ref{item1representationFsigmadelta}--\ref{item3representationFsigmadelta} of Lemma~\ref{lem:decreasing-limsup}. In addition, for each finite sequence $s=(s_0,\ldots,s_m)\in 2^{<\omega}$ define the (possibly empty) set 
$$
R_s=\bigcap \,\{(C_k \cap P): k \in \{0,1,\ldots,m\} \text{ such that }s_k=1\}. 
$$
(here, the intersection of the empty family is defined to be the empty set). For every $s\in 2^{<\omega}$ with $R_s\neq\emptyset$ pick any $r_s\in R_s$.

Now, define $i_s\in\omega$ for all $s\in 2^{<\omega}$ proceeding by induction on the length of $s$: First, set $i_\emptyset = \min(A_\emptyset)$. Then, assume that $n\geq 1$ and $i_s$ have been defined for all $s$ with length smaller than $n$.
Then for each sequence $s\in 2^{<\omega}$ of length $n$, define 
$i_s = \min (A_s\setminus\{i_{s\restriction k}:k<n\})$. This completes the construction of $i_s$ for all $s\in 2^{<\omega}$. 
Note that each integer $j\in\omega$ either is equal to $i_s$ for some $s\in 2^{<\omega}$ or belongs to 
$$
S=(\omega\setminus A_\emptyset)\cup\bigcup \{D_s: s\in 2^{<\omega}\}, \quad \text{ where }\quad D_s=A_{s}\setminus (A_{s^\frown 0} \cup A_{s^\frown 1}).
$$
Indeed, if $j\notin S$, then there is $x\in 2^\omega$ such that $j\in A_{x\restriction n}$ for all $n\in\omega$. Hence, in such case, we would have $j= i_{x\restriction n}$ for some nonnegative integer $n\leq j$.

Now, let us define the sequence $ y =(y_i: i \in \omega)$ with values in $X$ as it follows: 
\begin{enumerate}[label=(\alph*)]
    \item $y_{i} = p_0$ if $i=i_\emptyset$ or $i \in \omega\setminus A_\emptyset$;
    \item $y_i=r_s$ if $i=i_s$ for some nonempty $s=(s_0,\ldots,s_m)\in 2^{<\omega}$ such that $s_m=1$ and $R_s\neq \emptyset$; 
    \item $y_i=y_{i_{s\restriction m}}$ if $i=i_s$ for some nonempty $s=(s_0,\ldots,s_m)\in 2^{<\omega}$ such that $s_m=0$ or $R_s= \emptyset$; 
    \item $y_i=y_{i_s}$ for all $i \in D_s\setminus \{i_t: |t|<|s|\}$ and $s \in 2^{<\omega}$.
\end{enumerate}
Note that $y_i \in P$ for all $i \in \omega$. It follows also by construction that if $s=(s_0,\ldots,s_m)\in 2^{<\omega}$ is nonempty and $s_m=1$ then $y_i \in C_m$ for all but finitely many $i\in A_s$.

\textsc{Claim: $P\subseteq \Lambda_{ y }(\I)$.} Fix $p\in P$. 
Then there exists a strictly increasing sequence $(n_k: k\in \omega)$ of positive integers such that $p \in \bigcap_{j} C_{n_j}$. 
In particular, the latter intersection is nonempty, hence the diameter of $\bigcap_{j\leq k}C_{n_j}$ goes to zero as $k\to \infty$ by Lemma~\ref{lem:decreasing-limsup}.
Define $x\in 2^\omega$ by 
$x_n=1$ if and only if $n=n_k$ for some $k\in \omega$. 
Since $x\notin \Q(2^\omega)=B_\I(\mathcal{A})$, there exists $A\in \I^+$ such that $A\setminus A_{x\restriction n}$ is finite for all $n\in \omega$.
By the above observations on the sequence $ y $, we obtain that  
$$
\{i \in A: y_i \notin C_{n_0} \cap \cdots \cap C_{n_k}\} \in \fin.
$$
for every $k \in \omega$.
Since the sets $C_{n_j}$ are closed and $\bigcap_{j} C_{n_j}=\{p\}$, it follows 
that the subsequence $(y_i: i\in A)$ is convergent to $p$. Therefore $p \in \Lambda_{ y }(\I)$

\textsc{Claim: $\Lambda_{ y }(\I)\subseteq P$.} Fix $p\in \Lambda_{ y }(\I)$. Hence there is $A\in \I^+$ such that the subsequence $(y_i: i\in A)$ is convergent to $p$. We need to show that $p \in P$. To this aim, we split the remaining proof into three cases.

First, suppose that $A\setminus A_\emptyset$ is infinite. Then $y_i=p_0$ for infinitely many $i\in A$. Since $(y_i: i\in A)$ is convergent to $p$, we obtain $p=p_0\in P$.

Second, suppose that there is $s\in 2^{<\omega}$ such that 
$A\cap D_s$ is infinite. Then $y_i = y_{i_{s}}$ for infinitely many $i\in A$. Since $(y_i: i\in A)$ is convergent to $p$, it follows that $p=y_{i_{s}}\in P$. 

Third, suppose that the above cases fail, that is, $A\setminus A_\emptyset$ is finite and $A\cap D_s$ is finite for all $s \in 2^{<\omega}$. Then there exists $x\in 2^\omega$ such that $B_n=A\cap A_{x\restriction n}\in \I^+$ for all $n\in \omega$, and define $N=\{n \in\omega: x_n=1\}$. 

If $N$ is infinite, or equivalently $x \notin \mathbb{Q}(2^\omega)$, it follows by the observations above on the sequence $ y $ that
$$
\{i \in B_{n+1}: y_i\notin C_n\} \in \mathrm{Fin}
$$
for every $ n \in N$.
Since each $C_n$ is closed and $(y_i: i\in B_{n+1})$ is a subsequence of the sequence $(y_i: i \in A)$ (which is convergent to $p$), we obtain that $p\in C_{n}$ for all $n\in N$. This implies that $p\in \limsup_n C_n=P$. 

Lastly, suppose that $N$ is finite, that is, $x \in \mathbb{Q}(2^\omega)$. Set $k=\max N$. 
Recall that $B_n\in \I^+$ and $B_n\setminus A_{x\restriction j}$ is empty for every $j\leq n$ and $n \in \omega$.
Since  $x\in B_\I(\cA)$, for each $n>k$ there exists an integer $m_n>n$ such that $B_n\cap A_{{x\restriction m_n}^\frown 1}$ is infinite. 
Without loss of generality we can assume that the sequence $(m_n)_{n>k}$ is strictly increasing. Similarly as above, we get 
$$
\{i \in B_n\cap A_{{x\restriction m_n}^\frown 1}: y_i\notin C_{m_n}\} \in \mathrm{Fin}
$$
for every $n>k$.
Since each $C_{m_n}$ is closed and $(y_i: i\in B_n\cap A_{{x\restriction m_n}^\frown 1})$ is a subsequence of the sequence $(y_i: i \in A)$, $p\in C_{m_n}$ for all $n>k$. It follows that $p\in \limsup_n C_n=P$.

Therefore $P=\Lambda_{ y }(\I)$, completing the proof of the implication.

\ref{thm:F_sigmadelta-are-Lambda:Pi03} 
$\implies$
\ref{thm:F_sigmadelta-are-Lambda:Pi02}. 
It is obvious.

\ref{thm:F_sigmadelta-are-Lambda:Pi02} $\implies$ \ref{thm:F_sigmadelta-are-Lambda:analytic-set}. 
This follows by the fact that 
there exists a set in 
$\Pi^0_2\setminus \Sigma^0_2
$ 
(which is, of course, analytic), 
see e.g. \cite[Theorem 22.4]{MR1321597}.

\ref{thm:F_sigmadelta-are-Lambda:analytic-set} $\implies$ \ref{thm:F_sigmadelta-are-Lambda:copy-of-Cantor}. 
Pick an analytic set  $P\in \mathscr{L}(\I)$ which is not $\Sigma^0_2$.
Thanks to
\cite[Theorem~21.18]{MR1321597}, there exists a homeomorphic copy of the Cantor space $C\subseteq X$ such that $C\setminus P$ is countable dense in $C$.

\ref{thm:F_sigmadelta-are-Lambda:copy-of-Cantor} $\implies$ \ref{thm:F_sigmadelta-are-Lambda:P}. 
Fix $P\in \mathscr{L}(\I)$ and let $C\subseteq X$ be a homeomorphic copy of the Cantor space such that $C\setminus P$ is countable dense in $C$. 
Since the Cantor space is countable dense homogeneous, we can find a homeomorphism $f:2^\omega\to C$ such that $f[\Q(2^\omega)]=C\setminus P$. 
Observe that the latter implies that $f^{-1}[C\cap P]$ is dense in $2^\omega$. 
We conclude by Lemma~\ref{lem:I-scheme-with-a-given-B}\ref{item:2Ischeme} that there exists an $\I$-scheme $\mathcal{A}$ such that $B_\I(\mathcal{A})=f^{-1}[C\setminus P]=\Q(2^\omega)$. 
\end{proof}

\begin{corollary}
Let $\I$ be an analytic $P$-ideal on $\omega$. Then $\I\notin P(\Pi^0_3)$. 
\end{corollary}

\begin{proof}
    Let $X$ be an uncountable Polish space. Thanks to \cite[Theorem 2.2]{MR3883171}, we have $\mathscr{L}(\I)\subseteq \Sigma^0_2$ (see also \cite[Theorem 1]{MR1838788} for the case $\I=\Z$). 
    Since there exists a set in $\Pi^0_2\setminus \Sigma^0_2$, see e.g. \cite[Theorem 22.4]{MR1321597}, the claim follows by Theorem \ref{thm:F_sigmadelta-are-Lambda}. 
\end{proof}

For the next remark, recall that an ideal $\I$ on $\omega$ is said to be \emph{tall} if every infinite set $S\subseteq \omega$ contains an infinite subset which belongs to $\I$. 
\begin{remark}
Let $X$ be an uncountable Polish space and let $\I$ be an ideal on $\omega$ with the hereditary Baire property.
Thanks to Theorem \ref{thm:baire}, Proposition \ref{prop:P-for-closed2}\ref{item:1Ppi01}, and Theorem \ref{thm:closed-is-Lambda},
we obtain that $\mathrm{Fin}\le_{\mathrm{RB}} \I$ and $\Pi^0_1\subseteq \mathscr{L}(\I)$. In particular, trivially, 
$$
\mathrm{Fin}\le_{\mathrm{RB}} \I \quad \text{ if and only if }\quad \Pi^0_1\subseteq \mathscr{L}(\I).
$$
Moreover, it follows by Theorem \ref{thm:Fsigma-Lambda} that 
$$
\emptyset \otimes \mathrm{Fin}\le_{\mathrm{RB}} \I \quad \text{ if and only if }\quad \Sigma^0_2\subseteq \mathscr{L}(\I).
$$
At this point, one may be tempted to conjecture that 
\begin{equation}\label{eq:failconjecture}
\fin^2\le_{\mathrm{RB}}\I \quad \text{ if and only if } \quad \Pi^0_3\subseteq \mathscr{L}(\I).
\end{equation}
Then the \textsc{Only If part} holds, thanks to Theorem \ref{thm:highBorelMSP} and Proposition \ref{prop:basic-prop-about-Lambda-Gamma}\ref{prop:RB-for-Lambda:Lambda-sets}. However, the \textsc{If part} of \eqref{eq:failconjecture} does \emph{not} hold: for, note that $\mathscr{L}(\fin^2\oplus \fin)=\Pi^0_3$ by Theorem \ref{thm:highBorelMSP} and Proposition \ref{prop:basic-prop-about-Lambda-Gamma}\ref{prop:basic-prop-about-Lambda-Gamma:direct-sum}. On the other hand, $\fin^2$ is tall and $\fin^2\oplus \fin$ is not tall. This implies that $\fin^2\not\le_{\mathrm{RB}} \fin^2\oplus \fin$, as it is easy to check that if $\I \le_{\mathrm{RB}} \J$ with $\I$ tall then also $\J$ has to be tall.
\end{remark}


\section{Case: Analytic sets}\label{sec:analyticsets}

In this section, we show that the condition $\I\in P^?(\Sigma^1_1)$ implies that all (nonempty) analytic sets can be realized as sets of $\I$-limit points. 
\begin{theorem}
\label{thm:analytic}
    Let  $X$ be an uncountable Polish space and let $\I\in P^?(\Sigma^1_1)$ be an ideal on $\omega$.
Then $\Sigma^1_1\subseteq\mathscr{L}(\I)$.
\end{theorem}

\begin{proof}
    Let $\cA=\{A_s:s\in 2^{<\omega}\}$ be an $\I$-scheme such that $B_\I(\cA)=C_\I(\cA)=\Q(2^\omega)$.
Let $P\subseteq X$ a nonempty analytic set and fix 
$p_\star\in P$. 
Thanks to \cite[Theorem~25.7]{MR1321597}, there exists a Souslin scheme $\{P_t:t\in \omega^{<\omega}\}$ of nonempty closed sets such that 
$$
P = \bigcup_{x\in \omega^\omega}\bigcap_{n\in \omega}P_{x\restriction n}.
$$
In addition, it can be assumed that $P_s\supseteq P_t$ whenever $s\subseteq t$ for all $t,s\in \omega^{<\omega}$ and $\lim_n \mathrm{diam}(P_{x\restriction n})=0$ for all $x \in \omega^\omega$. 
Note that, since $X$ is complete, for each $x\in \omega^\omega$ there exists $p_x\in X$ such that 
$\bigcap_{n} P_{x\restriction n} = \{p_x\}$. 
Hence $P=\{p_x:x\in \omega^\omega\}$.

Define the map $\phi:2^{<\omega}\to \omega^{<\omega}$ by 
$$
\phi(0^n)=\emptyset
\quad \text{ and }\quad 
\phi(0^{t_0},1,0^{t_1},\dots,0^{t_{k}},1,0^n)
=t
$$
for every $n\in \omega$ and nonempty $t=(t_0,\ldots,t_k)\in \omega^{<\omega}$. 
It follows by construction that $\phi(s)\subseteq \phi(s^\prime)$ and $P_{\phi(s)}\supseteq P_{\phi(s^\prime)}$ 
for every $s,s^\prime \in 2^{<\omega}$ with $s\subseteq s^\prime$.

At this point, for each $s \in 2^{<\omega}$ define $i_s \in \omega$ and $D_s$ exactly as in the proof of the implication \ref{thm:F_sigmadelta-are-Lambda:P} $\implies$ \ref{thm:F_sigmadelta-are-Lambda:Pi03} of Theorem \ref{thm:F_sigmadelta-are-Lambda}. 
Let us define the sequence $ y =(y_i: i \in \omega)$ with values in $X$ as it follows: 
\begin{enumerate}[label=(\alph*)]
     \item $y_i=p_{0^\infty}$ if $i=i_\emptyset$; 
    \item $y_{i} = p_\star$ for all $i \in \omega\setminus A_\emptyset$;
    \item $y_i=p_{\phi(s)^\frown 0^\infty}$ if $i=i_s$ for some nonempty $s=(s_0,\ldots,s_m)\in 2^{<\omega}$ with $s_m=1$; 
    \item $y_i=y_{i_{s\restriction m}}$ if $i=i_s$ for some nonempty $s=(s_0,\ldots,s_m)\in 2^{<\omega}$ with $s_m=0$; 
    \item $y_i=y_{i_s}$ for all $i \in D_s\setminus \{i_t: |t|<|s|\}$ and $s \in 2^{<\omega}$.
\end{enumerate}
Note that $y_i \in P$ for all $i \in \omega$. It follows also by construction that if $s=(s_0,\ldots,s_m)\in 2^{<\omega}$ is nonempty and $s_m=1$ then $y_i \in P_{\phi(s)}$ for all but finitely many $i\in A_s$.

\textsc{Claim: $P\subseteq \Lambda_{ y }(\I)$.} Fix $p\in P$. Let $x\in \omega^\omega$ be such that $p=p_x$ and define $t = (0^{x_0},1,0^{x_1},1,\ldots)$. Since $t\in 2^\omega\setminus \Q(2^\omega)$, there is $A\in \I^+$ such that $A\setminus A_{t\restriction n}$ is finite for every $n\in \omega$. It will be enough to show that the subsequence $(y_i: i \in A)$ is convergent to $p$. For, fix $\varepsilon>0$. Since $\lim_k \mathrm{diam}(P_{x\restriction k})=0$, there exists $k_\star \in \omega$ such that $p\in P_{x\restriction k}\subseteq B(p,\varepsilon)$ for every $k\geq k_\star$. 
In addition, for every $n\in \omega$ there exists $k_n\in \omega$ such that $\phi(t\restriction (n-1))=x\restriction k_n$. It is clear that $(k_n)$ is non-decreasing and divergent to $+\infty$ as $n\to\infty$. 
Take $n$ such that $k_n \geq k_\star$ and $t\restriction (n-1)$ ends with one.
By the observations above, the set  
$A\setminus A_{t\restriction n}$ is finite
and 
$y_i\in P_{\phi(t\restriction (n-1))}$ for all but finitely many $i\in A_{t\restriction n}$, hence $y_i\in P_{\phi(t\restriction (n-1))} = P_{x\restriction k_n} \subseteq B(p,\varepsilon)$ for all but finitely many $i\in A$. Therefore $\lim_{i \in A}y_i=p$, so that $p \in \Lambda_{ y }(\I)$.

\textsc{Claim: $\Lambda_{ y }(\I)\subseteq P$.} 
Fix $p\in \Lambda_{ y }(\I)$. Hence there is $A\in \I^+$ such that the subsequence $(y_i: i\in A)$ is convergent to $p$. We need to show that $p \in P$. If $A\setminus A_\emptyset$ is infinite or if $A \cap D_s$ is infinite for some $s \in 2^{<\omega}$, we proceed as in the proof of Theorem \ref{thm:F_sigmadelta-are-Lambda}. Otherwise, let us suppose hereafter that $A\setminus A_\emptyset$ is finite and $A\cap D_s$ is finite for all $s \in 2^{<\omega}$. Let us define a sequence $(B_n: n \in \omega)$ of $\I$-positive sets and a sequence $(k_n) \in \omega^\omega$ by induction as it follows: 
  \begin{enumerate}[label={\rm (\roman{*})}] 
  \item Set $B_0=A\cap A_\emptyset$, which belongs to $\I^+$ since $A\notin \I$ and $A\setminus A_\emptyset$ is finite; 
  \item Since $0^\infty \in C_\I(\cA)$, it is possible to fix $k_0\in \omega$ such that $B_1=B_0\cap A_{(0^{k_0},1)}\in \I^+$;
   \item Suppose that $B_0,\ldots,B_{n} \in \I^+$ and $k_0,\ldots,k_{n-1}\in \omega$ have been defined for some $n\in\omega$ and satisfy $B_{i+1}=B_{i}\cap A_{(0^{k_0},1,0^{k_1},1,\ldots,0^{k_{i}},1)}\in \I^+$ for all $i<n$. Since $(0^{k_0},1,0^{k_1},1,\ldots,0^{k_{n-1}},1,0^\infty)\in C_\I(\cA)$, it is possible to fix $k_{n} \in \omega$ such that 
   $$
   B_{n+1}=B_{n}\cap A_{(0^{k_0},1,0^{k_1},1,\ldots,0^{k_{n}},1)} \in \I^+.
   $$
  \end{enumerate}
Now, define $x= (0^{k_0},1,0^{k_1},1,\dots,0^{k_{n}},1,\ldots)$ and $m_n=k_0+\ldots+k_n+n+1\in \omega$ for all $n \in \omega$, so that $x\restriction m_n=(0^{k_0},1,0^{k_1},1,\dots,0^{k_{n}},1)$. Then $x\in 2^\omega\setminus \Q(2^\omega)$ and $A\cap A_{x\restriction m_n}\in \I^+$ for every $n\in \omega$.
It follows by construction that, for each $n \in \omega$, there are infinitely many $i \in A$ such that 
$$
y_i\in P_{\phi(x\restriction m_n)}=P_{(k_0,k_1,\ldots,k_n)}.
$$
Since the sets $P_n$ are closed and $p$ is the limit of the subsequence $(y_i: i\in A)$, we obtain that $p \in \bigcap_n P_{(k_0,k_1,\ldots,k_n)}= \{p_{(k_0,k_1,\ldots)}\}$, hence $p=p_{(k_0,k_1,\dots)}\in P$.

Therefore $P=\Lambda_{ y }(\I)$, which finishes the proof. 
\end{proof}

We end this section with our first open question: 
\begin{question}
Let  $X$ be  an uncountable Polish space and let $\I$ be an ideal on $\omega$.
Is it true that $\I\in P^?(\Sigma^1_1)$ if and only if $\Sigma^1_1\subseteq\mathscr{L}(\I)$?
\end{question}


\section{Applications and Open Questions}\label{sec:applications}

In this section, we collect some consequences of our results, and we conclude with some open questions. 

\begin{remark}\label{rem:alltopspaces}
Note that the properties $P(\Pi^0_1)$, $P(\Sigma^0_2)$, $P(\Pi^0_3)$ and $P^?(\Sigma^1_1)$ are purely combinatorial. Hence, for instance:
\begin{enumerate}[label={\rm (\roman{*})}]    
   \item it follows from Theorem \ref{thm:closed-is-Lambda} that if there is some uncountable Polish space $X$ such that $\mathscr{L}_X(\I)$ contains an uncountable analytic set, then $\Pi^0_1(Y)\subseteq\mathscr{L}_Y(\I)$ for every uncountable Polish space $Y$;
   \item it follows from Theorem \ref{thm:F_sigma-are-Lambda} that if there is some uncountable Polish space $X$ such that $\mathscr{L}_X(\I)$ contains an analytic set which is not the union of a closed set and a countable set, then $\Sigma^0_2(Y)\subseteq\mathscr{L}_Y(\I)$ for every uncountable Polish space $Y$;
   \item it follows from Theorem \ref{thm:F_sigma-are-Lambda} that if $\I$ has the hereditary Baire property and there is some uncountable Polish space $X$ such that $\mathscr{L}_X(\I)$ contains a set which is not closed, then $\Sigma^0_2(Y)\subseteq\mathscr{L}_Y(\I)$ for every uncountable Polish space $Y$;
   \item it follows from Theorem \ref{thm:F_sigmadelta-are-Lambda} that if  there is some uncountable Polish space $X$ such that $\mathscr{L}_X(\I)$ contains an analytic set which is not $\Sigma^0_2(X)$, then $\Pi^0_3(Y)\subseteq\mathscr{L}_Y(\I)$ for every uncountable Polish space $Y$.
   \end{enumerate}
\end{remark}


\subsection{Which families can be obtained as \texorpdfstring{$\mathscr{L}(\I)$}{L(I)}?}

\begin{corollary}\label{cor:nonopenLI}\label{cor:LInonGdeltasigma}\label{cor:LInonGdelta}

Let $X$ be an uncountable Polish space and let $\I$ be an ideal on $\omega$. Then the following hold:
   \begin{enumerate}[label={\rm (\roman{*})}]    
   \item \label{item:nonopenLI} $\mathscr{L}(\I)\not\subseteq \Sigma^0_1$. In particular, $\mathscr{L}(\I)\neq \Sigma^0_1$;
   \item \label{item::LInonGdelta} $\Sigma^0_1\subseteq \mathscr{L}(\I)$ if and only if $\Sigma^0_2\subseteq \mathscr{L}(\I)$. In particular, $\mathscr{L}(\I)\neq \Pi^0_2$;
   \item \label{item::LInonGdeltasigma} $\Pi^0_2\subseteq \mathscr{L}(\I)$ if and only if $\Pi^0_3\subseteq \mathscr{L}(\I)$. In particular, $\mathscr{L}(\I)\neq \Sigma^0_3$.
   \end{enumerate}
\end{corollary}
\begin{proof}
\ref{item:nonopenLI} Pick an accumulation point $\eta \in X^{\,|}$. Thus, $\{\eta\}$ is not open. It follows by Lemma \ref{lem:Xlen} that $\{\eta\} \in \mathscr{L}(\I)$. Therefore $\mathscr{L}(\I)\not\subseteq \Sigma^0_1$.

\ref{item::LInonGdelta} The \textsc{If part} is clear. To show the \textsc{Only If part}, let us suppose that $\Sigma^0_1\subseteq \mathscr{L}(\I)$.  Thanks to Theorem \ref{thm:Fsigma-Lambda}, we get also $\Sigma^0_2\subseteq \mathscr{L}(\I)$. 
The last sentence follows by the facts that $\Sigma^0_1\subseteq \Pi^0_2$ and $\Sigma^0_2\setminus \Pi^0_2\neq \emptyset$, see e.g. \cite[Theorem 22.4]{MR1321597}.

\ref{item::LInonGdeltasigma} The proof goes 
as in item \ref{item::LInonGdelta}, replacing Theorem \ref{thm:Fsigma-Lambda} with Theorem \ref{thm:F_sigmadelta-are-Lambda}. 
\end{proof}

\begin{remark}
We remark that the proof of Corollary \ref{cor:nonopenLI}\ref{item:nonopenLI} actually works in every nondiscrete Hausdorff space.
\end{remark}

On the one hand, thanks to Theorem \ref{thm:highBorelMSP}, 
all families of the type $\Pi^0_{2\alpha-1}$ and $\Sigma^0_{2\alpha}$, with $\alpha$ positive integer, can be realized as some $\mathscr{L}(\I)$. On the other hand, it follows by Corollary \ref{cor:nonopenLI} that there are no ideals $\I$ on $\omega$ such that $\mathscr{L}(\I)=\Sigma^0_1$ or $\mathscr{L}(\I)=\Pi^0_2$ or $\mathscr{L}(\I)=\Sigma^0_3$. This leads to the next open question:
\begin{question}
    Let $X$ be an uncountable Polish space. Do there exist an ideal $\I$ on $\omega$ and a positive integer $\alpha$ such that 
    $$
    \mathscr{L}(\I)=\Sigma^0_{2\alpha-1}
    \quad \text{ or }\quad 
    \mathscr{L}(\I)=\Pi^0_{2\alpha}\,\,?
    $$
\end{question}


\subsection{Coanalytic ideals}

Recall that by Corollary \ref{cor:equalityLIwithclosedsets}, $\mathscr{L}(\I)=\Pi^0_1$ is equivalent to $\I\in P^{\,|}$ and $\I \in P(\Pi^0_1)$. However, in the case of ideals with the hereditary Baire property, $\I\in P^{\,|}$ can be replaced with $\I \notin P(\Sigma^0_2)$, which is a more similar property to $\I\in P(\Pi^0_1)$ and allows to extend this approach to the case $\mathscr{L}(\I)=\Sigma^0_2$:

\begin{corollary}\label{cor:inclusionfsigmasets}\label{cor:inclusionclosedsets}
Let $X$ be an uncountable Polish space and let $\I$ be an ideal on $\omega$. 
Then the following hold:
   \begin{enumerate}[label={\rm (\roman{*})}] 
   \item \label{item:1corLIinclusion} Assume that $\I$ has the hereditary Baire property. Then $\mathscr{L}(\I)\subseteq \Pi^0_1$ if and only if $\I\notin P(\Sigma^0_2)$. In particular, $\mathscr{L}(\I)= \Pi^0_1$ if and only if $\I \in P(\Pi^0_1)\setminus P(\Sigma^0_2)$;
   
   \item \label{item:1corLIinclusion22} Assume that $\I$ is coanalytic. Then $\mathscr{L}(\I)\subseteq \Sigma^0_2$ if and only if $\I\notin P(\Pi^0_3)$. In particular, $\mathscr{L}(\I)= \Sigma^0_2$ if and only if $\I \in P(\Sigma^0_2)\setminus P(\Pi^0_3)$.
   \end{enumerate}
\end{corollary}
\begin{proof}
\ref{item:1corLIinclusion} The first part is just a rewriting of the equivalence \ref{thm:Fsigma-Lambda:ideal} $\Longleftrightarrow$ \ref{thm:Fsigma-Lambda:for-hBP-ideals} in Theorem \ref{thm:Fsigma-Lambda}. The second part follows by Theorem \ref{thm:closed-is-Lambda}.

\ref{item:1corLIinclusion22} 
\textsc{If part}. Suppose that $\mathscr{L}(\I)\not\subseteq \Sigma^0_2$ and pick a set $P \in \mathscr{L}(\I)\setminus \Sigma^0_2$. Thanks to Theorem \ref{thm:MSP24}\ref{item:3MSP}, we have $\mathscr{L}(\I)\subseteq \Sigma^1_1$, hence $P$ is analytic. The claim follows from the equivalence \ref{thm:F_sigmadelta-are-Lambda:P} $\Longleftrightarrow$ \ref{thm:F_sigmadelta-are-Lambda:analytic-set} in Theorem \ref{thm:F_sigmadelta-are-Lambda}

\textsc{Only If part}. Suppose that $\I \in P(\Pi^0_3)$. It follows by Theorem \ref{thm:F_sigmadelta-are-Lambda} that $\Pi^0_3\subseteq \mathscr{L}(\I)$. Since $\Pi^0_3\setminus \Sigma^0_3$ is nonempty, see e.g. \cite[Theorem 22.4]{MR1321597}, we get $\mathscr{L}(\I)\not\subseteq \Sigma^0_3$. In particular, 
$\mathscr{L}(\I)\not\subseteq \Sigma^0_2$. 

The second part follows by Theorem \ref{thm:Fsigma-Lambda}.
\end{proof}

\begin{corollary}\label{cor:coanalytic}
Let $X$ be an uncountable Polish space and let $\I$ be a coanalytic ideal on $\omega$. 
Then exactly one of the following cases occurs:
   \begin{enumerate}[label={\rm (\roman{*})}] 
   \item \label{item:1coanalytic} $\I \in P(\Pi^0_1)\setminus P(\Sigma^0_2)$ and $\mathscr{L}(\I)= \Pi^0_1$;
   
   \item \label{item:2coanalytic} $\I \in P(\Sigma^0_2)\setminus P(\Pi^0_3)$ and $\mathscr{L}(\I)= \Sigma^0_2$;

   \item \label{item:3coanalytic} $\I \in P(\Pi^0_3)$ and $\Pi^0_3\subseteq\mathscr{L}(\I)\subseteq\Sigma^1_1$.
   \end{enumerate}
\end{corollary}

\begin{proof}
Let $\I$ be a coanalytic ideal. 
Then $\I$ has the hereditary Baire property, so $\I\in P(\Pi^0_1)$ by Proposition \ref{prop:P-for-closed2}\ref{item:1Ppi01}. Recall also that 
$
P(\Pi^0_3)\subseteq P(\Sigma^0_2)\subseteq P(\Pi^0_1),
$ 
thanks to Proposition \ref{prop:properties-of-P-like}\ref{prop:properties-of-P-like:implications}. 
Hence:
\begin{enumerate}[label={\rm (\roman{*})}] 
\item If $\I \in P(\Pi^0_1)\setminus P(\Sigma^0_2)$ then $\mathscr{L}(\I)=\Pi^0_1$, see Corollary \ref{cor:inclusionclosedsets}\ref{item:1corLIinclusion};

\item If $\I \in P(\Sigma^0_2)\setminus P(\Pi^0_3)$ then $\mathscr{L}(\I)=\Sigma^0_2$, see Corollary \ref{cor:inclusionfsigmasets}\ref{item:1corLIinclusion22};

\item If $\I \in P(\Pi^0_3)$ then  $\Pi^0_3\subseteq\mathscr{L}(\I)\subseteq\Sigma^1_1$, see Theorem \ref{thm:F_sigmadelta-are-Lambda} and Theorem \ref{thm:MSP24}\ref{item:3MSP}.
\end{enumerate}
\end{proof}

In the case of $P^-$-ideals we can do even better: there are only three possibilities for the family $\mathscr{L}(\I)$. 
\begin{theorem}
\label{thm:coanalytic-Pminus}
Let $X$ be an uncountable Polish space and let $\I$ be a coanalytic $P^-$-ideal on $\omega$. Then exactly one of the following cases occurs:
$$
\mathscr{L}(\I)=\Pi^0_1 
\quad \text{ or }\quad 
\mathscr{L}(\I)=\Sigma^0_2 
\quad \text{ or }\quad 
\mathscr{L}(\I)=\Sigma^1_1. 
$$

In particular, this applies to all $\Pi^0_4$ ideals. 
\end{theorem}

\begin{proof}
Let $\I$ be a coanalytic $P^-$-ideal. Then $\I\in P(\Pi^0_3)$ is equivalent to $\I\in P^?(\Sigma^1_1)$ by Proposition \ref{prop:P^-}\ref{item:4pminus}, so Corollary \ref{cor:coanalytic} and Theorem \ref{thm:analytic} finish the proof.

The last part follows by the fact that every $\Pi^0_4$ ideal is $P^-$ (see Proposition~\ref{prop:Pplus-vs-Pminus+Pvert}\ref{thm:Plike-properties-for-definable-ideals:Pi-zero-four-is-Pminus}) and, of course, analytic.
\end{proof}

As an application, we prove that the case $\mathscr{L}(\I)=\Sigma^1_1$ is indeed possible: 
\begin{theorem}\label{thm:IWanalytic}
    There exists a coanalytic $P^-$-ideal $\I$ such that 
$$
\mathscr{L}(\I)=\Sigma^1_1
$$
in every uncountable Polish space $X$.
\end{theorem}
\begin{proof}
Let $  x =(x_n: n \in \omega)$ be an enumeration of $\mathbb{Q}(2^\omega)$. 
    For each nonempty subset $W\subseteq 2^\omega$, 
    define the ideal
    $$
    \I_W=\left\{S\subseteq \omega: \Lambda_{  x \restriction A}(\fin)\cap W=\emptyset \text{ for all infinite }A\subseteq S\right\}.
    $$
    Thanks to \cite[Theorem 2.2(i)]{MR4393937}, we have $W \in \mathscr{L}(\I_W)$, cf. also \cite[Theorem 4.1]{MSP24}. 
    
    In addition, let us show that every ideal of the form $\I_W$ is $P^-$. Indeed, suppose that $(A_n: n \in \omega)$ is a decreasing sequence of $\I_W$-positive sets such that $A_n\setminus A_{n+1}\in \I_W$ for all $n \in \omega$. Since $A_0 \notin \I_W$, we can fix a subset $A \subseteq A_0$ such that the subsequence $(x_n: n \in A)$ is convergent to some limit $\eta \in W$. 
    Then $A\in\I_W^+$ and for each $n \in \omega$ we have 
    $$
    B_n=A\setminus A_n\subseteq A_0\setminus A_n=\bigcup_{i<n}A_i\setminus A_{i+1} \in \I_W.
    $$
    Suppose that $B_n$ is infinite for some $n \in \omega$. Then the subsequence $(x_n: n \in B_n)$ would be convergent to $\eta \in W$, so that $A_0\setminus A_n \notin \I_W$. It follows that $B_n$ is finite for every $n$, which proves the claim. 

    At this point, set $\I=\I_{W_{\mathrm{irr}}}$, where $W_{\mathrm{irr}}=2^\omega\setminus \mathbb{Q}(2^\omega)$. By the previous observations, $\I$ is a $P^-$-ideal such that $\mathscr{L}_{2^\omega}(\I)$ contains the set $W_{\mathrm{irr}}$, which is analytic and not $F_\sigma$. By Theorem \ref{thm:F_sigmadelta-are-Lambda}, $\I$ is $P(\Pi^0_3)$. Moreover, it has been shown in \cite[Section 4]{MSP24} that the ideal $\I$ is coanalytic (and not analytic). Thanks to Corollary \ref{cor:coanalytic} and Theorem \ref{thm:coanalytic-Pminus}, we conclude that $\mathscr{L}_X(\I)=\Sigma^1_1$ for every uncountable Polish space $X$. 
\end{proof}


\subsection{Main open question}

As a rather informal heuristic from our results, one might notice that the topological complexity of the family $\mathscr{L}(\I)$ is smaller than the complexity of the ideal $\I$ itself. 
For instance, it has been already asked in \cite[Question 2.11]{MR4393937} whether $\mathscr{L}(\I)$ is a Borel class for every Borel ideal $\I$; cf. also \cite[Question 3.11]{MR4393937}. Thus, one may expect that there are no $\Pi^0_4$ ideals $\I$ for which $\mathscr{L}(\I)=\Sigma^1_1$, which would imply that either $\mathscr{L}(\I)=\Pi^0_1$ or $\mathscr{L}(\I)=\Sigma^0_2$ for $\Pi^0_4$ ideals $\I$, see Theorem \ref{thm:coanalytic-Pminus}.

Hence, we formulate the following conjecture: 
\begin{question}\label{main-conjecture}
Let $X$ be an uncountable Polish space and let $\I$ be a $\Pi^0_4$ ideal on $\omega$. Is it true that either $\mathscr{L}(\I)=\Pi^0_1$ or $\mathscr{L}(\I)=\Sigma^0_2$? 
\end{question}

To substantiate the above hypothesis, recall that an ideal $\I$ on $\omega$ is a \emph{Farah ideal} (\cite[p.~199]{MR2849045}) if there is a family of hereditary compact sets $\{K_n : n \in \omega\}$ such that 
$$\I = \{S \subseteq \omega : \forall n \in \omega \, \exists k \in \omega \, (S \setminus k \in  K_n)\}.$$ 
It is known (\cite[Theorem 3.8]{MR2849045}) that every Farah ideal is $\Pi^0_3$. Farah conjectured (\cite[p.~199]{MR2849045}) that every $\Pi^0_3$ ideal is a Farah ideal.

\begin{remark}
\label{rem:Borel-comlexity-of-ideal-implies-Lambda-complexity}
Let $X$ be an uncountable Polish space and let $\I$ be an ideal on $\omega$. Then the following hold:
\begin{enumerate}[label={\rm (\roman{*})}] 
\item \label{item:2corfinal} If $\I$ is $\Sigma^0_3$ then $\mathscr{L}(\I)=\Pi^0_1$, see Theorem \ref{thm:xi}\ref{item:4xi} and Theorem \ref{thm:MSP24}\ref{item:1MSP};
\item \label{item:1corfinal} If $\I$ is Farah then either $\mathscr{L}(\I)=\Pi^0_1$ or $\mathscr{L}(\I)=\Sigma^0_2$, see Theorem \ref{thm:xi}\ref{item:3xi} and Theorem \ref{thm:MSP24}\ref{item:2MSP}.
\end{enumerate}
\end{remark}

Thanks to Remark \ref{rem:Borel-comlexity-of-ideal-implies-Lambda-complexity}\ref{item:1corfinal}, a positive answer to Farah's conjecture would imply a positive solution of Question \ref{main-conjecture} in the case of $\Pi^0_3$ ideals. 
In addition, a positive answer to Question \ref{main-conjecture} would solve in the affirmative \cite[Question 4.9]{MR4393937}, which asks whether $\mathscr{L}(\I)\subseteq \Sigma^0_2$ for every $\Pi^0_3$ ideal.

Another way to make precise the above heuristic is the next open question: 
\begin{question}\label{conjecture:I-not-in-Lambda_I}
    Consider $X=2^\omega$ and let $\I$ be an ideal on $\omega$. Is it true that 
    $$
    \I\notin \mathscr{L}_{2^\omega}(\I)\,?
    $$
\end{question}
Even weaker, we do not know whether there exists an ideal $\I$ such that $\mathcal{L}(\I)=\mathcal{P}(X)$. 

Note that a positive solution of Question \ref{conjecture:I-not-in-Lambda_I} implies a positive solution of Question \ref{main-conjecture} (by Remark \ref{rem:alltopspaces}).


\subsection{Higher Borel classes}

As we are going to show below, the family of ideals $\I$ for which $\mathscr{L}(\I)$ attains a given Borel class may have arbitrarily high Borel complexity. In particular, ideals $\I$ for which $\mathscr{L}(\I)=\Pi^0_1$ may attain arbitrarily high Borel complexity.

\begin{theorem}\label{thm:highdimensionalhighBorel}
    Let $X$ be a Polish space and let $\alpha$ be a positive integer and $\beta$ be an ordinal such that 
    $\beta\ge \alpha+2$. 
    Then the following hold: 
\begin{enumerate}[label={\rm (\roman{*})}]
\item \label{item:1mainlimitfamilieshighdimension} 
If $\alpha$ is odd, there exist ideals $\I \in \Sigma^0_{\beta}\setminus \Pi^0_{\beta}$ and $\J \in \Pi^0_{\beta}\setminus \Sigma^0_{\beta}$ such that 
$$\mathscr{L}(\I)=\mathscr{L}(\J)=\Pi^0_{\alpha};$$

\item \label{item:2mainlimitfamilieshighdimension} 
If $\alpha$ is even, there exist ideals $\I \in \Sigma^0_{\beta}\setminus \Pi^0_{\beta}$ and $\J \in \Pi^0_{\beta}\setminus \Sigma^0_{\beta}$ such that 
$$\mathscr{L}(\I)=\mathscr{L}(\J)=\Sigma^0_{\alpha}.$$
\end{enumerate}
\end{theorem}

\begin{proof}
   \ref{item:1mainlimitfamilieshighdimension} 
   First, let us prove the claim for $\alpha=1$. For, suppose that $\beta \ge 3$ and pick sets $S_\beta,P_\beta\subseteq 2^\omega$ such that $S_\beta\in \Sigma^0_{\beta}\setminus \Pi^0_{\beta}$ and $P_\beta \in \Pi^0_\beta\setminus \Sigma^0_\beta$, see e.g. \cite[Theorem 22.4]{MR1321597}. Let $\I_{S_\beta}$ be the smallest ideal on $2^{<\omega}$ containing $\mathcal{A}_{S_\beta}=\{\{x\restriction n: n \in\omega\}: x \in S_\beta\}$, and define $\I_{P_\beta}$ and $\mathcal{A}_{P_\beta}$ analogously. Hence 
   $$
   \I_{S_\beta} \in \Sigma^0_{\beta}\setminus \Pi^0_{\beta}\quad \text{ and }\quad 
   \I_{P_\beta} \in \Pi^0_\beta\setminus \Sigma^0_\beta,
   $$
   see \cite[Exercise 23.4 with its solution at p. 362]{MR1321597}. 
   Since $\mathcal{A}_{S_\beta}$ and $\mathcal{A}_{P_\beta}$ are almost disjoint families (that is, a family of infinite subsets of $2^{<\omega}$ with finite pairwise intersections), it follows by \cite[Example~2 and Lemma~1 in Section 9]{MR1442262} that both ideals $\mathcal{I}_{S_\beta}$ and $\mathcal{I}_{P_\beta}$ are $P^+$. In addition, since they are Borel, they have the hereditary Baire property. It follows by Theorem \ref{thm:Fsigma-Lambda} that 
   \begin{equation}\label{eq:firstclaimlastthm2}
   \mathscr{L}(\I_{S_\beta})=\mathscr{L}(\I_{P_\beta})=\Pi^0_1.
   \end{equation}

   Hence, let us suppose hereafter that $\alpha \ge 3$ is odd and $\beta \ge \alpha+2$. 
   Note that $\gamma=\frac{1}{2}(\alpha+1)$ is an integer such that $2\le \gamma$ and $2\gamma=\alpha+1< \beta$. 
   Set $\I=\fin^{\gamma}\oplus \I_{S_\beta}$ and $\J=\fin^\gamma \oplus \I_{P_\beta}$. Regarding $\fin^\gamma$ as a $\Sigma^0_{\alpha+1}$ (see Theorem \ref{thm:highBorelMSP}) ideal on $\omega$, it easily follows that $\I$ is a $\Sigma^0_\beta$ ideal. 
   In addition, notice that the map $\pi: \mathcal{P}(\{0,1\}\times \omega) \to \mathcal{P}(\omega)$ defined by $\pi(A)=\{k \in \omega: (1,k) \in A\}$ is open. Considering that $\I_{S_\beta} \notin \Pi^0_\beta$ and that $\pi[\I]=\I_{S_\beta}$, we obtain that $\I\notin \Pi^0_\beta$. With an analogous reasoning, we get $\J\in\Pi^0_\beta\setminus \Sigma^0_\beta$. Lastly, it follows by Theorem \ref{thm:highBorelMSP}, and Proposition \ref{prop:basic-prop-about-Lambda-Gamma}\ref{prop:basic-prop-about-Lambda-Gamma:direct-sum}, and Equation \eqref{eq:firstclaimlastthm2} that
   $$
   \mathscr{L}(\I)=\mathscr{L}(\J)=\{A\cup B: A \in \Pi^0_\alpha, B \in \Pi^0_1\}=\Pi^0_\alpha. 
   $$

   \ref{item:2mainlimitfamilieshighdimension} Fix an even integer $\alpha\ge 2$, so that $\delta=\alpha/2$ is an integer and $3\le 2\delta+1=\alpha+1< \beta$. 
   The proof proceeds verbatim as in item \ref{item:1mainlimitfamilieshighdimension}, by setting $\I=(\emptyset\otimes \fin^\delta)\oplus \I_{S_\beta}$ and $\J=(\emptyset\otimes \fin^\delta)\oplus \I_{P_\beta}$. 
\end{proof}

\begin{remark}\label{rmk:betaequalsalphaplisone}
In the case $\beta=\alpha+1$, with the same proof as above, we can show that:
\begin{enumerate}[label=(\alph*)]
\item 
If $\alpha$ is odd, there exists an ideal $\I \in \Sigma^0_{\alpha+1}\setminus \Pi^0_{\alpha+1}$ such that 
$\mathscr{L}(\I)=\Pi^0_{\alpha}$;

\item 
If $\alpha$ is even, there exists an ideal $\J \in \Pi^0_{\alpha+1}\setminus \Sigma^0_{\alpha+1}$ such that 
$\mathscr{L}(\J)=\Sigma^0_{\alpha}$.
\end{enumerate}
Alternatively, both items are also a consequence of Theorem \ref{thm:highBorelMSP} (choosing $\I=\fin^\gamma$ and $\J=\emptyset\otimes \fin^\delta$, with the above notations): indeed, it is known that $\fin^\alpha\notin \Pi^0_{2\alpha}$ and $\emptyset\otimes \fin^\alpha \notin \Sigma^0_{2\alpha+1}$ for all $\alpha \ge 1$, see \cite[Corollary 1.2]{MR2902755} and \cite[Exercise 23.3]{MR1321597}. 
\end{remark}

To complete the picture of Remark \ref{rmk:betaequalsalphaplisone} in the case $\beta=\alpha+1$ as in the statement of Theorem \ref{thm:highdimensionalhighBorel}, we consider also the following: 
\begin{enumerate}[label=(\alph*$^\prime$)]
\item \label{item:aprimelastremark}
If $\alpha$ is odd, there exists an ideal $\I \in \Pi^0_{\alpha+1}\setminus \Sigma^0_{\alpha+1}$ such that 
$\mathscr{L}(\I)=\Pi^0_{\alpha}$;

\item \label{item:bprimelastremark}
If $\alpha$ is even, there exists an ideal $\J \in \Sigma^0_{\alpha+1}\setminus \Pi^0_{\alpha+1}$ such that 
$\mathscr{L}(\J)=\Sigma^0_{\alpha}$.
\end{enumerate}
However, both items \ref{item:aprimelastremark} and \ref{item:bprimelastremark} \emph{fail} in general. Indeed, if $\alpha=1$, it is folklore that $\Pi^0_2$ ideals $\I$ do not exist, see e.g. \cite[Proposition 1.2.2]{alcantara-phd-thesis}. In addition, if $\alpha=2$, then every $\Sigma^0_3$ ideal $\J$ satisfies $\mathscr{L}(\J)=\Pi^0_1$, thanks to Remark \ref{rem:Borel-comlexity-of-ideal-implies-Lambda-complexity}\ref{item:2corfinal}. However, we do not know what happens in the case $\alpha \ge 3$. Hence, we state it as an open question:
\begin{question}\label{question:completenessreachedfamilies}
 Let $X$ be an uncountable Polish space and let $\I$ be an ideal on $\omega$. 
 Is it true that, if $\alpha$ is 
 odd 
 and $\I \in \Pi^0_{\alpha+1}\setminus \Sigma^0_{\alpha+1}$, then $\mathscr{L}(\I)\neq \Pi^0_{\alpha}$? 
 Is it true that, if $\alpha$ is 
 even 
 and $\J \in \Sigma^0_{\alpha+1}\setminus \Pi^0_{\alpha+1}$, then $\mathscr{L}(\J)\neq \Sigma^0_{\alpha}$?
\end{question}

A positive answer to Question \ref{main-conjecture} would lead to a positive answer to Question \ref{question:completenessreachedfamilies} in the case $\alpha=3$.

\begin{remark}
    We remark that the ideal $\I_{P_3}$, as defined in the proof of Theorem \ref{thm:highdimensionalhighBorel}, is Farah and it is not an analytic $P$-ideal. Indeed, it is not a $P$-ideal as witnessed by the sequence of sets $(\{x_k\restriction n: n \in \omega\}: k \in \omega)$, where $(x_k: k \in \omega)$ is any injective sequence with values in $P_3$. 
    Next, let us show that $\I_{P_3}$ is Farah. For, let $(C_n: n \in \omega)$ be a sequence of $\Sigma^0_2$-sets such that $P_3=\bigcap_n C_n$. Moreover, for each $B\subseteq 2^{<\omega}$, we write 
    $$
    \widehat{B}=\{x\in 2^\omega: |B \cap \{x\restriction n: n \in \omega\}|=\infty\}.
    $$
    Note that $\mathcal{I}_{2^\omega}$ is equal to the set of all $B\subseteq 2^{<\omega}$ which can be covered by finitely many branches. In particular, $B \in \I_{2^\omega}$ if and only if $\widehat{B}$ is finite. Since $P_3=\bigcap_n C_n$, it follows that
    \begin{displaymath}
        \begin{split}
            \mathcal{I}_{P_3}&=\I_{2^\omega} \cap \{B\subseteq 2^{<\omega}: \widehat{B}\subseteq P_3\}\\
            &=\bigcap_{n\in \omega} \left(\I_{2^\omega} \cap \{B\subseteq 2^{<\omega}: \widehat{B}\subseteq C_n\}\right)=\bigcap_{n\in \omega} \, \I_{C_n}. 
        \end{split}
    \end{displaymath}
    Observe that each $\I_{C_n}$ is an ideal (hence, hereditary and closed under finite changes) which is $\Sigma^0_2$ because $C_n\in \Sigma^0_2$, see again \cite[Exercise 23.4 with its solution at p. 362]{MR1321597}. 
    We conclude by \cite[Proposition 3.6]{MR2849045} 
    that $\I_{P_3}$ is Farah. 
    \end{remark}

We conclude with Figure \ref{fig:Borel-ideals-vs-Lambda}, which shows the relationships between classes of Borel ideals and their families of $\I$-limit points. Here, \textsc{Exh} stands for the family of analytic $P$-ideals, while $\I_{S_\beta}$ and $\I_{P_\beta}$ for the ideals defined in the proof of Theorem \ref{thm:highdimensionalhighBorel}.

\begin{figure}
\centering
\begin{tikzpicture}[even odd rule]

\shade[top color=yellow!5, bottom color=yellow!40]
(-1.75,-3) rectangle (4.25,3.75) 
(-5,-4) rectangle (4.5,4.5);

\shade[top color=blue!50, bottom color=blue!4]
(-5,-4) rectangle (-4,4.5)
(-5.5,-5.5) rectangle (-4,5);

\draw (-5.75,-6) rectangle (5.75,5.5);

\draw[very thick,dashed] (-4,-6) -- (-4,-4);
\draw[dashed, ultra thin] (-4,-4) -- (-4,4.5);
\draw[very thick,dashed] (-4,4.5) -- (-4,5.5); 
\draw[dotted](-4,5.5) -- (-4,6);

\draw[very thick,dashed] (-2,-6) -- (-2,-4);
\draw[dashed, ultra thin] (-2,-4) -- (-2,4.5);
\draw[very thick,dashed] (-2,4.5) -- (-2,5.5); 
\draw[dotted](-2,5.5) -- (-2,6);

\draw[very thick,dashed] (1,-6) -- (1,-4);
\draw[dashed, ultra thin] (1,-4) -- (1,-3);
\draw[very thick,dashed](1,-3) -- (1,3.75);
\draw[dashed, ultra thin] (1,3.75) -- (1,4.5);
\draw[very thick,dashed] (1,4.5) -- (1,5.5); 
\draw[dotted](1,5.5) -- (1,6);

\draw[dotted] (-5.75,-6) -- (-6.1,-6);
\draw[dotted] (-5.75,5.5) -- (-6.1,5.5);
\draw[dashed,-latex] (-6.1,-6) -- (-6.1,5.5);
\draw[dashed,-latex] (-6.1,5.5)--(-6.1,-6);
\node[rotate=90] at (-6.4,0) {{\small \textsc{Borel ideals}}};

\node at (-5,5.75) {\ldots};
\node at (-3,5.75) {{\footnotesize $\mathscr{L}(\I) = \Pi^0_3$}};
\node at (-0.5,5.75) {{\footnotesize $\mathscr{L}(\I) = \Sigma^0_2$}};
\node at (3.5,5.75) {{\footnotesize $\mathscr{L}(\I) = \Pi^0_1$}};

\draw[dotted] (5.75,5.5) -- (5.75,6.1);
\draw[dotted] (-5.75,5.5) -- (-5.75,6.1);
\draw[dashed,-latex] (-5.75,6.1) -- (5.75,6.1);
\draw[dashed,-latex] (5.75,6.1) -- (-5.75,6.1);
\node at (-1,6.4) {{\small $\mathscr{L}(\I) \subseteq  \Sigma^1_1$}};

\draw (-5.5,-5.5) rectangle (5.5,5);
\node at (-5.25+.1,-5.25+.08) {\large $\Sigma^0_4$};
\node at (-3,-4.4-.3) {{\footnotesize $\fin^2$}};
\draw[fill] (-3.5,-4.4-.3) circle (0.06cm);
\node at (-.5,-5.2+.4) {{\footnotesize $(\emptyset\otimes \fin) \oplus \I_{S_4}$}};
\draw[fill] (-.4,-4.9+.4) circle (0.06cm);
\node at (3.75,-5.25) {\footnotesize $\I_{S_{4}}$};
\node at (3.75,-5.25) {\footnotesize $\I_{S_{4}}$};
\draw[fill] (3.3,-5.25) circle (0.06cm);
\node at (-4.75, -4.6){\tiny \emph{Ques.\ref{conjecture:I-not-in-Lambda_I}}};

\draw (-5,-4) rectangle (4.5,4.5);
\node at (-4.75+.1,-3.75+.08) {\large $\Pi^0_3$};
\node at (-3.3,0){\footnotesize \emph{Farah conjecture}};

\node at (2.75,3.4) {\footnotesize $\I_{P_{3}}$};
\draw[fill] (2.3,3.4) circle (0.06cm);

\draw (1.5,-4.75) rectangle (5,3);
\node at (1.75+.1,-4.5+.08) {\large $\Sigma^0_3$};
\node at (3.75,-4.5) {\footnotesize $\I_{S_{3}}$};
\draw[fill] (3.3,-4.5) circle (0.06cm);

\draw (2,-2.75) rectangle (3.5,1.75);
\node at (2.25+.1,-2.5+.08) {\large $\Sigma^0_2$};
\node at (3,1) {\footnotesize $\fin$};
\draw[fill] (2.5,1) circle (0.06cm);
\node at (3,0) {\footnotesize $\I_{1/n}$};
\draw[fill] (2.38,0) circle (0.06cm);

\draw(-0.5,-1) rectangle (4,2.5);
\node at (-.1+.04,-0.75) {{\small \textsc{Exh}}};
\node at (0.3,0.5) {\footnotesize $\mathcal{Z}$};
\draw[fill] (0,0.5) circle (0.06cm);
\node at (0.3,2) {\footnotesize $\emptyset \otimes \fin$};
\draw[fill] (0.1,1.7) circle (0.06cm);

\node at (0.3-1.2,1.2) {\footnotesize $\mathcal{B}$};
\draw[fill] (0.1-1.3,1.2) circle (0.06cm);

\draw (-1.75,-3) rectangle (4.25,3.75);
\node at (-1.15+.04,-2.75) {{\small \textsc{Farah}}};
\node at (-0.2,-2) {{\footnotesize $\mathrm{nwd}$}};
\draw[fill] (-.7,-2) circle (0.06cm);

\end{tikzpicture}

\caption{Relationships between the topological complexity of Borel ideals $\I$ and the families $\mathscr{L}(\I)$.}
    \label{fig:Borel-ideals-vs-Lambda}
\end{figure}
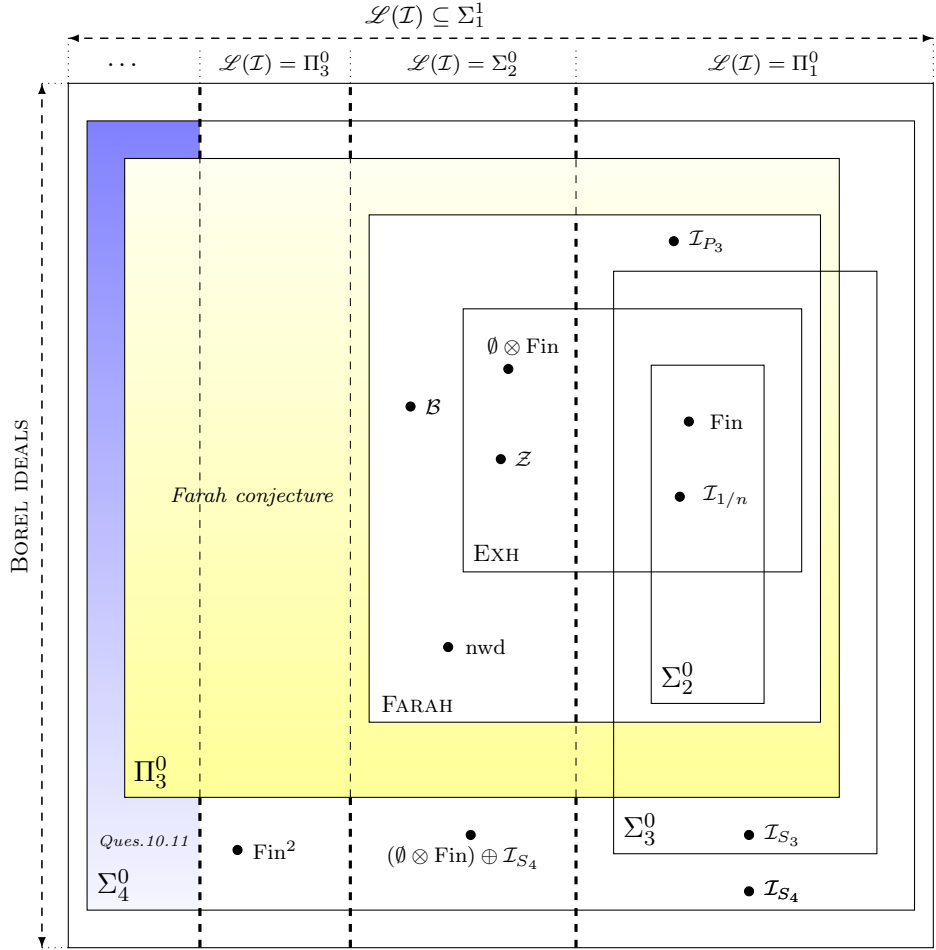


\bibliographystyle{amsplain}
\bibliography{ideals}

\end{document}